\documentclass[openany]{amsart}
\usepackage{amsmath,amssymb,amsfonts,dsfont}
\usepackage[all,arc,2cell,cmtip]{xy}
\usepackage{enumerate}
\usepackage[usenames,dvipsnames]{color}
\usepackage[normalem]{ulem}

\usepackage{mathrsfs}
\usepackage{color}
\usepackage{graphicx}
\UseAllTwocells  
\usepackage{tikz}

\include{scrload}

\usepackage{microtype}

\newlength{\storeparskip}
\usepackage{hyperref}  
\hypersetup{%
  bookmarksnumbered=true,%
  bookmarks=true,%
  colorlinks=true,%
  linkcolor=blue,%
  citecolor=blue,%
  filecolor=blue,%
  menucolor=blue,%
  pagecolor=blue,%
  urlcolor=blue,%
  pdfnewwindow=true,%
  pdfstartview=FitBH}
  
\let\fullref\autoref
\def\makeautorefname#1#2{\expandafter\def\csname#1autorefname\endcsname{#2}}
\def\equationautorefname~#1\null{(#1)\null}
\makeautorefname{footnote}{footnote}%
\makeautorefname{item}{item}%
\makeautorefname{figure}{Figure}%
\makeautorefname{table}{Table}%
\makeautorefname{part}{Part}%
\makeautorefname{appendix}{Appendix}%
\makeautorefname{chapter}{Chapter}%
\makeautorefname{section}{Section}%
\makeautorefname{subsection}{Section}%
\makeautorefname{subsubsection}{Section}%
\makeautorefname{paragraph}{Paragraph}%
\makeautorefname{subparagraph}{Paragraph}%
\makeautorefname{theorem}{Theorem}%
\makeautorefname{thm}{Theorem}%
\makeautorefname{addm}{Addendum}%
\makeautorefname{mainthm}{Main theorem}%
\makeautorefname{introthm}{Theorem}%
\makeautorefname{introrem}{Remark}%
\makeautorefname{corollary}{Corollary}%
\makeautorefname{cor}{Corollary}%
\makeautorefname{lemma}{Lemma}%
\makeautorefname{lem}{Lemma}%
\makeautorefname{sublemma}{Sublemma}%
\makeautorefname{sublem}{Sublemma}%
\makeautorefname{subl}{Sublemma}%
\makeautorefname{prop}{Proposition}%
\makeautorefname{property}{Property}
\makeautorefname{pro}{Property}
\makeautorefname{sch}{Scholium}%
\makeautorefname{step}{Step}%
\makeautorefname{conject}{Conjecture}%
\makeautorefname{conj}{Conjecture}%
\makeautorefname{questn}{Question}
\makeautorefname{quest}{Question}
\makeautorefname{qn}{Question}
\makeautorefname{definition}{Definition}%
\makeautorefname{defn}{Definition}%
\makeautorefname{defi}{Definition}%
\makeautorefname{def}{Definition}%
\makeautorefname{dfn}{Definition}%
\makeautorefname{notation}{Notation}
\makeautorefname{notn}{Notation}
\makeautorefname{rem}{Remark}%
\makeautorefname{rems}{Remarks}%
\makeautorefname{rmk}{Remark}%
\makeautorefname{rk}{Remark}%
\makeautorefname{remarks}{Remarks}%
\makeautorefname{rems}{Remarks}%
\makeautorefname{rmks}{Remarks}%
\makeautorefname{rks}{Remarks}%
\makeautorefname{example}{Example}%
\makeautorefname{examp}{Example}%
\makeautorefname{exmp}{Example}%
\makeautorefname{exam}{Example}%
\makeautorefname{exa}{Example}%
\makeautorefname{axiom}{Axiom}%
\makeautorefname{axi}{Axiom}%
\makeautorefname{ax}{Axiom}%
\makeautorefname{case}{Case}%
\makeautorefname{claim}{Claim}%
\makeautorefname{clm}{Claim}%
\makeautorefname{assumpt}{Assumption}%
\makeautorefname{asses}{Assumptions}%
\makeautorefname{conclusion}{Conclusion}%
\makeautorefname{concl}{Conclusion}%
\makeautorefname{conc}{Conclusion}%
\makeautorefname{cond}{Condition}%
\makeautorefname{const}{Construction}%
\makeautorefname{con}{Construction}%
\makeautorefname{criterion}{Criterion}%
\makeautorefname{criter}{Criterion}%
\makeautorefname{crit}{Criterion}%
\makeautorefname{exercise}{Exercise}%
\makeautorefname{exer}{Exercise}%
\makeautorefname{exe}{Exercise}%
\makeautorefname{problem}{Problem}%
\makeautorefname{problm}{Problem}%
\makeautorefname{prob}{Problem}%
\makeautorefname{prob}{Problem}%
\makeautorefname{soln}{Solution}%
\makeautorefname{sol}{Solution}%
\makeautorefname{sum}{Summary}%
\makeautorefname{oper}{Operation}%
\makeautorefname{obs}{Observation}%
\makeautorefname{ob}{Observation}%
\makeautorefname{conv}{Convention}%
\makeautorefname{cvn}{Convention}%
\makeautorefname{warn}{Warning}%
\makeautorefname{note}{Note}%
\makeautorefname{fact}{Fact}%
\makeautorefname{ouch0}{Counterexample}%
\makeindex

\newtheorem{thm}{Theorem}[section]
\newtheorem{cor}{Corollary}[section]
\newtheorem{prop}{Proposition}[section]
\newtheorem{lem}{Lemma}[section]

\newtheorem{introthm}{Theorem}
\newtheorem*{introcor}{Corollary}

\theoremstyle{definition}

\newtheorem{defn}{Definition}[section]
\newtheorem{con}{Construction}[section]

\newtheorem{notn}{Notation}[section]

\newtheorem{assumpt}{Assumption}[section]

\newtheorem{rem}{Remark}[section]

\makeatletter
\let\c@cor=\c@thm
\let\c@prop=\c@thm
\let\c@lem=\c@thm
\let\c@conj=\c@thm
\let\c@defn=\c@thm
\let\c@notn=\c@thm
\let\c@exmp=\c@thm
\let\c@rem=\c@thm
\let\c@sch=\c@thm
\let\c@con=\c@thm
\let\c@assumpt=\c@thm
\let\c@equation\c@thm
\numberwithin{equation}{section}
\makeatother

\newcommand{\DAlg}{\sD\text{-}\bf{Alg_{ps}}}

\newcommand{\DGAlg}{\sD_G\text{-}\bf{Alg}_{ps}}

\newcommand{\DGA}{\sD_G\text{-}\bf{Alg}}
\newcommand{\AlgPs}{{\bf Alg_{ps}}}
\newcommand{\PAlg}{\PI\text{-}{\bf Alg}}
\newcommand{\FAlg}{\sF\text{-}{\bf Alg}}

\newcommand{\FGAlg}{\sF_G\text{-}{\bf Alg}}

\newcommand{\OAlg}{\oO\text{-}{\bf Alg_{ps}}}
\newcommand{\OAT}{\oO\text{-}{{\bf Alg}_{\bf ps}^{G\sT}}}
\newcommand{\OADT}{\oO\text{-}{{\bf Alg}_{\bf ps}^{G\sT}\!(\sD_G)}}

\newcommand{\psFAlg}{\sF\text{-}{\bf PsAlg}}
\newcommand{\FTop}{\sF_G\text{-}{\bf Top}}
\newcommand{\psFGAlg}{\sF_G\text{-}{\bf PsAlg}}
\newcommand{\CAlg}{\cC\text{-}{\bf{Alg}}}
\newcommand{\CBAlg}{\cC\text{-}{\bf{Alg_{ps\cB}}}}
\newcommand{\CPsAlg}{\cC\text{-}{\bf{PsAlg}}}
\newcommand{\DPsAlg}{\cD\text{-}{\bf{PsAlg}}}

\newcommand{\Top}{\mathbf{\sU}}
\newcommand{\Topp}{\mathbf{\sU_{\ast}}}
\newcommand{\St}{\mathrm{St}}
\newcommand{\Cat}{\mathbf{Cat}}
\newcommand{\Ass}{\mathbf{Assoc}}

\newcommand{\VCat}{\Cat(\sV)}
\newcommand{\CatV}{\ul{\Cat}(\sV)}

\newcommand{\VCatp}{\Cat(\sV_{\bpt})}
\newcommand{\CatVp}{\ul{\Cat}(\sV_{\bpt})}
\newcommand{\CatGVp}{\ul{\Cat}(G\sV_{\bpt})}

\newcommand{\GUCat}{\Cat(G\sU)}
\newcommand{\GUCatp}{\Cat(G\sU_{\bpt})}

\newcommand{\GVCat}{\Cat(G\sV)}
\newcommand{\GVCatp}{\Cat(G\sV_{\bpt})}
\newcommand{\GFT}{\sF_G\text{-}G\sT}

\newcommand{\GFU}{\sF_G\text{-}G\Topp}
\newcommand{\GDU}{\sD_G^{top}\text{-}G\Topp}
\newcommand{\GDT}{\sD_G^{top}\text{-}G\sT}
\newcommand{\GPT}{\PI_G\text{-}G\sT}
\newcommand{\SGD}{\bS_G^{\sD_G}}

\newcommand{\oursectionG}{\oursection_G}
\newcommand{\oursection}{\zeta}

\newcommand{\ob}{\mathbf{Ob}}
\newcommand{\mor}{\textbf{Mor}}

\newcommand{\GTop}{G\mathbf{\sU}}

\newcommand{\SpG}{\mathbf{Sp}_G}
\newcommand{\Sp}{\mathbf{Sp}}
\newcommand{\bOp}{\bO_+}
\newcommand{\bpt}{\ast} 

\newcommand{\Vptwo}{$\sV_{\bpt}$-$2$-category}

\newcommand{\Vtwofun}{$\sV$-$2$-functor}
\newcommand{\Vptwofun}{$\sV_{\bpt}$-$2$-functor}

\newcommand{\VtwoCats}{$\sV$-$2$-categories}

\newcommand{\bsma}{\, \overline\sma\, }

\newcommand{\ul}{\underline}
\newcommand{\Prod}{\prod\limits}

\newcommand{\bk}{\mathbf{k}}
\newcommand{\bm}{\mathbf{m}}
\newcommand{\bn}{\mathbf{n}}
\newcommand{\bp}{\mathbf{p}}
\newcommand{\bq}{\mathbf{q}}
\newcommand{\br}{\mathbf{r}}
\newcommand{\bs}{\mathbf{s}}
\newcommand{\id}{\mathrm{id}}
\newcommand{\iso}{\cong}     
\newcommand{\sma}{\wedge}    
\newcommand{\com}{\circ}     
\newcommand{\rtarr}{\longrightarrow}
\newcommand{\into}{\hookrightarrow}
\newcommand{\xrtarr}{\xrightarrow}
\newcommand{\squiggly}[2]{\xymatrix@1{#1 \ar@{~>}[r] & #2}}

\newenvironment{pf}{\begin{proof}}{\end{proof}}

\newcommand{\Mult}{\mathbf{Mult}}
\newcommand{\ourdefn}[1]{{\it #1}}

\DeclareMathAlphabet{\eus}{U}{eus}{m}{n}
\let\opsymbfont\mathcal
\let\catsymbfont\eus
\let\algsymbolfont\mathcal

\newcommand{\sD}{{\mathscr{D}}}
\newcommand{\sE}{\mathscr{E}} 
\newcommand{\sF}{\mathscr{F}}
\newcommand{\sI}{\mathscr{I}}

\newcommand{\sP}{\opsymbfont{P}}

\newcommand{\sT}{\mathscr{T}}
\newcommand{\sU}{\mathcal{U}}
\newcommand{\sV}{\mathcal{V}}
\newcommand{\sW}{\mathcal{W}}

\newcommand{\sY}{\mathscr{Y}}

\newcommand{\bE}{\mathbb{E}}
\newcommand{\bF}{\mathbb{F}}
\newcommand{\bJ}{\mathbb{J}}
\newcommand{\bK}{\mathbb{K}}
\newcommand{\bL}{\mathbb{L}}
\newcommand{\bN}{\mathbb{N}}
\newcommand{\bO}{\mathbb{O}}
\newcommand{\bP}{\mathbb{P}}
\newcommand{\bR}{\mathbb{R}}
\newcommand{\bS}{\mathbb{S}}
\newcommand{\bT}{\mathbb{T}}
\newcommand{\bU}{\mathbb{U}}
\newcommand{\bV}{\mathbb{V}}

\newcommand{\cA}{{\catsymbfont{A}}}
\newcommand{\cB}{{\catsymbfont{B}}}
\newcommand{\cC}{{\catsymbfont{C}}}
\newcommand{\cD}{{\catsymbfont{D}}}
\newcommand{\cE}{{\catsymbfont{E}}}
\newcommand{\cM}{{\catsymbfont{M}}}
\newcommand{\cN}{{\catsymbfont{N}}}

\newcommand{\oO}{{\opsymbfont{O}}}

\newcommand{\aA}{{\algsymbolfont{A}}}
\newcommand{\aB}{{\algsymbolfont{B}}}
\newcommand{\aC}{{\algsymbolfont{C}}}

\newcommand{\aW}{{\algsymbolfont{W}}}
\newcommand{\aX}{{\algsymbolfont{X}}}
\newcommand{\aY}{{\algsymbolfont{Y}}}
\newcommand{\aZ}{{\algsymbolfont{Z}}}

\newcommand{\al}{\alpha}
\newcommand{\be}{\beta}
\newcommand{\de}{\delta}
\newcommand{\DE}{\Delta}
\newcommand{\epz}{\varepsilon}
\newcommand{\et}{\eta}
\newcommand{\ga}{\gamma}
\newcommand{\io}{\iota}
\newcommand{\la}{\lambda}
\newcommand{\LA}{\Lambda}
\newcommand{\om}{\omega}
\newcommand{\pa}{\partial}   
\newcommand{\rh}{\rho}
\newcommand{\si}{\sigma}
\newcommand{\ph}{\phi}
\newcommand{\ps}{\psi}
\newcommand{\PI}{\Pi}
\newcommand{\SI}{\Sigma}
\newcommand{\ta}{\tau}
\newcommand{\tha}{\theta}
\newcommand{\THA}{\Theta}
\newcommand{\ze}{\zeta}

\newcommand{\siginv}{\upsilon}

\newcommand{\bzo}{\mathbf{0}}
\newcommand{\opair}{\circledast} 
\newcommand{\spair}{\otimes} 
\newcommand{\esma}{\sma} 
\newcommand{\smaD}{\circledast} 

\newcommand{\enGU}{\ul{G\Top}_*} 

\newcommand{\diag}[2]{\Delta^{\!\scriptscriptstyle #2}\!(#1)}

\newcommand{\oid}{\mathds{1}} 

\newcommand{\mybox}[1]{{\overline{ #1 \! }}}

\makeatletter
\let\c@equation\c@thm
\makeatother
\numberwithin{equation}{section}

\newcommand{\mb}[1]{\mathbf{#1}}

\title[Multiplicative equivariant $K$-theory and the BPQ theorem]{Multiplicative equivariant $K$-theory and the Barratt-Priddy-Quillen theorem}                 
\begingroup%
\author[B.~J. Guillou]{Bertrand J. Guillou}
\address{Department of Mathematics, University of Kentucky, Lexington, KY 40506}
\email{bertguillou@uky.edu}
\author[J.~P. May]{J. Peter May}
\address{Department of Mathematics, The University of Chicago, Chicago, IL 60637}
\email{may@math.uchicago.edu}
\author[M. Merling]{Mona Merling}
\address{Department of Mathematics, University of Pennsylvania, Philadelphia, PA 19104}
\email{mmerling@math.upenn.edu}
\author[A.~M. Osorno]{Ang\'elica M. Osorno}
\address{Department of Mathematics, Reed College, Portland, OR 97202}
\email{aosorno@reed.edu}
\endgroup%

\thanks{B.~J.~Guillou was partially supported  by Simons Collaboration Grant  No. 282316 and NSF grants DMS-1710379 and DMS-2003204. M.~Merling was partially supported by NSF grant DMS-1709461/1850644, a Simons AMS travel grant, and NSF CAREER grant DMS-1943925. A.~M.~Osorno was partially supported by the Simons
Collaboration Grant No. 359449, the Woodrow Wilson Career Enhancement
Fellowship, and NSF grant DMS-1709302.}

\makeatletter
\@namedef{subjclassname@2020}{\textup{2020} Mathematics Subject Classification}
\makeatother

\begin{document}

\begin{abstract} 
We prove a multiplicative version of the equivariant Barratt-Priddy-Quillen theorem,  
starting from the additive version proven in \cite{GMPerm}.  The proof uses 
a multiplicative elaboration of an additive equivariant infinite loop space machine that manufactures orthogonal $G$-spectra from symmetric monoidal $G$-categories.  The new machine produces highly structured associative ring and module $G$-spectra from appropriate multiplicative input.  It relies on new operadic multicategories that are of considerable independent interest and are defined in a general, not necessarily equivariant or topological, context.   Most of our work is focused on constructing and comparing them.  We construct a multifunctor from the multicategory of symmetric monoidal $G$-categories to the multicategory of orthogonal $G$-spectra.  With this machinery in place, we prove that the equivariant BPQ theorem can be lifted to a multiplicative equivalence. That is the heart of what is needed for the presheaf reconstruction of the category of $G$-spectra in \cite{GM}. 
\end{abstract}

\maketitle

\begingroup%
\setlength{\parskip}{\storeparskip}
\tableofcontents
\endgroup%

\section{Introduction}
\label{sec:Intro}

We can view algebraic $K$-theory as a machine that takes as input a category with a structured additive  operation and produces a spectrum by \emph{group-completing} 
 the operation in a homotopy coherent way. The homotopy groups of this spectrum---the higher $K$-groups---are rich invariants which connect homotopy theory with number theory, algebraic geometry, and geometric topology. For example, the homotopy groups of the $K$-theory spectrum of the category of finitely generated projective $R$-modules for a ring $R$ are Quillen's higher $K$-groups of $R$, which are related to important problems and conjectures in number theory, especially when $R$ is a number ring.

Classically, there were two approaches for building the $K$-theory spectrum associated to a symmetric monoidal category: Segal's  approach based on $\Gamma$-spaces \cite{Seg}, and the operadic approach of \cite{BVbook, MayGeo, MayPerm}. These two infinite loop space machines were shown to be equivalent in \cite{MT, MayPerm2}. One fundamental problem in infinite loop space theory is to determine what structure on the input category ensures that its $K$-theory spectrum is a highly structured ring spectrum. If the input has a second, related, structured multiplicative operation, making it into a ``ring category", then a suitably multiplicative $K$-theory  machine should yield a ring spectrum.  The  study of multiplicative infinite loop space theory saw much development early on \cite{MayPair, maymultiplicative, MQR, woolfson1, woolfson2}.  A space level modernized survey is given in \cite{Rant1} and a modernized categorical treatment is given in \cite{Rant2}. A treatment of multiplicative infinite loop space theory that is structured around the use of multicategories and multifunctors is given in \cite{EM1}, and that has served for inspiration in this paper.

For a finite group $G$, the Segal infinite loop space machine has been generalized equivariantly by Shimakawa in \cite{Shim}, and the operadic infinite loop space machine has been generalized equivariantly by two of us in \cite{GMPerm} to build (genuine) orthogonal  $G$-spectra from categories with additive operations that are suitably equivariant.\footnote{A much earlier operadic machine with target Lewis-May $G$-spectra \cite{LMS} was developed by Hauschild, May, and Waner. It was never published, but is outlined by Costenoble and Waner \cite{CW}.}   These equivariant infinite loop space machines have been shown to be equivalent by three of us in \cite{MMO}. 

It is a natural question to ask what kind of structure on a $G$-category makes its $K$-theory into an equivariant ring spectrum, and this is not addressed in any of the papers just mentioned.  Nonequivariantly, the question can be answered without serious use of $2$-category theory, but we have not found such an answer equivariantly.  
The multiplicative structure at the categorical level is encoded via multilinear maps that are distributive up to coherent natural isomorphisms and is thus intrinsically $2$-categorical.  The very different but essentially combinatorial ways around this found nonequivariantly in \cite{EM1, Rant2}  do not appear to generalize equivariantly, or at least not easily.   Our work involves conceptual categorical processing of $2$-categorical input so that it feeds into  an equivariant version of the $1$-categorical Segal machine, whose multiplicative properties we have established in \cite{GMMO}.

An equivariant version of the Barratt-Priddy-Quillen theorem, which expresses the suspension $G$-spectrum of a $G$-space as the equivariant algebraic $K$-theory of a $G$-category, was proven in \cite{GMPerm} using the equivariant operadic machine. However, this equivalence does not a priori preserve the multiplicative structure coming from the smash product of based $G$-spaces. 
The main result of \cite{GM} relies on having a multiplicative equivariant $K$-theory machine starting at the level of $G$-categories that is compatible with the Barratt-Priddy-Quillen theorem, and we provide that in this paper.  An easier multiplicative version of the equivariant Barratt-Priddy-Quillen theorem is proven in  \cite[Theorem~6.7]{GMMO}), but that starts from categorical input that is quite different from the input needed in \cite{GM}.
 
We start with an equivariant $K$-theory machine  $\bK_G$ producing orthogonal  $G$-spectra from  structured $G$-categories, which we take to be algebras over a suitable operad $\oO$. 
In the nonequivariant case, the input would be permutative categories, which are algebras over the Barratt-Eccles operad.
Conceptually, we would like to extend $\bK_G$ to a monoidal functor from structured $G$-categories to orthogonal $G$-spectra. 
However, the ring $G$-categories that arise in nature are not the monoids for a monoidal structure on structured $G$-categories.
Rather, as in \cite{EM1, Lein} and elsewhere, we have a multicategory structure on structured $G$-categories. 
A multicategory structure on a category $\cC$ allows one to make sense of the notion of monoid in $\cC$ as well as module over a monoid. 
We will thus extend $\bK_G$ to a multifunctor, meaning that it is compatible with the multicategory structure.

We give some intuition for finding the structure on an operad that ensures that its category of algebras is a multicategory. We think of the operad $\oO$ as parametrizing addition. Now suppose that we want to define a multiplication that distributes over addition. Just as the product of integers $mn$ is the $n$-fold addition of the integer $m$, we can define a pairing $\oO(m)\times \oO(n)\rtarr \oO(mn)$ by repeating $n$ times the variable in $\oO(m)$  and then ``adding"   using the operad structure map. The diagram that we obtain when we compare this with the map $\oO(n)\times \oO(m)\rtarr \oO(nm)$ that we get by twisting in the source and using a reordering permutation in the target does not strictly commute in general. We define a \emph{pseudo-commutative} operad to be one for which this comparison diagram commutes up to natural isomorphism (see \autoref{pseudocom}), and we show that this condition allows us to define a multicategory structure on the category of $\oO$-algebras.  A key example is the permutativity operad $\sP_G$ of \cite [Definition 3.4]{AddCat1}; its algebras are the permutative $G$-categories and its pseudoalgebras are the symmetric monoidal $G$-categories \cite{AddCat1}. The operad $\sP_G$ is a categorical $E_\infty$ $G$-operad as defined in \cite[Definition 2.1]{GMPerm}, and when $G=e$ it is just the categorical $E_\infty$ Barratt-Eccles operad.

We write $G\Top$ for the category of $G$-spaces and $\Cat(G\Top)$ for the 2-category of categories internal to $G\Top$, as described in \autoref{NotnSectSub}.
Fix a chaotic (\autoref{chaotic})
$E_{\infty}$  $G$-operad $\oO$ in $\Cat(G\Top)$. We construct a multicategory $\Mult(\oO)$ whose underlying category is the category $\OAlg$ of $\oO$-algebras and pseudomorphisms. 
Writing $\SpG$ for the category of orthogonal $G$-spectra, we construct a functor 
\[ \bK_G\colon \OAlg \rtarr \SpG \]
that group completes the additive structure,
and most of the paper is devoted to establishing the following result, which appears as \autoref{KGMultiFun}.

\begin{introthm}\label{IntroMultiKG} 
Let $\oO$ be a chaotic $E_\infty$ $G$-operad in $\Cat(G\Top)$. 
Then the functor \mbox{$\bK_G\colon \OAlg \rtarr \SpG$} extends to a multifunctor.
\end{introthm}

We have the following direct corollary of \autoref{IntroMultiKG}.

\begin{introcor} If $\aA$ is a monoid
in $\OAlg$, then $\bK_G(\aA)$ is a  ring $G$-spectrum. If $\aB$ is an $\aA$-module in $\OAlg$, then $\bK_G(\aB)$ is a $\bK_G(\aA)$-module $G$-spectrum.
\end{introcor}

We warn the reader, however, that \autoref{IntroMultiKG} does not assert that $\bK_G$ is {\it symmetric}. In particular, we do not claim that our version of $\bK_G$ produces commutative ring $G$-spectra as output. Constructing a symmetric  equivariant $\bK$-theory multifunctor is an ongoing challenge. Our multifunctor $\bK_G$ is a composite of multifunctors all but one of which are symmetric, and we shall keep track of symmetry as we go along.   However, associative and unital multiplicative properties are all that are needed for the following result, which is the heart of what is needed in \cite{GM}. 
We prove the following theorem in \autoref{sec:BPQ}. Here we use that $G\sU$ embeds in $\Cat(G\Top)$, as recalled from \cite[Remark 1.8]{AddCat1} in \autoref{embed}.

\begin{introthm}[Multiplicative equivariant Barratt-Priddy-Quillen]
\label{IntroBPQ} 
Let $\oO$ be a topologically discrete chaotic $E_\infty$ $G$-operad in $\Cat(G\Top)$ and $\bOp$ the associated monad. 
There is a lax monoidal natural transformation 
\[ \alpha\colon  \Sigma^\infty_{G+}  {\longrightarrow} \bK_G \bOp\]
of functors $\GTop\rtarr \SpG$ such that $\al_X$ is a stable equivalence of orthogonal $G$-spectra for all input 
$G$-CW complexes $X$.
\end{introthm}

The main result of \cite{GM}  gives a Quillen equivalence between the category of orthogonal $G$-spectra and the category of ``spectral Mackey functors," i.e., spectrally enriched functors $G\mathscr{A}\rtarr \Sp$, where $G\mathscr{A}$ is a spectral version of the Burnside category.  The proof of that result rests on having a multiplicative machine $\bK_G$ which satisfies \autoref{IntroBPQ}.  Its construction was deferred to this paper.  An alternative $\infty$-categorical perspective on spectral Mackey functors as a model for $G$-spectra is given in \cite{BGS, niko, denis}.
Moreover, a version of the multiplicative equivariant Barratt-Priddy-Quillen theorem appears as \cite[Theorem~10.6]{BGS}. As the input for their machine differs from that of ours, a direct comparison of their result with ours would be nontrivial but worthwhile.

\begin{rem}\label{ComparingToInfinityCats} An illuminating $\infty$-category  treatment of multiplicative infinite loop space theory is given in \cite{GGN}. We briefly compare that approach to the theory here.   The input  with that approach is symmetric monoidal $\infty$-categories, which are $\infty$-categorical generalizations of Segal's special $\Gamma$-spaces. In this paper, as classically, the input is symmetric monoidal $1$-categories, but we need to work with the 2-category of such, in order to keep track of the multiplicative structure.
The focus of this paper is the passage from there to special $\Gamma$-categories,  while  the machine $\bS_G\com B$ from  special $\Gamma$-categories  to $\Omega$-$G$-spectra is taken as a black box.  We view the machine $\bS_G\com B$ as essentially formal.  Like the $\infty$-category machine,  it is symmetric monoidal, at least in the variant form given in \cite{GMMO}.  Philosophically, from the $\infty$-category point of view,  we are showing that, even equivariantly, the passage from symmetric monoidal 1-categories to symmetric monoidal $\infty$-categories preserves multiplicative structure, albeit with a loss of symmetry.
\end{rem}

\subsection{A road map}

The organization of this paper focuses on the multiplicative elaboration of the following diagram, which displays $\bK_G$ as the composite of a  sequence of functors.   

\begin{equation}\label{RoadMap}
\xymatrix{
 \OAlg \ar[d]_{\bR} \ar[rr]^{\bK_G}  & &   \SpG\\
 \DAlg \ar[d]_\bP & &  \FTop \ar[u]_{\bS_G}\\
 \DGAlg \ar[r]_-{\oursectionG^*} & \psFGAlg \ar[r]_-{\St} & \FGAlg \ar[u]_{B} \\ }
\end{equation}

The notation $\AlgPs$ denotes $2$-categories of {\em {strict}} algebras and {\em {pseudomorphisms}} between them.  Modulo just a bit of additional notational complexity, we can just as well replace the $2$-categories (-)-$\AlgPs$ in the left hand column by more general $2$-categories (-)-$\bf{PsAlg}$ of pseudoalgebras and pseudomorphisms, with no significant change in constructions or  proofs.  We shall often write $\bR_G$ for the composite $\bP\bR$ in the left column.

 After a few preliminaries setting up our categorical framework of operads and multicategories in \autoref{NotnSect}, multicategories with underlying categories of the form $\OAlg$ are defined in \autoref{sec:Oalg}.  Here $\oO$ is a chaotic operad or, a bit more generally, a pseudo-commutative operad.  We develop a general categorical framework that will specialize to an understanding of categories of operators over both finite sets and finite $G$-sets, together with their algebras and pseudoalgebras, in \autoref{sec:V2}.   We do this in a general framework that will later clarify some key distinctions. Multicategories with underlying categories of the form $\DAlg$ are defined in \autoref{sec:DAlg}, where $\sD$ is a category of operators over the category $\sF$ of finite sets.  Categories of operators were first introduced in \cite{MT}, where they mediated between the operadic and Segalic infinite loop space machines.  They play the same role here.

Taking $\sD$ to be the category of operators associated to a chaotic operad $\oO$, the functor $\bR$ is constructed as a multifunctor in \autoref{Multiop2}, but with a key proof deferred to \autoref{Rpf}.  All of this works in a general categorical context that a priori has nothing to do with either equivariance or topology.   A crucial technical point is that the ``pseudo-commutative pairing'' on an operad that we have already mentioned gives rise to an analogous ``pseudo-commutative pairing'' on its category of operators.   The term ``pseudo-commutative'' was first coined by Hyland and Power \cite{HP} in a monadic avatar of our categories of operators; it can be viewed as shorthand for ``pseudo-symmetric strict monoidal".

Working equivariantly, but in fact in a specialization of our general categorical framework, multicategories with underlying categories of the form $\DGAlg$, where $\sD_G$ is a category of operators over the category $\sF_G$ of finite $G$-sets, are defined in \autoref{sec:DG}. Taking $\sD_G$ to be the category of operators associated to a chaotic operad $\oO$, the multifunctor $\bP$ is also defined in that section.

The categories of operators $\sD$ and $\sD_G$ come with projections $\xi\colon \sD\rtarr \sF$ and $\xi_G\colon \sD_G\rtarr \sF_G$.   
Pulling back structure along these projections gives functors $\xi^*$ and $\xi_G^*$ that send $\sF$-algebras to $\sD$-algebras and $\sF_G$-algebras to $\sD_G$-algebras, and similarly for pseudoalgebras.  Taking full advantage of the equivariant context, we construct a section $\oursectionG \colon \sF_G\rtarr \sD_G$ to $\xi_G$ in
\autoref{sec:section}.  Pulling back along $\oursectionG$ gives the functor $\oursectionG^*$.  

However, since $\oursectionG$ does not preserve structure as strictly as one might hope, $\oursectionG^*$ takes strict algebras to pseudoalgebras.   As we explain in \autoref{sec:PowerLack},  $\St$ is a specialization of a general strictification functor due to Power and Lack \cite{Power, Lack} that rectifies the loss of strictness and lands us in the multicategory associated to the symmetric monoidal $2$-category $\FGAlg$ of strict $\sF_G$-algebras in categories internal to $G$-spaces and strict maps between them.  Specializing general theory developed in \cite{AddCat1, AddCat2}, we explain in \autoref{sec:StrictMult} how it extends to a multifunctor.

From here, $B$ is the standard classifying space functor and $\bS_G$ is the space level multiplicative equivariant infinite loop space machine of \cite{MMO}; $\FTop$ is the category of $\sF$-$G$-spaces and  $\SpG$ is the category of orthogonal $G$-spectra. We use these functors to complete the proof of \autoref{IntroMultiKG} in \autoref{sec:CattoTop}, and we combine our results here with results of \cite{MMO} and \cite{GMPerm} to prove \autoref{IntroBPQ} in \autoref{sec:BPQ}.

All of the multifunctors in \autoref{RoadMap} are symmetric except $\St$ and $\bS_G$.  We could equally well have used the slightly more elaborate but equivalent choice for $\bS_G$  constructed in \cite{GMMO}, which is symmetric.  However, although $\oursectionG^*$ is itself symmetric, loss of strict structure along it engenders the loss of symmetry of $\St$, as we  shall explain in \autoref{sec:nono}.

\begin{rem} We alert the reader to an alternative route to Theorems \ref{IntroMultiKG} and \ref{IntroBPQ} that  was found at the same time as the one presented here. It will be presented in \cite{MayToBe}.   It is illuminating, but it is more categorically intensive since it focuses on $2$-monads, which we have avoided here despite this being a paper that is  intrinsically all about them.  We will see in \cite{MayToBe} that the $k$-ary morphisms in our operadic multicategories are the pseudoalgebras over a $2$-monad $M_k$ and that the $M_k$  form a graded comonoid of $2$-monads.  Such structure also appears  in other multicategorical contexts. 

The alternative route uses a $2$-monadic reinterpretation of the vertical arrows in \autoref{RoadMap}, but it replaces the 
horizontal composite $\St\com \oursectionG^*$  by a multifunctor whose underlying map of $2$-categories is the composite  of Power-Lack strictification $\St\colon {\DGAlg} \rtarr \DGA$  and a derived variant of the left adjoint  
$\xi^G_*\colon {\DGA}\rtarr \FGAlg$  to the forgetful functor $\xi_G^{*}\colon \FGAlg\rtarr \DGA$. The section 
$\oursectionG^{*} \colon   {\DGA}\rtarr \psFGAlg$ is a categorical shortcut that avoids use of $\xi_*^G$, whose homotopical behavior is problematic.  The alternative route avoids any use of pseudoalgebras over $\sF$ or $\sF_G$, but we again lose symmetry, now due to the passage from $\xi^G_*$ to a homotopically well-behaved derived variant.   Conceivably, a more sophisticated derived variant might circumvent this.
\end{rem}

\subsection{Acknowledgements}  This project has taken a long rocky road, and we have many people to thank, too many to do justice to any of them.  We are happy to thank
Clark Barwick, Andrew Blumberg, Anna Marie Bohmann, David Gepner, Nick Gurski,  Mike Hill, Akhil Mathew, Niko Naumann, Thomas Nikolaus, Emily Riehl, David Roberts, Jonathan Rubin, Stefan Schwede, Michael Shulman,  and Dylan Wilson.   We apologize to anyone we may have forgotten.

\section{Preliminaries on operads and multicategories}\label{NotnSect} 

We begin here by introducing our categorical framework. We also recall the notions of operads, their algebras, and pseudomorphisms between those. Finally, we recall the notion of a multicategory.

\begin{notn}
Throughout the document, we will denote pseudomorphisms of various types (for example, see \autoref{OPseudoMap} or \autoref{Vpseudo}) by arrows $\squiggly{}{}$. 
\end{notn}

\subsection{$\sV$-categories}\label{NotnSectSub}

The categorical framework we begin with is the same as the one explained in more detail in \cite[Section 1]{AddCat1}, hence we shall be brief.  

\begin{assumpt}\label{ass1}
We let $\sV$ be a cartesian closed, bicomplete category.  
\end{assumpt}

The examples of primary interest are $\sV=\Top$ or $\sV = \GTop$, where $\Top$ is the 
category of (compactly generated weak Hausdorff) spaces and $\GTop$ is the category of $G$-spaces and $G$-maps for a finite group $G$.  The reader focused on topology is free to read $\sV$ as $\sU$, but nothing before \autoref{sec:section} (or after 
\autoref{sec:BPQ}) would change in any way.   We defer further discussion of the equivariant context to \autoref{sec:DG}.  

 As discussed in more detail in \cite[Section 1.1]{AddCat1}, we let $\VCat$ denote the 2-category of categories, functors, and natural transformations internal to $\sV$. We will refer to these as \ourdefn{$\sV$-categories}, \ourdefn{$\sV$-functors} , and \ourdefn{$\sV$-transformations}.
Thus any $\sV$-category $\cC$ consists of objects $\ob\cC$ and $\mor\cC$ in $\sV$, and source, target, identity, and composition structure maps, which are all required to be morphisms in $\sV$. 
For $\cC$ and $\cD$ in $\VCat$, a $\sV$-functor $\cC\rtarr \cD$ is given by morphisms $\ob\cC \rtarr \ob\cD$ and $\mor\cC \rtarr \mor\cD$ in $\sV$ that are suitably compatible with the internal category structure.  A $\sV$-transformation $\al$ between $\sV$-functors $F_1,F_2\colon \cC \rightrightarrows \cD$ is given by a morphism $\al\colon \ob\cC \rtarr \mor\cD$ in $\sV$ that makes the naturality diagrams commute.

Since $\sV$ is complete, so is $\VCat$. 
\begin{assumpt}\label{ass2}We assume that $\VCat$ is moreover 
cocomplete.
\end{assumpt}

This assumption holds if  either $\sV$ is locally presentable or if $\sV=\sU$  \cite[(3.24) and (3.25)]{StreetCosmoi}. For similar reasons, it is also true for $\sV=G\sU$.

\begin{rem}\label{VCatclosed}
We note that $\VCat$ is cartesian closed since $\sV$ is assumed to be cartesian closed \cite[Lemma 2.3.15]{Johnstone}.
\end{rem}

\begin{defn}\label{chaotic}
We say that a $\sV$-category $\cC$ is \ourdefn{chaotic}, 
or indiscrete, if  the source and target maps yield an isomorphism
$ {\mor}\cC \xrtarr{(S,T)} {\ob}\cC \times {\ob}\cC.$
\end{defn}

Chaotic $\sV$-categories and their properties are discussed in detail in \cite[\S1.2]{AddCat1}. 

\subsection{Based $\sV$-categories}\label{sec:based} 

Let $\bpt$ denote the terminal object of $\sV$.  A basepoint of an object $V$ of $\sV$ is a map $\bpt \rtarr V$ in $\sV$. Write $\sV_{\bpt}$ for the category of based objects of $\sV$ and based maps.  Let $\bpt$ also denote the  $\sV$-category whose object and morphism objects are both given by the object $\bpt$ of $\sV$.  

\begin{defn}\label{Catstar}
A \ourdefn{based $\sV$-category}, or {\em $\sV_*$-category}, is a category internal to $\sV_{\bpt}$. Equivalently, it is a category $\cC$ internal to $\sV$ equipped with a $\sV$-functor $\bpt\rtarr \cC$. Its structure maps source, target, identity, and composition must be in $\sV_{\bpt}$.
There are corresponding notions of based functors, called  {\em $\sV_{\bpt}$-functors}, namely $\sV$-functors compatible with basepoints, and based $\sV$-transformations, called 
{\em $\sV_{\bpt}$-transformations}, whose component morphisms $\ob\cC \rtarr \mor\cD$ are based.  As noted in \cite[Remark 1.6]{AddCat1}, the resulting $2$-category, here denoted $\VCatp$, can be identified with $\VCat_\ast$. 
\end{defn}

\begin{rem}\label{object}
For a $\sV$-category or $\sV_*$-category $\cC$, 
an {\em object} of $\cC$ will mean a functor $\ast \rtarr \cC$ or, equivalently, a morphism $\ast \rtarr \ob\cC$ in $\sV$.  {We warn the reader that we are using the term ``object" in a technical sense.  For example, when $\sV$ is the category of $G$-spaces, an object is a $G$-fixed point of the $G$-space  $\ob\cC$, hence $\cC$ may have no objects.}
\end{rem}

We can form the wedge and smash product of based $\sV$-categories $\cA$ and $\cB$ via the pushout diagrams 
\[ \xymatrix{ 
\bpt \ar[r] \ar[d] & \cA \ar@{-->} [d] \\
\cB \ar@{-->}[r] & \cA \vee \cB\\} \ \ \text{and} \ \ 
\xymatrix{\cA \vee \cB \ar[r] \ar[d] & \cA\times \cB \ar@{-->}[d] \\
\bpt \ar@{-->}[r] & \cA\sma \cB\\} \] 
just as for spaces. 
Since the objects functor $\ob\colon \VCat \rtarr \sV$ has both a left and a right adjoint and therefore preserves limits and colimits, it follows that 
\begin{equation}\label{ObCommutesSma}
\ob( \cA \sma \cB) \iso \ob(\cA) \sma \ob(\cB)
\end{equation}
for $\sV_\bpt$-categories $\cA$ and $\cB$. 

By the universal property of the smash product, a $\sV_\bpt$-functor 
$\cA\sma \cB\rtarr \cC$ corresponds to a   $\sV$-functor $\cA\times \cB\rtarr \cC$ whose restriction to $\ast\times \cB$ and $\cA \times \ast$ is the constant functor at the basepoint of $\cC$.
This will allow us to define maps from smash products by specifying basepoint conditions  on  $\sV$-functors defined on products.

Similarly, a $\sV_\bpt$-transformation
\[\xymatrix{\cA \sma \cB \rtwocell^F_G{\varphi} & \cC}\]
corresponds to a $\sV$-transformation of functors defined on $\cA \times \cB$ whose restriction to $\cA\times \ast$ and $\ast\times \cB$ is the identity.

\begin{rem}\label{VCatpclosed}
Our standing assumptions on $\sV$ imply that $\VCatp$ is closed symmetric monoidal with internal hom adjoint to $\sma$ (see \cite[Lemma 4.20]{EM2}, \cite[Construction 3.3.14]{Riehl}).\footnote{The associativity of $\sma$ is not formal from the universal property of the pushout and requires $\VCat$ to be closed.}
\end{rem}

We will use the symmetric monoidal structure to enrich categories over $\VCatp$ starting in \autoref{sec:V2}.

In our applications, categories often have disjoint base objects, and we write $\cA_+$ for the coproduct (disjoint union in the relevant examples) 
of $\bpt$ with an unbased $\sV$-category $\cA$.  Then
\[ \cA_+ \sma \cB_+ \iso (\cA\times \cB)_+ \]
(see \cite[Lemma 3.3.16]{Riehl}).

 \subsection{Operads in $\VCat$}\label{sec:Ops}

We will work throughout with a reduced operad $\oO$ in $\VCat$, \ourdefn{reduced} meaning that $\oO(0)$ is the trivial category $\bpt$. We will often assume that $\oO$ is chaotic, meaning that each $\sV$-category $\oO(n)$ is chaotic. We will use the notation
\[\gamma\colon \oO(k)\times \oO(j_1)\times\dots\times\oO(j_k)\rtarr \oO(j_1+\dots+j_k)\]
 for the operad structure $\sV$-functors and $\oid\colon \ast\rtarr \oO(1)$  for the unit object in $\oO(1)$.

\begin{defn}\label{Oalg}An $\oO$-algebra is an object $\aA$ in $\VCat$ equipped with action $\sV$-functors
\[ \tha(n)\colon \oO(n)\times \aA^n \rtarr \aA\]
that are appropriately $\SI_n$-equivariant, unital, and associative, as in \cite{MayGeo}.  Since $\oO$ is assumed to be reduced, the functor $\tha(0)\colon \bpt\rtarr \aA$ specifies a basepoint  $0 = 0_{\aA} \in \aA$.
\end{defn}

We will be using non-strict maps between $\oO$-algebras, called $\oO$-pseudomorphisms. The full definitions of these and of $\oO$-transformations between them are given in \cite{CG} and, with some minor  emendations, in  \cite[Definitions 2.23 and 2.24]{AddCat1}.  We shall not repeat details, but we remind the reader of the key features. 

\begin{defn}\label{OPseudoMap} Let $\aA$ and $\aB$ be $\oO$-algebras.  An \ourdefn{$\oO$-pseudomorphism} 
$\squiggly{\aA}{\aB}$ is a $\sV$-functor $F\colon \aA\rtarr \aB$ such that $F( 0_\aA) = 0_\aB$, 
together with invertible $\sV$-transformations $\pa_n$  
\[\xymatrix{   \oO(n)\times \aA^n \ar[d]_{\tha(n)} \ar[r]^{\id\times F^n}  \drtwocell<\omit>{<0> \, \pa_n} & \oO(n)\times \aB^n \ar[d]^{\tha(n)} \\
     \aA \ar[r]_-{F} & \aB } \]
for $n\geq 0$ such that 
$\pa_0$ and the restriction of $\pa_1$ along $\oid\times \id\colon \aA \cong \bpt\times \aA \rtarr \oO(1)\times \aA$ are  identity $\sV$-transformations and such that the appropriate equality of associativity pasting diagrams relating the $\pa_n$ 
to the structure maps of the operad holds (see \cite[Definition 2.23]{AddCat1}).
It is a (strict) $\oO$-map if the $\pa_n$ are identity $\sV$-transformations.
\end{defn}

\begin{defn}\label{Otran} An \ourdefn{$\oO$-transformation} between $\oO$-pseudomorphisms $E$ and $F$ 
is a $\sV$-transformation $\om\colon E \Longrightarrow F$ such that the equality 
\[\xymatrixrowsep{1.3cm}\xymatrix{  
 \oO(n)\times \aA^n \ar[d]_{\tha(n)} \ar[r]^{\id\times E^n}  \drtwocell<\omit>{<-1.5> \, \pa_n} & \oO(n)\times \aB^n \ar[d]^{\tha(n)}  \\
     \aA \rtwocell^E_F{\om} & \aB }          
     \ \ \ \ 
\raisebox{-6ex}{=}
\ \ \ \ 
\xymatrixrowsep{1.3cm}\xymatrixcolsep{1.8cm}\xymatrix{
 \oO(n)\times \aA^n \ar[d]_{\tha(n)} \rtwocell^{\id\times E^n}_{\id\times F^n}{\qquad \id\times \om^n}  \drtwocell<\omit>{<1.5> \, \pa_n} & \oO(n)\times \aB^n \ar[d]^{\tha(n)} \\
  \aA \ar[r]_F & \aB}\]
holds for all $n$.  We do {\em not} require the $\om$ to be invertible.
\end{defn}

\begin{notn}\label{notn:OAlg}
We will work throughout with the 2-category $\OAlg$ of $\oO$-algebras,
 $\oO$-pseudomorphisms, and  $\oO$-transformations. 
\end{notn}
 
There is a more general definition of $\oO$-pseudoalgebras, as defined in \cite{CG, AddCat1}, but we choose not to introduce it since it is not needed for the purposes of this paper.

\subsection{Review of multicategories}\label{Multicat}

We shall not repeat the complete definition of a multicategory given in such sources as \cite{EM1, Lein, Yau}.
A multicategory ${\cM}$ 
has a class $\ob(\cM)$ 
of objects and for each sequence 
$\ul{a} = \{a_1,\dots,a_k\}$ of objects, where $k\geq 0$,
 and each object $b$, it has a set of $k$-ary morphisms 
\[ {\cM}_k(\ul a;b) ={\cM}_k(a_1,\dots,a_k;b). \] 
A quintessential example is that of $k$-linear maps in the category of vector spaces, which is why $k$-ary morphisms in arbitrary multicategories are sometimes 
called $k$-linear maps, even when there is no linear structure in sight.

Throughout, we understand multicategories to be symmetric, 
so that the symmetric group $\SI_k$ acts from the right on the collection of $k$-ary morphisms via maps
\[ \si\colon {\cM}_k(a_1,\dots,a_k;b)\rtarr {\cM}_k(a_{\si(1)},\dots,a_{\si(k)};b). \] 
For each object $a$ there is an identity $1$-ary morphism $a\rtarr a$ 
and there are composition functions 
\begin{equation}\label{gamma} \ga\colon {\cM}_k(\ul b; c) \times {\cM}_{j_1}(\ul a_1;b_1)\times \dots \times {\cM}_{j_k}(\ul a_k;b_k)
\rtarr {\cM}_j(\{\ul{a}_1,\dots,\ul{a}_k\};c),
\end{equation}
where $\ul b$ is a $k$-tuple,  $\ul{a}_q$ for $1\leq q\leq k$ is a $j_q$-tuple,  and, with $j = j_1 + \dots + j_k$, 
$\{\ul{a}_1,\dots,\ul{a}_k\}$ is the $j$-tuple $\{a_{1,1}, \dots, a_{1,j_1}, \dots, a_{k,1},\dots, a_{k,j_k} \}$.

The $\ga$ are subject to direct generalizations of the associativity, identity, and equivariance properties required 
of an operad in \cite{MayGeo}. These properties are spelled out diagrammatically in \cite[2.1]{EM1}
and, with exceptional care, in \cite[Chapter 11]{Yau}.\footnote{The colored operads in \cite{Yau} are symmetric multicategories
with a set of objects, called colors, but the generalization to a class of objects is evident.}

All of our multicategories are enriched in $\Cat$, but since that is only used peripherally we will not go into detail.\footnote{In fact, they are enriched in $\VCat$ when $\sV$ is closed.}
A multicategory with one object is then the same thing as an operad in $\Cat$. 
Multicategories are often called colored operads, with objects thought of as colors.  The objects and $1$-ary morphisms of a multicategory ${\cM}$ specify its underlying category, which is often also denoted ${\cM}$ by abuse of notation.   

\begin{rem}\label{symmonmulti}
There is a canonical\footnote{There is a slight subtlety here.  It has been said that there is a choice of such multicategories depending on the chosen order of associating variables.  With
an unbiased operadic definition of a symmetric monoidal category, the specification of $\Mult(\aC)$ is unambiguous.} multicategory $\Mult(\aC)$ associated to a symmetric monoidal category $(\aC,\otimes)$. 
Its objects are those of $\aC$, and 
\[ \Mult_k(\aC)(a_1, \dots, a_k;b) = \aC(a_1\otimes \cdots \otimes a_k, b). \]
It has the evident symmetric group actions and units. In schematic  elementwise notation, using the notations of \autoref{gamma}, the composite of
a $k$-ary morphism $F\colon \ul{b} \rtarr c$ with $(E_1,\dots, E_k)$, where $E_r\colon \ul{a}_r\rtarr b_r$ is a $j_r$-ary morphism for $1\leq r\leq k$, is the
composite
\begin{equation}\label{gamma2}
\xymatrix@1{\bigotimes\limits_{1\leq r\leq k}\bigotimes\limits_{1\leq s\leq j_r}a_{r,s} \ar[r]^-{\otimes_r E_r} & \bigotimes\limits_{1\leq r\leq k} b_r \ar[r]^-{F} & c. \\} 
\end{equation}
This generalizes the example in vector spaces where $k$-linear maps correspond to maps out of the tensor product.   
\end{rem}

A morphism $\bF\colon \cM\rtarr \cN$ of multicategories, called a multifunctor, is a function
$\bF\colon \ob(\cM)\rtarr \ob(\cN)$ together with functions
\[ \bF\colon \cM_k(a_1,\dots,a_k;b) \rtarr \cN_k(\bF a_1,\dots,\bF a_k;\bF b)  \]
for all objects $a_i$ and $b$ such that $\bF(\id_a) = \id_{\bF(a)}$ and $\bF$ preserves composition.  If these functions are  $\SI_k$-equivariant, we say that $\bF$ is a \emph{symmetric multifunctor}.   A lax monoidal (resp. lax symmetric monoidal) functor between symmetric monoidal categories gives rise to a multifunctor (resp. symmetric multifunctor) between the corresponsing multicategories.

 Given a multicategory $\cM$, one can define  the notion of monoid in $\cM$ (see \cite[Example~2.1.11]{Lein} or \cite[\S14.2]{Yau}). This can be done
 using ``parameter multicategories'', so that a monoid in $\cM$ is given by a multifunctor out of the appropriate parameter multicategory into $\cM$.  One can similarly define the notion of module over a monoid (see \cite[Definition~2.5]{EM1}). 
 These notions agree with the usual ones when dealing with the multicategory associated to a symmetric monoidal category as in \autoref{symmonmulti}. A multifunctor preserves associative and unital algebraic structures, and a symmetric multifunctor moreover preserves commutative ones.

\section{The multicategory of $\oO$-algebras}
\label{sec:Oalg}

The goal of this section is to establish a multiplicative structure on the category $\OAlg$ of algebras over an operad. 
After some initial setup in \autoref{Operads}, we introduce the key concept of a pseudo-commutative operad in \autoref{BE}, following Corner and Gurski \cite{CG}. 
We then establish a multicategory $\Mult(\oO)$ for any pseudo-commutative operad $\oO$ in \autoref{Multiop},  following Hyland and Power \cite{HP}, and describe some variants in \autoref{variant}. 
Finally, we show that the free $\oO$-algebra functor extends to a multifunctor in \autoref{Free}.

\subsection{The intrinsic pairing of an operad}\label{Operads} 
Surprisingly, the following elementary structure implicit in the definition of an operad is central to our work.  It is present in any reduced operad $\oO$ in any cartesian monoidal category $\sW$.  

\begin{defn}\label{intrinsicpairing} The \ourdefn{intrinsic pairing} $\opair \colon (\oO, \oO) \rtarr \oO$ of an operad $\oO$ is given by the composites
\[ \xymatrix@1{\oO(j)\times \oO(k) \ar[r]^-{\id\times \DE^j} & \oO(j)\times \oO(k)^j \ar[r]^-{\ga} & \oO(jk),\\} \]
where $\ga$ is the structure map of the operad and $j\geq 0$ and $k\geq 0$. 
\end{defn}

Thinking of $\ga$ as specifying additive structure, the ``product" $\opair$ is taking seriously
that $jk = k+\cdots + k$.  Thus the intrinsic pairing is an operadic manifestation of the grade school lesson that 
multiplication is iterated addition.

\begin{rem} The intrinsic pairing is {\em not} a pairing of operads in the sense originally defined in \cite[1.4]{MayPair}.
For many operads occurring naturally in topology, such as the little cubes or Steiner operads, the intrinsic pairing appears to be of no real
interest.  However, as we shall see  in \fullref{Multiop}, it appears naturally when trying to construct a multicategory of algebras over an operad. 
\end{rem} 

\begin{prop}\label{intrinsicmon} Let $\oO$ be an operad in a cartesian monoidal category $\sW$. Then $\ul{\oO} = \coprod_{j\geq 0} \oO(j)$ is a monoid in $\sW$ with product $\opair$ and unit  the unit object $\oid\in \oO(1)$. It has a zero object  $\ast\in \oO(0)$.
\end{prop}

\begin{proof} The unit properties of an operad are 
$\ga(\oid; x) = x$ and $\ga(x;\oid^j) = x$
for $x\in \oO(j)$. These say that $\oid$ is a unit for $\ul{\oO}$.  The associativity of the pairing
is an easy diagram chase from the following special case of the associativity diagram for $\ga$ in the definition of an operad. 
\[
\xymatrix{ \oO(j)\times \oO(k)^j \times \oO(\ell)^{jk} \ar[rr]^-{\ga\times \id} \ar[dd]_{\iso}  & &
\oO(jk)\times \oO(\ell)^{jk} \ar[dr]^{\ga} &\\
& & &\oO(jk\ell)\\
\oO(j)\times (\oO(k) \times \oO(\ell)^{k})^{j}  \ar[rr]_-{\id\times \ga^j} & &\oO(j)\times \oO(k\ell)^j \ar[ur]_{\ga}  }\]
Since $\oO$ is reduced, $\ast\in \oO(0)$ is a zero element.
\end{proof} 

Consider the category $\SI$ of sets $\ul{n}=\{1,\dots,n\}$ and isomorphisms. It is bipermutative under disjoint union and cartesian product. To be precise, the two monoidal structures, $\oplus$ and $\spair$, are given by  sum and product at the level of objects. To apply $\oplus$ to 
permutations $\si\in \SI_j$ and $\ta\in \SI_k$ and regard the result as a permutation of the 
$j+k$ letters $\ul{j+k} = \{1, \dots, j+k\}$, we are implicitly applying the evident isomorphism  
\[\ze_{j,k}\colon \ul{j+k} \rtarr \ul{j}\amalg \ul{k},\]
then taking the disjoint union of $\si$ and $\ta$, 
and then applying $\ze_{j,k}^{-1}$.  That is, $\si\oplus \ta$ is defined by the commutative diagram 
\begin{equation}\label{oplus}
\xymatrix{
\ul{j+k} \ar[r]^-{\si\oplus \ta} \ar[d]_{\ze_{j,k}} & \ul{j+k} \\
\ul{j}\amalg \ul{k} \ar[r]_-{\si\amalg \ta} & \ul{j}\amalg \ul{k}. \ar[u]_{\ze_{j,k}^{-1}}}
\end{equation}
Similarly, define 
\begin{equation}\label{lexico}
\la = \la_{j,k} \colon \ul{jk} \rtarr \ul{j}\times \ul{k}
\end{equation}
to be the order-preserving bijection, where $\ul{j}\times \ul{k}$ is ordered 
lexicographically. Then, 
$\mu\spair \nu$ is defined by the commutative diagram
\begin{equation}\label{otimes}
\xymatrix{
\ul{jk} \ar[r]^-{\mu\spair \nu} \ar[d]_{\la_{j,k}} & \ul{jk} \\
\ul{j}\times \ul{k} \ar[r]_-{\mu\times \nu} & \ul{j}\times \ul{k}. \ar[u]_{\la_{j,k}^{-1}}\\}
\end{equation}

Recall that the associative operad $\Ass$ is given by $\Ass(n)=\SI_n$. Then $\spair$ gives the intrinsic pairing of \autoref{intrinsicpairing} on $\Ass$. Moreover,
 if we think of the groups $\SI_j$ as categories
with a single object and thus think of $\Ass$ as an operad in $\Cat$, then $\ul{\Ass} = \SI$ and the monoidal structure of \autoref{intrinsicmon} is given by $\spair$.

In particular, $e_j\spair e_k = e_{jk}$. Since $\spair$ is a group homomorphism, it is equivariant 
in the sense that $\mu \si \spair \nu \ta = (\mu\spair \nu)(\si\spair \ta)$.  Clearly $e_1\spair \nu = \nu$ 
and $\mu\spair e_1 = \mu$. 

\begin{rem}\label{rem:opair-equivariant}
The equivariance formulas for an operad $\oO$ imply that the pairing on $\SI$ and the pairing on $\oO$ are compatible in the sense that for all $\si \in \SI_j$ and $\ta \in \SI_k$ the following diagram commutes.
\[
\xymatrix{
\oO(j) \times \oO(k) \ar[r]^-{\opair} \ar[d]_{\si \times \ta}
    & \oO(jk)
 \ar[d]^{\si\spair\ta}  \\
\oO(j) \times \oO(k) \ar[r]_-{\opair}       & \oO(jk)
 }
\]
\end{rem}

\begin{defn}\label{tau} Let  $\ta_{j,k}\in \SI_{jk}$ be the permutation specified by the composite
\[ \xymatrix@1{ 
\ul{jk} \ar[r]^-{\la_{j,k}} & \ul{j} \times \ul{k} 
\ar[r]^-{t} & \ul{k} \times \ul {j} \ar[r]^-{\la_{k,j}^{-1}} & \ul{kj} = \ul{jk} \\} \]
It reorders the set $\ul{j}\times\ul{k}$ from lexicographic ordering to reverse lexicographic ordering. 
Clearly $\ta_{j,k}^{-1} = \ta_{k,j}$ and $\ta_{1,n} = e_n = \ta_{n,1}$.
\end{defn}

The  $\tau_{i,j}$ are the symmetry isomorphisms
for $\spair$ in $\SI$. More precisely, for $\mu\in \SI_j$ and $\nu\in \SI_k$, we have the commutative diagram
\[ \xymatrix{
\ul{jk} \ar[d]_{\ta_{j,k}} \ar[r]^-{\la_{j,k}} & \ul{j} \times \ul{k} \ar[r]^-{\mu\times \nu}\ar[d]_{t} 
& \ul{j} \times \ul{k} \ar[d]^{t} \ar[r]^-{\la_{jk}^{-1}} & \ul{jk} \ar[d]^{\ta_{j,k}} \\
\ul{kj} \ar[r]_-{\la_{k,j}} & \ul k \times \ul j \ar[r]_{\nu\times \mu} 
& \ul k\times \ul j \ar[r]_-{\la_{k,j}^{-1}} & \ul{kj}.\\} \] 
That is,
\begin{equation}\label{ttt}
\ta_{j,k} (\mu\spair \nu) = (\nu\spair \mu) \ta_{j,k} \ \ 
\text{or equivalently} \ \ (\mu\spair \nu)\ta_{k,j} = \ta_{k,j}(\nu\spair \mu).
\end{equation}

\subsection{Pseudo-commutative operads}\label{BE}

Recall the permutativity operad $\sP$ 
from \cite[Definition 4.1]{GMPerm}. It is the chaotic categorification of $\Ass$, and as an operad in $\Cat$, its algebras are in one-to-one correspondence with permutative categories. The intrinsic pairing $\opair$ on $\sP$ is inherited from that of $\Ass$. 
Note that equation \autoref{ttt} implies that the diagram
\[ \xymatrix{
\sP(j)\times \sP(k) \ar[d]_{t} \ar[r]^-{\opair} & \sP(jk) \ar[d]^{\ta_{k,j}} \\
\sP(k)\times \sP(j) \ar[r]_-{\opair} & \sP(kj) } \]
does not in general commute. Rather, since $\sP(kj)$ is chaotic, there exists a natural isomorphism $\al_{j,k}\colon \ta_{k,j} \circ \opair \Longrightarrow \opair \circ t$. This is an example of a pseudo-commutative operad, as defined by Corner and Gurski \cite[\S4]{CG}.  

We summarize their definition here.  As usual, when defining categorical structures, coherence axioms are essential for completeness and rigor. However, they can be lengthy and may not make  for enjoyable reading.  To avoid disrupting  the flow of exposition, we generally defer their precise formulation to \autoref{sec:coh}.  In particular, we give the axioms needed to complete the following definition in \autoref{cohpseudo}.\footnote{The original definition of \cite{CG} requires some minor corrections that are given there.}

\begin{defn}\label{pseudocom} Let $\oO$ be an operad in $\VCat$. A \ourdefn{pseudo-commutative structure} on an operad $\oO$ is a 
collection of 
invertible $\sV$-transformations, one for each $(j,k)$, of the form
\begin{equation}\label{alphas}
\xymatrix{
\oO(j) \times \oO(k) \ar[r]^-{\opair} \ar[d]_{t} \drtwocell<\omit>{<0>\quad\ \  \alpha _{j,k}}  & \oO(jk)  \ar[d]^{\tau_{k,j}}  \\
\oO(k) \times \oO(j)  \ar[r]_-{\opair} & \oO(kj). \\}
\end{equation}
The $\al_{j,k}$ must satisfy coherence axioms for identity, symmetry, equivariance, and operadic compatibility that are specified and discussed in \autoref{cohpseudo}.
\end{defn}

If $\oO$ is chaotic, such transformations $\al$ always exist and all conditions are automatically satisfied \cite[\S1.2]{AddCat1}, thus giving the following result.
\begin{lem}(\cite[Corollary~4.9]{CG})\label{ChaoticPseudoCom} A chaotic operad has a unique pseudo-com\-mutative structure.
\end{lem}

We will often write ``pseudo-commutative operad'' when we really mean ``operad equipped with a pseudo-commutative structure'', but there is no ambiguity when $\oO$ is chaotic, and we later prefer to specialize to chaotic operads.

\subsection{The multicategory of $\oO$-algebras}\label{Multiop} 
Hyland and Power \cite{HP} show that there is a multicategory of algebras over a pseudo-commutative monad, and Corner and Gurski show in \cite{CG} that the monad corresponding to a pseudo-commutative operad is pseudo-commutative in the sense of \cite{HP}. We follow these sources to describe the multicategory of algebras over a pseudo-commutative operad. 

For the pairings of $\oO$-algebras we want to consider maps $F\colon \aA \times \aB \rtarr \aC$ that preserve the algebra structure on each variable up to canonical isomorphism. For example, if $+$ denotes a binary operation in $\oO$, we need to make sense of a distributivity 
law of the general form
\[  F(a,b_1+ b_2)\cong F(a,b_1) + F(a,b_2). \]

Diagonal maps enter since $a$ appears 
once on the left and twice 
on the right. The definition contains a number of schematic
coherence diagrams to the effect that whenever two natural transformations have a chance to be equal
they are equal.  We shall explain the diagrams after giving the definition.  
The following maps $s_i$
play a key role.\footnote{The elementary maps $s_i$ correspond to the ``strengths'' $t_i$ in Hyland and Power 
\cite[p. 156]{HP}; in their categorical treatment, the existence of $t_i$ with suitable properties is an axiom on a given 2-monad, 
although they do make the strengths explicit in the case of permutative categories.}

\begin{notn}\label{strengths} Let $\aA_i$, $1\leq i\leq k$, be $\sV$-categories and let $n\geq 0$. Define $s_i$ to be the 
composite $\sV$-functor displayed in the diagram
\[ \scalebox{0.95}{
\xymatrix {
\aA_1\times \dots\times \aA_{i-1} \times \oO(n) \times \aA_i^n \times \aA_{i+1}\times \dots \times \aA_k 
\ar[r]^-{s_i} \ar[d]_t^\iso  & \oO(n) \times (\aA_1\times \cdots \times \aA_k)^n \\
\oO(n) \times \aA_1 \times \dots \times \aA_{i-1} \times \aA_i^n \times \aA_{i+1} \times \dots \times \aA_k. 
\ar[r]_-{\id \times \DE} & \oO(n) \times \aA_1^n\times \cdots \times \aA_k^n \ar[u]_{}^{\iso}\\
}} \]
Here $t$ is the evident transposition, $\DE$ is obtained by applying the diagonal maps 
$\aA_j\rtarr \aA_j^n$ for $j\neq i$, and the right hand isomorphism is obtained by transposing from a product 
of $n$th powers to an $n$th power of a product.
\end{notn}

\begin{defn}\label{MultiO} 
Let $\oO$ be a (reduced) pseudo-commutative operad in $\VCat$. 
We define the (symmetric) multicategory $\Mult(\oO)$ of $\oO$-algebras and pseudomorphisms. 
Its underlying 2-category is $\OAlg$, so its objects, morphisms, and $2$-cells are the $\oO$-algebras, the
$\oO$-pseudomorphisms (\autoref{OPseudoMap}), and the $\oO$-transformations (\autoref{Otran}).
Recall that since $\oO$ is reduced, all $\oO$-algebras are assigned basepoints. Its 0-ary morphisms are (unbased) maps $\ast \rtarr \aB$, that is, they correspond to a choice of object in $\aB$.
For $k>1$, its $k$-ary morphisms $(\aA_1,\dots,\aA_k) \rtarr \aB$
are the tuples $(F,\delta_i)$, where 
\begin{enumerate}[(a)]
\item
$F\colon \aA_1\sma \cdots \sma \aA_k \rtarr \aB$ is a 
 $\sV_*$-functor, which we may equally well express as a $\sV$-functor
$F\colon \aA_1\times \cdots \times \aA_k \rtarr \aB$
such that $F(a_1,\dots,a_k)$ is equal to $0_\aB$ if any object 
$a_i$ is $0_{\aA_i}$ and $F(f_1,\dots,f_k)$ is $\id_{0_\aB}$ if any $f_i$ is $\id_{0_{\aA_i}}$, and
\item the $\de_i$, $1\leq i\leq k$, are sequences of invertible $\sV$-transformations $\de_i(n)$ as indicated in the following diagram.
\begin{equation}\label{deltan}
 \xymatrix{
\oO(n)\times (\aA_1\times \cdots \times \aA_k)^n \ar[rr]^-{\id\times F^n} 
\ddrrtwocell<\omit>{<0>\quad  \  \delta_i(n)} 
& &  \oO(n)\times \aB^n  \ar[dd]^{\tha(n)} \\ 
\aA_1\times \cdots \times \oO(n)\times \aA_i^n \times \cdots \times \aA_k \quad \quad \quad 
\ar[u]^-{s_i} \ar[d]_{\id\times\tha(n)\times \id}  &  &\\
\aA_1\times \cdots \times \aA_k \ar[rr]_-{F} & &\aB \\}
\end{equation}
\end{enumerate}
The distributivity isomorphisms $\delta_i(n)$ must satisfy coherence axioms that are specified and discussed in \autoref{cohMultO}.

For $\si\in \SI_k$, the right action of $\SI_k$ on the $k$-ary morphisms of $\Mult(\oO)$ sends 
$(F,\delta_i)\colon (\aA_1,\dots ,\aA_k)\rtarr \aB$
to  
the composite
\begin{equation}\label{OObSym}
\xymatrix{
\aA_{\si(1)}\times \cdots \times \aA_{\si(k)} \ar[r]^-{\si} & \aA_1 \times \cdots\times \aA_k \ar[r]^-{F} & \aB,}
\end{equation}
where $\si$ denotes the reordering of terms given by $\si(a_{\si(1)},\dots,a_{\si(k)}) = (a_1,\dots,a_k).$
Permuting the indices, the $\de_i$ for $F\si$ are inherited from the $\de_i$ for $F$.  
Precisely, $\de_{\si^{-1}(i)}(n)$ for $F\si$ is induced from $\de_i(n)$ for $F$ by pasting the
defining diagram  \autoref{deltan} 
to the right of the following diagram, which is easily checked to be commutative.
\[
 \scalebox{0.9}{
 \xymatrix{
\oO(n)\times (\aA_{\si(1)}\times \cdots \times \aA_{\si(k)})^n \ar[rr]^-{\id\times \si^n} 
& &  \oO(n)\times (\aA_1\times \cdots \times \aA_k)^n \\
\aA_{\si(1)}\times \cdots\times \oO(n)\times \aA_{i}^n \times  \cdots  \times \aA_{\si(k)}
\ar[u]_-{s_{\si^{-1}(i)}} \ar[d]_{\id\times\tha(n)\times \id}  \ar[rr]^-{\si} & & 
\aA_1 \times \cdots \times \oO(n)\times \aA_i^n  \times \cdots  \times \aA_k
\ar[u]_-{s_i} \ar[d]^{\id\times\tha(n)\times \id} \\
\aA_{\si(1)}\times \cdots \times \aA_{\si(k)} \ar[rr]_-{\si} & & \aA_1\times \cdots \times \aA_k\\
}} \]
Since $i = \si\si^{-1}(i)$, the term $\oO(n)\times \aA_i^n$ appears in the $\si^{-1}(i)$th factor of 
the middle left term. Note that $(F\si)\ta = F(\si\ta)$, both mapping $\aA_{\si\ta(1)}\times\cdots\times \aA_{\si\ta(k)}$ to $\aB$.

The identity functor of $\aA$ gives the unit element $\id_\aA\in \Mult(\oO)(\aA;\aA)$.  With the notation for sequences from \autoref{Multicat}, the composition multiproduct 
\[\xymatrix{ \Mult(\oO)(\ul{\aB};\aC)\times \prod_{q=1}^k \Mult(\oO)(\ul{\aA}_q;\aB_q) \ar[r]^-{\ga} & \Mult(\oO)( \{\ul{\aA}_{1},\dots \ul{\aA}_k\};\aC) \\} \]
is given by 
\[ \ga(F;E_1,\dots,E_k) = F\com (E_1\sma \cdots \sma E_k). \]
We identify $\{(q,r)\}$, $1\leq q\leq k$ and $1\leq r\leq j_q$, with $\{1\leq i\leq j_1 + \cdots j_k\}$ by letting $(q,r)$ correspond to 
$i=j_1+\cdots + j_{q-1} + r$.   Then $\de_{q,r}(n)$ for the multicomposition is given by pasting the diagrams for $\de_q^F$ and $\de_r^{E_q}$.  We show this explicitly in the case of $(q,r)=(1,1)$.   The general case is shown similarly.  We use the notation $\mybox{\aA_q}$ for the product  $\aA_{q,1}\times \dots \times \aA_{q,j_q}$.

\begin{equation}\label{delta-comp}
 \scalebox{0.8}{\xymatrix{
 \oO(n)\times \left(\hspace{0.1em} \mybox{\aA_{1}}\times \cdots \times \overline{\aA_{k}} \hspace{0.1em} \right)^{\! n} \ar[rr]^-{\id\times (E_1  \times \dots \times E_k)^n} 
 & &  \oO(n)\times (\aB_{1} \times \cdots \times \aB_{k})^n \ar[rr]^-{\id \times F^n} 
 & & \oO(n)\times \aC^n \ar[ddd]^{\tha(n)}  
 \\ 
\oO(n)\times  \mybox{\aA_{1}}^{\,n}\times \mybox{\aA_{2}}\times \cdots \times \mybox{\aA_{k}} \ar[u]^{s_1}\ar[rr]^-{\id\times E_1^n \times E_2 \times \dots \times E_k} 
\ddrrtwocell<\omit>{<0>\qquad  \qquad  \delta_1^{E_1}(n)\times \id} 
& &  \oO(n)\times \aB_{1}^n \times \aB_2\times \cdots \times \aB_{k}  \ar[dd]^{\tha(n)\times \id }\ar[u]_{s_1} \ddrrtwocell<\omit>{<0>\qquad  \delta_1^{F}(n)} \\ 
\oO(n) \times \aA_{1,1}^n \times \aA_{1,2} \times \cdots  \times \aA_{k,j_k} \quad \quad \quad 
\ar[u]^-{s_1\times \id } \ar[d]_{\tha(n)\times \id}  &  &  & &
\\
\mybox{\aA_{1}}\times \cdots \times \mybox{\aA_{k}} \ar[rr]_-{E_1\times \dots \times E_k} & &\aB_1 \times \dots \times \aB_k \ar[rr]_-{F}& & \aC \\}}
\end{equation}

One can check that the $\delta$s satisfy the coherence axioms, and that this composition is associative, unital and respects equivariance. 
 Further verifications are needed to show that this all really does specify a multicategory. 
For example, the symmetry axiom  (ii) in \autoref{cohpseudo} is  used in the verification that $F\sigma$ 
satisfies the axioms when $F$ does. However, we omit further details.
We have translated the axioms
of Hyland and Power to our operadic setting. Their \cite[Proposition 18]{HP} applies to show
that $\Mult(\oO)$ is a multicategory enriched in the category $\Cat$ of small categories.
We learned the central role played by pseudo-commutativity  from them.
\end{defn}

\subsection{Variants of $\Mult(\oO)$ and comparisons}\label{variant} 
We remark that our definition of $\Mult(\oO)$ applies almost verbatim to define a multicategory of $\oO$-pseudoalgebras as defined in \cite{CG, AddCat1}. Pseudoalgebras  over an operad are defined by relaxing the  coherence diagrams for the operadic multiplication with the structure map of the algebra to only commute up to coherent natural isomorphisms. The only axioms in the definition of $\oO$-pseudomorphisms (listed in \autoref{cohMultO}) that would differ slightly
for  $\oO$-pseudoalgebras as opposed to $\oO$-algebras are \eqref{OperCompAxiom} and \eqref{DeiDejCommute}. In these conditions it wouldn't make sense to ask for equality of 2-cells as written, since the maps on the boundary of the 2-cells would not be equal, and instead one would need to paste them with the coherence isomorphisms from the definition of $\oO$-pseudoalgebras. 

When $\sV=\Top$ and $\oO$ is the permutativity operad $\sP$,
the multicategory $\Mult(\sP)$ 
is not quite the same as the multicategory of symmetric strict monoidal categories 
defined by Hyland and Power \cite{HP} and the multicategory of permutative categories defined by Elmendorf and Mandell \cite{EM1}. 
The difference is that we have taken our distributivity 2-cells $\de_i$ to be invertible. Neither \cite{HP} nor \cite{EM1} do so, and we have 
drawn our arrows in the direction used in \cite{HP}, which is opposite to the choice in \cite{EM1}.
This difference in the choice of direction of the 2-cells $\de_i$ would matter 
if we relaxed the isomorphism requirement. For example, the strengths $s_i$ of \autoref{strengths} 
would no longer be relevant with the opposite choice, so the definition in \cite{EM1} would no
longer be a specialization of \cite{HP} and would not be compatible with the conventions of Corner 
and Gurski \cite{CG} or with LaPlaza's classical coherence theory for symmetric bimonoidal categories
 \cite{La, La2}.  It  would therefore lead to some erroneous conclusions, as explained in \cite[Scholium 12.3]{Rant2}.  

The work of \cite{EM1} used the classical biased definition of permutative categories rather than 
its unbiased operadic equivalent, and that simplifies details when comparing operadic algebraic
structures to their classical biased equivalents.  
With our unbiased operadic reformulation, the equivariant generalization is immediate. For example, we can take
$\oO$ to be the categorical equivariant Barratt-Eccles operad
$\sP_G$ of \cite{GMPerm}
to obtain the multicategory $\Mult(\sP_G)$ of genuine permutative $G$-categories; genuine permutative and symmetric monoidal $G$-categories are defined to be  $\sP_G$-algebras and $\sP_G$-pseudoalgebras, respectively,  in \cite{GMPerm, AddCat1}. Our work also applies to the normed symmetric monoidal categories of Rubin \cite{RubinNSMC}, which are defined as algebras over an operad, but which also admit a biased definition.

\subsection{The free $\oO$-algebra multifunctor $\bOp$}\label{Free}
As explained in \cite[Section 4]{Rant1}, a  (reduced) operad $\oO$ in a category $\sW$ has two associated monads, $\bO$ defined on the ground category $\sW_*$ and $\bOp$ defined on the ground category $\sW$.   Their categories of algebras are isomorphic.  The first takes the basepoint as given and requires the basepoint built in by the operad action to agree with the given one, and it is defined using basepoint identifications.  The second just builds in the basepoint by the action.  The first is the one central to topology and is in principle more general.  For an unbased object $X\in\sW$, 
\begin{equation}\label{OplusX}
  \bOp X = \bO(X_+) =  \coprod_{j\geq 0} \oO(j) \times_{\SI_j} X^j.
\end{equation}
This is a based object with basepoint given by the inclusion of $\ast = \oO(0) \times X^0$.

Starting on the category level with $\sW = \VCat$, we prefer to avoid basepoint identifications and we therefore focus on $\bOp$.   When $\sV$ is $\Top$ or $\GTop$, applying the classifying space functor gives an operad $B\oO$, and we denote the associated monad by $\bOp^{top}$. If $\oO$ is $\SI$-free,  in the sense that $\SI_j$ acts freely on $\oO(j)$ for each $j$, we then have the basic commutation relation
\begin{equation}\label{BcommutesO}  B(\bOp \cC) \iso  \bOp^{top}(B \cC)  \end{equation} 
for $\cC\in \VCat$, and that is essential to our applications.

Thus fix a chaotic operad $\oO$ in $\VCat$ in this section.

\begin{defn}\label{omdefn} For $\sV$-categories $\cC$ and $\cD$, define a $\sV_*$-functor
\[ \om\colon \bOp \cC\sma \bOp \cD\rtarr \bOp(\cC\times \cD) \]
by passage to  orbits from the maps
\begin{equation}\label{maps}  
\xymatrix@1{
\oO(j)\times \cC^j \times \oO(k) \times \cD^k \ar[r]^{t} & \oO(j)\times \oO(k) \times \cC^j  \times \cD^k
\ar[r]^-{\opair\times \ell} & \oO(jk)\times (\cC\times \cD)^{jk}. \\} 
\end{equation}
The map $\ell$ here is defined using the lexicographic ordering $\la$ of \autoref{lexico}; explicitly, it is given by
\[ (c_1, \dots, c_j), (d_1, \dots, d_k)\mapsto (c_1, d_1), \dots (c_1, d_k), \dots \dots, (c_j, d_1), \dots, (c_j, d_k).\]
Since $\oO$ is reduced, the maps \autoref{maps}  factor through the smash product and, using the equivariance axiom for operad composition, they also pass to orbits with respect to symmetric group actions; therefore they induce a well-defined 
map $\om$.  
\end{defn}

\begin{prop}
\label{prop:alphaNatTran}
A pseudo-commutativity structure on $\oO$ induces  an invertible $\sV_\bpt$-transformation
\begin{equation}\label{pseudocommMonad} \xymatrix{
\bOp \cC \sma \bOp \cD \ar[r]^-\om  \ar[d]_{t} \drtwocell<\omit>{<0> \  \al} & \bOp(\cC\times \cD) \ar[d]^-{\bOp(t)} \\
\bOp \cD \sma \bOp \cC \ar[r]_-\om & \bOp(\cD\times \cC).
}\end{equation}
\end{prop}

\begin{pf}
By axiom \eqref{symmpseudocomequiv} from  \fullref{pseudocom},
the transformations $\alpha_{j,k}$ descend to orbits.
It is straightforward to check that they  define the claimed invertible $\sV_\bpt$-transformation.
\end{pf}

We now extend $\om$ to a binary morphism in $\Mult(\oO)$.
We need to define the transformations $\de_i(n)$ for $i=1,2$. Careful inspection shows that we can take $\de_1(n)$ to be the identity transformation. That is, we claim that the following diagram commutes.
\[ \scalebox{.95}{
\xymatrix{
\oO(n)\times (\bOp \cC\times \bOp \cD)^n \ar[rr]^-{\id\times {\om}^n} 
& &  \oO(n)\times \bOp(\cC \times \cD)^n  \ar[dd]^{\tha(n)} \\  
\oO(n)\times (\bOp \cC)^n \times  \bOp \cD 
\ar[u]^-{s_1} \ar[d]_{\tha(n)\times \id}  &  &\\
\bOp \cC \times \bOp \cD \ar[rr]_-{\om} & &\bOp(\cC\times \cD) \\}}
\]
In particular, it is important to notice that the variables in $\cC$ and $\cD$ are arranged lexicographically under either 
composite. The variables in the operad agree under the composite by iterated application 
of the associativity diagram for the structure maps $\ga$, as in \fullref{intrinsicmon}. 

We define $\de_2(n)$ as the following pasting diagram
\begin{equation}\label{yuckytwo}
\scalebox{0.7}{
\xymatrix @C=2em{
\oO(n)\times (\bOp \cC\times \bOp \cD)^n \ar[rrr]^-{\id\times {\om}^n} \ar[dr]^{\id\times t^n}
& \drtwocell<\omit>{<0>\qquad \  \id\times\al^n}  & &  \oO(n)\times \bOp(\cC\times  \cD)^n  \ar[dddd]^{\tha(n)}  \ar[dl]^{\id\times \bOp(t)^n}\\  
 & \oO(n)\times (\bOp \cD\times \bOp \cC)^n \ar[r]^(0.55){\id\times\om^n} & \oO(n) \times \bOp(\cD\times \cC)^n \ar[dd]^{\tha(n)}  \\
\bOp \cC\times  \oO(n)\times (\bOp \cD)^n \ar[uu]^-{s_2} \ar[dd]_{\id\times\tha(n)} \ar[r]^{t} &   \oO(n)\times (\bOp \cD)^n \times \bOp \cC \ar[u]^{s_1} \ar[d]_{\tha(n)\times \id}  &\\
 & \bOp \cD \times \bOp \cC \ar[r]^\om  \ar[dl]_t \drtwocell<\omit>{<0>\quad \  \al} 
& \bOp( \cD\times \cC) \ar[dr]^{\bOp(t)} \\
\bOp \cC\times \bOp  \cD \ar[rrr]_-{\om} & & &\bOp(\cC\times \cD), \\}
}
\end{equation}
where the inner pentagon is $\de_1=\id$. 

\begin{prop}
The axioms on $\al$ imply that $(\om, \de_1=\id, \de_2)$  satisfies  the conditions for a 2-ary morphism in $\Mult(\oO)$ 
given in \autoref{cohMultO}. 
\end{prop}

\begin{pf}
The most difficult axiom to verify is \eqref{DeiDejCommute}. 
The operadic compatibility condition on $\alpha$ (\autoref{pseudocom} \eqref{pseudocommhex1}) is central to this verification. 
We leave the details to the reader.
\end{pf}

For ease of notation we will denote this morphism by $\om$.
The following result follows easily from the definition of $\om$.

\begin{lem}\label{om-nat}
The pairing $\om$ is natural, in the sense that for all $\sV$-functors $F\colon \cA \rtarr \cC$, $H\colon \cB \rtarr \cD$, the diagram 
\[
\xymatrix{
(\bOp \cA , \bOp \cB) \ar[r]^-{\om} \ar[d]_{(\bOp F , \bOp H)} &\bOp(\cA \times \cB) \ar[d]^{\bOp(F\times H)}\\
(\bOp \cC , \bOp \cD) \ar[r]_-{\om} &\bOp(\cC \times \cD)
}
\]
in $\Mult(\oO)$ commutes.
\end{lem}

\begin{lem}\label{om-assoc}
Given $\sV$-categories $\cC$, $\cD$, and $\cE$, the following diagram in $\Mult(\oO)$ commutes.
\[\xymatrix{
(\bOp \cC , \bOp \cD , \bOp \cE) \ar[r]^-{(\om , \id)}  \ar[d]_{(\id , \om)} &(\bOp( \cC \times  \cD), \bOp \cE )\ar[d]^{\om}\\
(\bOp \cC , \bOp (\cD \times  \cE ) )\ar[r]_-{\om}& \bOp (\cC \times \cD \times  \cE)
}
\]
\end{lem}

\begin{proof}
Since the intrinsic pairing $\opair$ is associative, the diagram
\[\xymatrix{
\bOp \cC \sma \bOp \cD \sma \bOp \cE \ar[r]^-{\om \sma \id}  \ar[d]_{\id \sma \om} &\bOp( \cC \times  \cD) \sma \bOp \cE \ar[d]^{\om}\\
\bOp \cC \sma \bOp (\cD \times  \cE ) \ar[r]_-{\om}& \bOp (\cC \times \cD \times  \cE)
}
\]
commutes; the compatibility of the $\de_i$ follows from the conditions on $\al$.
\end{proof}

Thus, given $\sV$-categories $\cC_1,\dots,\cC_k$, we have a corresponding $k$-ary morphism
\[\om_k\colon (\bOp\cC_1,\dots, \bOp\cC_k) \rtarr \bOp(\cC_1\times \dots \times  \cC_k)\]
 in $\Mult(\oO)$, defined by using $\om$ iteratively.  Since this is a composition in $\Mult(\oO)$, the $\de_i$ for $\om_k$ are obtained from those for $\om$ using the pasting \autoref{delta-comp}. We take $\om_1=\id$ and take $\om_0$ to be the choice of object $(\oid,\ast)\in \oO(1)\times \ast \subset \bOp(\ast)$.

\begin{thm}\label{freeinput} 
The functor $\bOp$ from $\sV$-categories to $\oO$-algebras extends to a multifunctor 
\[ \bOp\colon \Mult(\VCat) \rtarr \Mult(\oO).\] 
\end{thm}

We do not claim that the multifunctor we construct is symmetric, and we
shall show that it is not in \autoref{ONotSym}.

\begin{proof} 
In ${\bf{Mult}}(\VCat)$, a $k$-ary morphism is just a $\sV$-functor
$F\colon \cC_1\times \cdots \times \cC_k\rtarr \cD$. Its image under the multifunctor is the composite
\[ \xymatrix@1{( \bOp\cC_1 , \dots , \bOp\cC_k )\ar[r]^-{\om_k} 
& \bOp(\cC_1\times \cdots \times \cC_k) \ar[r]^-{\bOp F} & \bOp \cD.} \]

It is clear that this assignment sends the identity of $\cC$ to the identity of $\bOp(\cC)$. The fact that the assignment preserves composition follows from the functoriality of $\bOp$ and the naturality of $\om$ (\autoref{om-nat}). 
\end{proof}

\begin{rem}\label{ONotSym}
The multifunctor $\bOp$ is not symmetric. To see this, consider $\id_{\cC\times\cD}$ as a bilinear map in $\Mult(\VCat)$. On one hand, if we first hit it with the symmetry $t$ and then $\bOp$, we end up with a bilinear map whose $\sV_*$-functor is given by
\[ \xymatrix{ \bOp\cD \sma \bOp\cC \ar[r]^-{\om} & \bOp(\cD\times \cC) \ar[r]^-{\bOp(t)} & \bOp(\cC\times \cD).}\]
If, on the other hand, we first do $\bOp$ and next $t$, we get
\[ \xymatrix{\bOp\cD \sma \bOp\cC \ar[r]^-{t} & \bOp(\cC) \times \bOp(\cD) \ar[r]^-{\om} & \bOp(\cC \times \cD).}\]
These maps do not agree since they differ by the two-cell $\alpha$.
\end{rem}

\section{$\sV_\bpt$-$2$ categories and their algebras and pseudoalgebras}\label{sec:V2}

This section establishes terminology and notation that will be used frequently in the coming sections. 
Much of what we do in \autoref{NotnSectSub2} is to describe explicitly what it means to do enriched category theory over the 2-category $\VCat$.
In \autoref{sec:Cpseudo}, we introduce algebras and pseudoalgebras in this context, and we set ourselves up to discuss multiplicative structures on our categories of algebras by introducing monoidal structures on our enriched categories in \autoref{sec:permutative}.

\subsection{$\sV_{\bpt}$-$2$-categories}\label{NotnSectSub2} 

All of our categories of operators, which we introduce in Sections \ref{Dalgs}  and \ref{COFG}, are examples of $\sV_\bpt$-$2$-categories, and we explain what those are here. Returning to \autoref{NotnSect}, we note that $\VCat$ and $\VCatp$ are symmetric monoidal $2$-categories, with cartesian and smash product, respectively. As such, we can enrich and weakly enrich over them, in the sense of \cite{GS}.  Classically, a $2$-category is a category enriched in the category $\Cat$ of (small) categories.

\begin{defn}
 We refer to categories, functors, and natural transformations enriched in $\VCat$ as \ourdefn{$\sV$-$2$-categories}, \ourdefn{$\sV$-$2$-functors}, and \ourdefn{$\sV$-$2$-natural transformations}, respectively.  
 Similarly, we call categories, functors, and natural transformations enriched in $\VCatp$  \ourdefn{$\sV_*$-$2$-categories}, \ourdefn{$\sV_*$-$2$-functors}, and \ourdefn{$\sV_*$-$2$-natural transformations}, respectively.
\end{defn} 

We briefly unpack these definitions. A $\sV$-2-category $\cC$ consists of a collection of objects ($0$-cells) and  a morphism $\sV$-category $\cC(c,d)$ for each pair of objects  $(c,d)$ (giving objects in $\sV$ of $1$-cells and $2$-cells).  For a $\sV_*$-2-category we moreover have that each $\cC(c,d)$ has a basepoint, and the composition factors through the smash product.
If $\cC$ and $\cD$ are $\sV$-$2$-categories, a $\sV$-$2$-functor $F\colon \cC\rtarr \cD$ is given by a function $F$ on objects and $\sV$-functors $\cC(c,d) \rtarr \cD(F(c),F(d))$ satisfying the evident unit and associativity conditions (strictly). For a $\sV_*$-2-functor we further require that the latter are $\sV_*$-functors.

Finally, if $E$ and $F$ are $\sV$-2-functors  $\cC\rtarr \cD$, where $\cC$ and $\cD$ are $\sV$-$2$-categories, a
 $\sV$-$2$-natural transformation $\ze\colon E\Longrightarrow F$ consists of
1-cells\footnote{As in \autoref{object}, a 1-cell here means a $\sV$-functor $\ast\rtarr  \cD(E(c),F(c))$.}
$\ze_c\colon E(c) \rtarr F(c)$, meaning objects of $\cD(E(c),F(c))$, such that the naturality diagrams 
\[  \xymatrix{
\cC(c,d) \ar[r]^-{E} \ar[d]_{F} & \cD(E(c),E(d)) \ar[d]^{(\ze_d)_*} \\
\cD(F(c),F(d)) \ar[r]_-{(\ze_c)^*} & \cD(E(c),F(d))\\
}
\]
commute for all pairs $(c,d)$ of objects of $\cC$. The same is true for a $\sV_*$-2-natural transformations, except that the diagram above lives in $\VCatp$.

As in \cite[Section~1.3]{MMO}, there is a close relationship between $\sV$-2-categories with a zero object and $\sV_*$-2-categories, which we now explain. For a $\sV$-2-category $\cC$, we say that $\bzo$ in $\cC$ is a zero object if $\cC(c,d) \cong \ast$ if either $c=\bzo$ or $d=\bzo$.
If $\cC$ has a zero object, each hom $\sV$-category $\cC(c,d)$ is based with basepoint
\begin{equation}
\label{eq:bptC}
\xymatrix@1{  0_{c,d}\colon \ast \iso \cC(\bzo,d)\times \cC(c,\bzo) \ar[r]^-{\circ} & \cC(c,d).}\\ 
\end{equation}

\begin{prop}\label{ZeroObjEnrCatp}
For $\cC$ a $\sV$-2-category with zero object $\bzo$, the basepoints \eqref{eq:bptC} give $\cC$ an enrichment in $\VCatp$.
\end{prop}

\begin{proof}
 We only need to check that composition in $\cC$ factors through the smash product. The associativity diagram
\[ \xymatrix{  \cC(d,e) \times \cC(\bzo,d)\times \cC(c,\bzo) \ar[r]^-{\id\times \circ} \ar[d]_-{\circ\times\id} & \cC(d,e)\times \cC(c,d) \ar[d]^{\circ}\\
\ast\iso \cC(\bzo,e) \times \cC(c,\bzo) \ar[r]_-{\circ}   & \cC(c,e)  \\ } \]
shows that composition sends $\cC(d,e) \times 0_{c,d}$ to $0_{c,e}$  and the symmetric argument shows that composition also sends $0_{d,e} \times \cC(c,d)$
to  $0_{c,e} $.
\end{proof}

The following result, whose proof we leave to the reader, characterizes $\sV_\bpt$-2-functors for $\sV$-2-categories with a zero object.

\begin{prop}\label{ReducedVCatpEnr}
If $\cC$ and $\cD$ have zero objects, then $\sV_{\bpt}$-2-functors  $F\colon \cC \rtarr \cD$ correspond bijectively to
 $\sV$-2-functors $F\colon \cC \rtarr \cD$  that are \ourdefn{reduced}, in the sense that $F(\bzo_\cC)\cong\bzo_\cD$. Moreover, $\sV_*$-2-natural transformations correspond bijectively to $\sV$-2-natural transformations between reduced $\sV$-2-functors (that is, there is no extra condition).
\end{prop}

Following \cite[Sections~3.5 and 3.7]{GS}, we now introduce the weaker notions of $\sV$-pseudofunctor and $\sV$-pseudonatural transformation, together with their based 
variants.\footnote{In the language of \cite{AddCat1}, we are restricting to normal
$\sV$-pseudofunctors.}

\begin{defn}\label{Vpseudo}  A \ourdefn{$\sV$-pseudofunctor} $\squiggly{F\colon \cC}{\cD}$
between $\sV$-$2$-categories consists of a function $F$ on objects and $\sV$-functors 
\[ F\colon \cC(b, c)\rtarr \cD(F(b),F(c))\]
such that the following diagram commutes
\begin{equation}\label{unitpseudo}
  \xymatrix{
\ast \ar[r]^-{\id_c} \ar[dr]_{\id_{F(c)}}  & \cC(c,c) \ar[d]^F\\
& \cD(F(c),F(c)), \\
}
\end{equation}
together with invertible coherence $\sV$-transformations 
\begin{equation}\label{varphi}
\xymatrix{
\cC(b,c)\times \cC(a,b) \ar[d]_{\circ} \drtwocell<\omit>{<0>  \ \varphi}
\ar[r]^-{F\times F} & \cD(F(b),F(c))\times \cD(F(a),F(b)) \ar[d]^{\circ}\\
\cC(a,c) \ar[r]_-{F} & \cD(F(a),F(c)) \\
} 
\end{equation}
that are unital ($\varphi_{\id,-}$ and $\varphi_{-,\id}$ are the identity) and  associative in the sense that the relevant equalities of pasting diagrams relating to triple composition hold 
(see for example \cite[Section 1.1]{LeinBB} or \cite[Section 3.5]{GS}).
\end{defn} 

The based variant is defined similarly.

\begin{defn}\label{V*pseudo} 
 A \ourdefn{$\sV_\bpt$-pseudofunctor} $\squiggly{F\colon \cC}{\cD}$
between $\sV_\bpt$-$2$-categories is a $\sV$-pseudofunctor such that the functors on morphisms are $\sV_\bpt$-functors, with the unit diagram \autoref{unitpseudo} replaced with the diagram of $\sV_\bpt$-functors with source $\bpt \amalg \bpt$,\footnote{This is equivalent to requiring that the unit diagram \autoref{unitpseudo} commutes as a diagram of underlying $\sV$-functors.} and the transformations $\varphi$ in \autoref{varphi} descend to $\sV_\bpt$-transformations of $\sV_\bpt$-functors with source $\cC(b,c)\sma \cC(a,b)$.
\end{defn}

That is, in the diagram \autoref{varphi} for a $\sV_\bpt$-pseudofunctor, both instances of $\times$ in the top row are replaced by $\sma$.

\begin{defn}\label{Vpseudotran} Let $E$ and $F$ be $\sV$-pseudofunctors $\squiggly{\cC}{ \cD}$, where $\cC$ and $\cD$ are $\sV$-$2$-categories.  Then a 
$\sV$-pseudotransformation\footnote{This is meant as an abbreviation of $\sV$-pseudonatural transformation.}  
$\ze\colon E \Longrightarrow F$ consists of 
1-cells $\ze_c\colon E(c) \rtarr F(c)$
 for objects $c\in\cC$ and invertible $\sV$-transformations
\[\xymatrix{
\cC(b,c) \ar[r]^-F \ar[d]_E \drtwocell<\omit>{<0> \ \quad \ze_{b,c}} 
& \cD(F(b),F(c)) \ar[d]^{(\ze_b)^*}\\
\cD(E(b),E(c)) \ar[r]_{(\ze_c)_*} & \cD(E(b),F(c))\\}\]
for objects $b,c\in \cC$ such that 
the component of $\ze_{c,c}$ at $\id_c$ is the identity 2-cell for all $c\in \cC$
and  the relevant coherence diagram expressing compatibility with composition commutes  (see for example
\cite[Section~1.2]{LeinBB} or \cite[Section~3.7]{GS}).  
\end{defn}

Again, the definition is essentially the same in the based context.

\begin{defn}\label{V*pseudotran}
Let $E$ and $F$ be $\sV_\bpt$-pseudofunctors $\squiggly{\cC}{\cD}$, where $\cC$ and $\cD$ are $\sV_\bpt$-$2$-categories.
Then a $\sV_\bpt$-pseudotransformation $\ze\colon E \Longrightarrow F$ is a $\sV$-pseudotransformation of the underlying $\sV$-pseudofunctors such that each $\ze_{b,c}$ is a $\sV_\bpt$-transformation, meaning that  its component at $0_{b,c}$ is the identity.
\end{defn}

As in \autoref{ReducedVCatpEnr}, we have the following characterization of $\sV_\bpt$-pseudofunctors.

\begin{prop}\label{ReducedVCatpEnrPseudo} 
If $\cC$ and $\cD$ have zero objects, then $\sV_{\bpt}$-pseudofunctors  \mbox{$\squiggly{ \cC }{ \cD}$} correspond bijectively to
 $\sV$-pseudofunctors $\squiggly{ \cC }{ \cD}$  that are \ourdefn{reduced}, in the sense that $F(\bzo_\cC)\cong\bzo_\cD$, and are such that $\varphi$ restricts to the identity on the subcategories $0_{b,c}\times \cC(a,b)$ and $\cC(b,c)\times 0_{a,b}$. 
 
Moreover, $\sV_\bpt$-pseudotransformations $\ze\colon E \Longrightarrow F$ correspond bijectively to $\sV$-pseudotransformations between the underlying $\sV$-pseudofunctors  
 (that is, there is no extra condition).
\end{prop}

As in any enriched setting, we have the following construction. 

\begin{defn}\label{defn:enrsmash}
If $\cC$ and $\cD$ are $\sV$-$2$-categories, their product $\sV$-$2$-category $\cC\times \cD$ has objects  $\ob(\cC)\times \ob(\cD)$ and morphism $\sV$-categories
\[  (\cC\times \cD)((c,d), (c',d')) = \cC(c,c')\times \cD(d,d'), \]
with the evident composition.  If $\cC$ and $\cD$ 
have zero objects, so does
 $\cC\times \cD$, with zero object $(\bzo,\bzo)$. 
 
Similarly, if $\cC$ and $\cD$ are $\sV_\bpt$-2-categories, we have a $\sV_*$-$2$-category 
$\cC\sma \cD$ with objects $\ob(\cC)\times \ob(\cD)$ and with
\[  (\cC\sma\cD)((c,d), (c',d')) = \cC(c,c')\sma \cD(d,d') \]
as the $\sV_*$-category of morphisms. 
\end{defn}
 
The following remark allows us to regard $2$-categories as $\sV$-$2$-categories and establishes a convenient context.

\begin{rem} \label{bUbV}  The underlying set functor $\bU\colon \sV\rtarr \mathbf{Set}$ specified by $\bU X = \sV(\ast,X)$ has left adjoint  $\bV\colon \bf{Set} \rtarr \sV$ specified by $\bV S =  \coprod_{s\in S} \ast$, the coproduct of copies of the terminal object $\ast$ indexed on the elements of the set $S$. Thus
\begin{equation}\label{trivadj1}
\mathbf{Set}(S,\bU X) \iso \sV(\bV S,X). 
\end{equation}
When $\sV$ is strongly complete \cite[Definition 6.2]{AddCat1}, a mild condition that holds in all relevant examples,
$\bV$ preserves finite limits.  We can apply this with $\sV$ replaced by  $\VCat$, so that
\begin{equation}\label{trivadj2}
\mathbf{\Cat}(\cB,\bU \aX) \iso  \VCat(\bV \cB,\aX).
\end{equation}
for a category $\cB$ and a $\sV$-category $\aX$.  We regard categories $\cB$ as discrete $2$-categories, meaning that they have only identity $2$-cells.  Applying $\bV$ to the hom categories of $2$-categories allows us to change their enrichment from $\Cat$ to $\VCat$.  Thus we may regard categories $\cB$ as
$\sV$-$2$-categories, and we agree to do so without change of notation. We then call $\cB$ a \ourdefn{discrete $\sV$-$2$-category}.   In our examples, $\cB$ has a zero object and therefore gives rise to a $\sV_\bpt$-2-category.
\end{rem}

 \subsection{Algebras and pseudoalgebras 
over  $\sV_\bpt$-$2$-categories}\label{sec:Cpseudo}

Classically, algebras over a category $\cD$ enriched in based spaces
 can be defined as enriched functors $\aX\colon \cD\to \sU_\bpt$. 
 The ``action" of the category $\cD$ can be seen by translating this enriched functor into adjoint form as the data of a family of compatible continuous maps $\cD(a,b)\sma \aX(a)\to  \aX(b)$ indexed by objects $a,b$ of $\cD$. Similarly, algebras over an operad $\oO$ in $\sU$ can be  defined either via action maps $\oO(j)\times X^j\to X$ or equivalently via a map of operads from $\oO$ to the endomorphism operad $\mathrm{End}(X)$.  

\begin{rem}
This second interpretation can always be given when the ambient category is closed, 
as we have assumed.
 Indeed, if $\sV$ is closed, it has an internal hom object $\ul{\sV}(V,W)$ for each pair of objects and an adjunction 
\[   \sV(V\times W, Z) \iso \sV(V, \ul{\sV}(W,Z)). \]
Each object $V$ then has an endomorphism operad $\mathrm{End}(V)$ with  
\[\mathrm{End}(V)(j) = \ul{\sV}(V^j,V).\]
 An action of an operad $\oO$ in $\sV$ on an object $V\in \sV$ is then the same as a map of operads $\oO\rtarr \mathrm{End}(V)$.  This gives an adjoint specification of $\oO$-algebras.
 \end{rem}

As $\VCat$ is cartesian closed (see \autoref{VCatclosed}),
we will write $\CatV$ for the resulting $\sV$-2-category. 
Similarly, $\VCatp$ is closed by \autoref{VCatpclosed}, and we will write $\CatVp$ for the resulting \Vptwo.

If $\cC$ is a $\sV$-$2$-category, it is natural to define a $\cC$-algebra $\aX$ to be a $\sV$-$2$-functor  
$\aX\colon \cC\rtarr \CatV$.  If $\cC$ has a zero object,
we say that $\aX$ is reduced if $\aX(\bzo)= \ast$.  We have the following enhancement of Propositions
\ref{ReducedVCatpEnr} and \ref{ReducedVCatpEnrPseudo}.

\begin{prop}\label{ReducedAlg}
Let $\cC$  be a $\sV$-2-category with zero object. Then
\begin{enumerate}
\item {\Vptwofun}s $\cC\rtarr \CatVp$ correspond bijectively to reduced  {\Vtwofun}s 
$\cC\rtarr \CatV$.
\item $\sV_{\bpt}$-pseudofunctors  \mbox{$\squiggly{ \cC }{ \CatVp}$} correspond bijectively to reduced \linebreak
 $\sV$-pseudofunctors $\squiggly{ \cC }{ \CatV}$  such that $\varphi$ restricts to the identity on the subcategories $0_{b,c}\times \cC(a,b)$ and $\cC(b,c)\times 0_{a,b}$. 
\end{enumerate}
\end{prop}

\begin{proof}
 To prove (1), let $\aX\colon \cC \rtarr \CatV$ be a reduced {\Vtwofun}. Then for $c\in \cC$, the adjoint of the map
 \[\cC(\bzo,c) \rtarr \CatV(\aX(\bzo),\aX(c))\] 
 endows $\aX(c)$ with a basepoint:
 \[\ast \cong \cC(\bzo,c)\times \aX(\bzo)\rtarr \aX(c).\]
  Functoriality can then be used to check that this indeed gives rise to a {\Vptwofun} $\cC\rtarr \CatVp$. For the converse we apply \autoref{ReducedVCatpEnr} with $\cD=\VCatp$, and then forget the basepoints to get a map with target $\VCat$. 
  
The argument for (2) is similar, with the caveat that the condition on $\varphi$ is necessary to ensure that we do get a map with target $\VCatp$.
\end{proof}

\begin{defn}\label{strict}  
Let $\cC$ be a $\sV_\ast$-2-category.
A (strict) $\cC$-algebra is a {\Vptwofun} $\aX\colon \cC\rtarr \CatVp$.
Unpacking the definition in adjoint form, this consists of a function that assigns a $\sV_\ast$-category $\aX(c)$ to each object $c$ of $\cC$, 
together with action $\sV_\bpt$-functors 
\[  \tha\colon \cC(c,d) \sma \aX(c)\rtarr \aX(d) \] 
such that the unit and composition diagrams of $\sV_\bpt$-functors
\[ \xymatrix{ (\ast \amalg \ast)\sma \aX(c) \ar[d]_{\id_c\sma \id}  \ar[r]^-{\iso} &  X \\
\cC(c,c)\sma \aX(c) \ar[ur]_{\tha} & } 
\ \ \ \text{and} \ \ \ 
\xymatrix{
\cC(d,e)\sma \cC( c,d) \sma \aX(c)  
\ar[r]^-{\id\sma \tha} \ar[d]_{\circ\sma \id} 
&  \cC( d, e)\sma \aX( d) \ar[d]^{\tha} \\
\cC( c, e) \sma \aX(c)  \ar[r]_-{\tha} & \aX(e) \\}
\]
commute. 
\end{defn}

\begin{rem}
The unit diagram can equally well be expressed as a diagram of $\sV$-functors, with source  $\ast \times \aX(c)$. 
\end{rem}

We do not discuss general $\cC$-pseudoalgebras here, leaving such consideration for \cite{AddCat2}. 
However, we will need a version of pseudoalgebras in the special case of $\cC=\sF_G$ starting in \autoref{sec:section}. We define these now. 

\begin{defn}\label{wCAlg}
Let $\cC$ be a $\sV_\ast$-2-category. A \ourdefn{weak $\cC$-pseudoalgebra} $\aX$  is a $\sV_{\bpt}$-pseudofunctor  \mbox{$\squiggly{ \cC }{ \CatVp}$. Unpacking the definition in adjoint form, this}
consists of a function that assigns a $\sV_\bpt$-category $\aX(c)$ to each $c \in \cC$, together with action
$\sV_\bpt$-functors 
\[ \tha: \cC(c,d) \sma \aX(c) \rtarr \aX(d)\]
and invertible $\sV_\bpt$-transformations
\[ \xymatrix{
\cC(d,e) \sma \cC(c,d) \sma \aX(c) \ar[r]^-{\id \sma \tha} \ar[d]_{\circ \sma\id} \drtwocell<\omit>{<0>\quad  \  \varphi} &\cC(d,e) \sma \aX(d) \ar[d]^\tha \\
\cC(c,e) \sma \aX(c) \ar[r]_-\tha & \aX(e)
}\]
which are the identity when either morphism is the identity 1-cell. Moreover, the $\varphi$ are required to be coherent as in \autoref{Vpseudo}.
\end{defn}

\begin{rem} The strictness with respect to basepoints encoded in the definitions above expresses the intuition that additive zero objects should behave strictly in multiplicative structures.  The strictness with respect to identity arrows expresses the intuition that identity operations should be the identity.
\end{rem}

\begin{rem}\label{NoPseudoAlgs}
Though it is not the choice made in this article, it is also possible to consider less general $\cC$-pseudoalgebras. As $\VCatp$ is closed, these can be described as $\sV_\bpt$-pseudofunctors $\squiggly{\cC}{\CatVp}$ that are strictly functorial
on a  subcategory $\cB \subset \cC$.
Different choices of $\cB$ lead to several possible versions of pseudoalgebras. While we will only deal with the weakest possible variant of pseudoalgebras here, the different types of pseudomorphisms described in the next definition will play an important role. We will return to consider pseudoalgebras in more detail in \cite{AddCat2}.
\end{rem}

\begin{defn}\label{Cpseudo}  
Let $\cC$ be a $\sV_\ast$-2-category and  $\cB\subset \cC$ a subcategory, and let
 $\aX$ and $\aY$ be $\cC$-algebras (or weak $\cC$-pseudo\-algebras).  
 A \ourdefn{$(\cC,\cB)$-pseudomorphism}, 
$F\colon \squiggly{\aX}{\aY}$ is a $\sV_\bpt$-pseudo\-transformation between $\sV_\bpt$-2-functors (or $\sV_\bpt$-pseudo\-functors) that is strict when restricted to $\cB$. Unpacking this definition in adjoint form, this
consists of  $\sV_\bpt$-functors 
\[F(c)\colon \aX(c) \rtarr \aY(c)\] 
together with invertible $\sV_\bpt$-transformations $\de$ as in the diagram
\begin{equation}\label{delta1}
 \xymatrix{
\cC(c, d) \sma \aX(c) \ar[r]^-{\id\sma F(c)} \ar[d]_{\tha}   \drtwocell<\omit>{ \hspace{.5em} \de}  &  \cC(c,d) \sma \aY(c) \ar[d]^{\tha} \\
\aX(d) \ar[r]_-{F(d)} & \aY(d)\\
} 
\end{equation}
that are the identity transformation on the subcategory $\cB(c,d) \sma \aX(c)$.
The $\de$ must satisfy the relevant equalities of pasting diagrams relating to composition in $\cC$. 
 If $\de$ is always the identity, then $F$ is a (strict) \ourdefn{$\cC$-map} .  
Thus the condition on $\cB$ says that $F$ restricts to a strict $\cB$-algebra map.
 We refer to the extreme case, in which $\cB$ consists only of identity morphisms and basepoint morphisms,  as \ourdefn{weak $\cC$-pseudomorphisms}.
 \end{defn}
 
 The following notion will also be needed later.

\begin{defn}\label{defn:leveleqv}
A $\cC$-pseudomorphism $\squiggly{F \colon \aX}{\aY}$ of $\cC$-pseudoalgebras is said to be a \ourdefn{level equivalence} if each component $F(c)\colon \aX(c) \rtarr \aY(c)$ is an (internal) equivalence of $\sV_\bpt$-categories.
\end{defn} 

\begin{defn}\label{Ctran} A \ourdefn{$\cC$-transformation} $\om$ between $\cC$-pseudomorphisms $E$ and $F$ 
consists of $\sV_{\bpt}$-transformations $\om_c \colon E(c) \Longrightarrow F(c)$ that 
are suitably compatible with the $\sV_\bpt$-transformations $\de^E$ and $\de^F$, as in \cite[Section~3.10]{GS}.
We do {\em not} require the $\om_c$ to be isomorphisms. 
\end{defn}

This definition is just a translation of the notion of a $\sV_\bpt$-modification between $\sV_\bpt$-pseudo\-transformations. We will only use the explicit description just given. 

These notions assemble to form various 2-categories of interest to us.

\begin{notn}\label{CatCAlg}
For a given $\sV_\bpt$-2-category $\cC$ with a subcategory $\cB\subset \cC$, we define the following 2-categories, with $\cC$-transformations as the 2-cells in all cases. 
\begin{itemize}
\item $\CAlg$ of $\cC$-algebras and strict $\cC$-maps;
\item $\CBAlg$ of $\cC$-algebras and  $(\cC,\cB)$-pseudo\-morphisms;
\item $\CPsAlg$ of weak $\cC$-pseudoalgebras and  weak $\cC$-pseudo\-morphisms.
\end{itemize} 
\end{notn}

We have inclusions
\[ \CAlg \subset \CBAlg \subset \CPsAlg.\]

This article largely focuses on the 2-categories $\CBAlg$
in the case that $\cC$ is a category of operators. As noted in Notations \ref{defn:DAlg} and \ref{defn:DGAlg}, we will then fix the relevant $\cB$ and drop the 
subscript $\cB$ from the notation.
However, pulling back along the section $\oursectionG$ in \autoref{sec:section} will land in a category of type $\CPsAlg$, and the strictification theorem in \autoref{sec:PowerLack} will land in a category of type $\CAlg$.

We end this subsection with two constructions on (pseudo)algebras.

\begin{notn}\label{notn:pullingback}
Given a $\sV_\bpt$-pseudofunctor $\squiggly{\xi\colon \cD}{\,\cC}$ and a $\cC$-pseudoalgebra $\aX$, we define the $\cD$-pseudoalgebra $\xi^* \aX$ as the composite
\[ 
\xymatrix{
\cD \ar@{~>}[r]^-{\xi} & \cC \ar@{~>}[r]^-{\aX} & \CatVp.
}
\]
This construction extends to $\cC$-pseudomorphisms and $\cC$-transformations, giving a 2-functor
\[\xi^* \colon \CPsAlg \rtarr \DPsAlg.\]
If $\xi$ is a strict $\sV_\bpt$-2-functor, this construction restricts to give a 2-functor
\[\xi^* \colon \CAlg \rtarr \cD\text{-}{\bf{Alg}}.\]
\end{notn}

\noindent The 2-functor $\xi^*$ preserves level equivalences of pseudoalgebras as defined in \autoref{defn:leveleqv}.

\begin{defn}\label{PairsPairs} Suppose given 
$\sV_\bpt$-2-categories
$\cC$ and $\cD$.  Given a $\cC$-pseudoalge\-bra $\aX$ and an 
$\cD$-pseudoalgebra  $\aY$, we define their \ourdefn{external smash product}  $\aX \bsma \aY$ to be the 
$\cC\sma\cD$-pseudoalgebra given by the composite
\[
\xymatrix{
\cC \esma \cD \ar@{~>}[r]^-{\aX \esma \aY} & \CatVp \esma \CatVp \ar[r]^-{\sma} & \CatVp.
}
\]
If $\aX$ and $\aY$ are strict algebras, so is $\aX \bsma \aY$.

This construction extends to pseudomorphisms and transformations.
\end{defn}

\subsection{Permutative structures on $\sV_\bpt$-$2$-categories}\label{sec:permutative}
In order to encode multiplicative structures on algebras, we use monoidal structures on $\sV_\bpt$-2-categories, as defined in this section. Even in the case when $\sV=\bf{Set}$, what we present here is not the most general definition of a symmetric monoidal structure on a 2-category (see \cite{ChengGurski, GurskiOsorno}). Here, we present a rather strict notion in which the monoidal product is allowed to be a pseudofunctor, but must strictly satisfy associativity and unitality; while the symmetry is allowed to be a pseudotransformation, it must satisfy the symmetry axiom strictly.

\begin{defn}\label{permv2cat}
 A \ourdefn{strict pseudo-monoidal $\sV_\bpt$-2-category}  consists of a $\sV_\bpt$-2-category $\cC$ together with an object $I$, and a $\sV_\bpt$-pseudofunctor
 \[ \squiggly{\smaD \colon \cC \esma \cC}{\cC}\]
 that is strictly associative and strictly unital with respect to $I$. A \ourdefn{pseudo-permutative $\sV_\bpt$-2-category} is a 
 strict pseudo-monoidal $\sV_\bpt$-2-category together with a $\sV_\bpt$-pseudo\-transformation
\[ \xymatrix{
\cC \esma \cC \ar@{~>}[dr]_{\smaD} \ar[rr]^{t}  &  & \cC\esma \cC \ar@{~>}[dl]^{\smaD}\\
& \cC  \utwocell<\omit>{<0>  \ta} & \\} \]
such that the following three axioms hold. 
\begin{enumerate}[(i)]
\item\label{ta_sym_axiom}
The following pasting diagram is equal to the identity of $\smaD$:
\[ \xymatrixcolsep{4pc}\xymatrix{
\cC \esma \cC \ar@{~>}[dr]_{\smaD} \ar[r]^{t}  & \cC \esma \cC \ar@{~>}[d]_{\smaD} \ar[r]^{t} & \cC\esma \cC \ar@{~>}[dl]^{\smaD}\\
& \cC.  \utwocell<\omit>{<-4>  \ta} \utwocell<\omit>{<4>  \ta}& \\} 
\]
\item\label{ta_hex_axiom}
The 2-cell
\[\xymatrix{
& \cC^{\esma 3} \ar[rr]^-{\id\esma t} &&  \cC^{\esma 3} \ar@{~>}[rd]^-{\smaD\esma\id}\\
\cC^{\esma 3} \ar[ru]^-{t\sma\id} \ar@{~>}[rr]^-{\id\esma\smaD} \ar@{~>}[rd]_-{\smaD\esma\id} && \cC^{\esma 2} \ar[rr]^-{t} \ar@{~>}[rd]_-{\smaD}&& \cC^{\esma 2} \ar@{~>}[ld]^-{\smaD}\\
& \cC^{\esma 2} \ar@{~>}[rr]_-{\smaD}&& \cC \utwocell<\omit>{<0>  \ta}
}
\]
is equal to the 2-cell
\[\xymatrix{
& \cC^{\esma 3} \ar[rr]^{\id\esma t} \ar@{~>}[dd]^-{\smaD\esma\id} \ar@{~>}[rd]_-{\id\esma\smaD} &&  \cC^{\esma 3} \ar@{~>}[rd]^-{\smaD\esma\id} \ar@{~>}[ld]^-{\id\esma\smaD}\\
\cC^{\esma 3} \ar[ru]^-{t\sma \id}  \ar@{~>}[rd]_-{\smaD\esma\id} && \cC^{\esma 2} \ar@{~>}[rd]^-{\smaD} \utwocell<\omit>{<0>  \id \esma\ta}&& \cC^{\esma 2} \ar@{~>}[ld]^-{\smaD}\\
& \cC^{\esma 2} \ar@{~>}[rr]_-{\smaD} \ultwocell<\omit>{<3>  \ta\esma \id \quad }&& \cC.
}
\]
The unlabeled regions in both diagrams commute, the quadrilaterals by the strict associativity of $\smaD$ and the pentagon by the naturality of $t$.
\end{enumerate}
If $\smaD$ is a strict $\sV_\bpt$-2-functor and $\ta$ is a strict $\sV_\bpt$-transformation, we say $(\cC,I,\smaD,\ta)$ is a \ourdefn{permutative $\sV_\bpt$-2-category}.
\end{defn}

\begin{rem}\label{likepermutative} Classically, a permutative category is a symmetric strict monoidal category, strict meaning that the product is strictly associative and unital.
The definition above is similar, just done in the context of the 2-category of $\sV_\bpt$-2-categories, $\sV_\bpt$-pseudofunctors, and $\sV_\bpt$-pseudotransformations. 
Thus ``strict pseudo-monoidal'' here means that $\smaD$ is a strictly associative and unital operation given by a $\sV$-pseudofunctor; it respects composition only up to coherent isomorphisms.
The standard coherence theorem for permutative categories still applies in this case:  for any permutation $\si\in \SI_k$, there exists a unique composite of instances of $\ta$ that fits in the diagram below.
\[
 \xymatrix{
\cC ^{\esma k} \ar@{~>}[dr]_{\smaD^k} \ar[rr]^{t_\si}  &  & \cC^{\esma k} \ar@{~>}[dl]^{\smaD^k}\\
& \cC  \utwocell<\omit>{<0>  \ta_\si} & } 
\]
Here $\smaD^k$ denotes the $k$-ary product induced by iterating $\smaD$, and the map $t_\si$ sends a $k$-tuple $(c_1,\dots,c_k)$ to $(c_{\si^{-1}(1)},\dots,c_{\si^{-1}(k)})$. The 1-cell component of $\ta_\si$ is the (unique) composite of instances of the 1-cell of $\ta$ that reorders 
\[c_1\smaD \cdots \smaD c_k \rtarr c_{\si^{-1}(1)} \smaD \cdots \smaD c_{\si^{-1}(k)}.\]
\end{rem}

\begin{defn}\label{permfunctor}
Let $(\cC,I,\smaD,\ta)$ and $(\cD,I',\smaD',\ta')$ be pseudo-permutative $\sV_\bpt$-2-categories. A \ourdefn{symmetric monoidal pseudofunctor}  $\squiggly{(\Psi,\mu)\colon \cC}{\cD}$ 
consists of a $\sV_\bpt$-pseudofunctor $\squiggly{\Psi\colon \cC}{\cD}$ such that $\Psi(I)=I'$, together with a $\sV_\bpt$-~pseudotransformation
\[ \xymatrix{
\cC \sma \cC \ar@{~>}[r]^{\Psi \sma \Psi} \ar@{~>}[d]_{\smaD} \drtwocell<\omit>{\mu}
& \cD \sma \cD \ar@{~>}[d]^{\smaD'} \\
\cC \ar@{~>}[r]_{\Psi} & \cD
}\]
such that the following axioms hold.
\begin{enumerate}[(i)]
\item $\mu$ is unital, meaning that its restrictions to $\{I\} \sma \cC$ and  $\cC \sma \{I\}$ are the identity transformation,
where $\{I\}\subset \cC$ denotes the discrete $\sV_\bpt$-2-category on the single object $I$.\footnote{The $\sV_\bpt$-object of morphisms in $\{ I \}$ is  $\bV\{0_I,\id_I\}\cong * \amalg *$. Thus, $\{ I\} \sma \cC \cong \cC$.} 
\item $\mu$ is associative, meaning that  
\[ \xymatrixcolsep{3.5pc}\xymatrix{
\cC \sma \cC \sma \cC \ar@{~>}[r]^{\Psi \sma \Psi \sma \Psi} \ar@{~>}[d]_{\smaD \sma \id} \drtwocell<\omit>{\quad \mu \sma \id}
& \cD \sma \cD \sma \cD \ar@{~>}[d]^{\smaD' \sma \id}\\
\cC \sma \cC \ar@{~>}[r]^{\Psi \sma \Psi} \ar@{~>}[d]_{\smaD} \drtwocell<\omit>{\mu}
& \cD \sma \cD \ar@{~>}[d]^{\smaD'} \\
\cC \ar@{~>}[r]_{\Psi} & \cD
}
\ \ \ \ 
\raisebox{-9ex}{=}
\ \ \ \ 
\xymatrixcolsep{3.5pc}\xymatrix{
\cC \sma \cC \sma \cC \ar@{~>}[r]^{\Psi \sma \Psi \sma \Psi} \ar@{~>}[d]_{\id \sma \smaD} \drtwocell<\omit>{\quad \id \sma \mu}
& \cD \sma \cD \sma \cD \ar@{~>}[d]^{\id \sma \smaD'}\\
\cC \sma \cC \ar@{~>}[r]^{\Psi \sma \Psi} \ar@{~>}[d]_{\smaD} \drtwocell<\omit>{\mu}
& \cD \sma \cD \ar@{~>}[d]^{\smaD'} \\
\cC \ar@{~>}[r]_{\Psi} & \cD.
}
\]
The vertical boundaries of these diagrams are equal because of the associativity of $\smaD$ and $\smaD'$.
\item The following equality of pasting diagrams holds:
\[ \xymatrix @R=1.8em{
\cC \sma \cC \ar@{~>}[r]^{\Psi\sma \Psi} \ar@{~>}[dd]_{\smaD} \ddrtwocell<\omit>{\mu}
& \cD \sma \cD \ar@{~>}[dd]^{\smaD'} \ar[dr]^{t} \ddtwocell<\omit>{<-5>\tau'} \\
 & & \cD \sma \cD \ar@{~>}[dl]^{\smaD'} \\
\cC \ar@{~>}[r]_{\Psi} & \cD & 
} 
\ \ \ \ 
\raisebox{-3.2em}{=}
\ \ \ \ 
\xymatrix @R=1.8em{
\cC \sma \cC \ar@{~>}[dd]_{\smaD} \ar[dr]^{t} \ddtwocell<\omit>{<-4>\tau} \ar@{~>}[r]^{\Psi \sma \Psi}  & \cD \sma \cD \ar[dr]^{t} & \\
 & \cC \sma \cC \ar@{~>}[dl]^{\smaD} \ar@{~>}[r]^{\Psi \sma \Psi} \dtwocell<\omit>{\mu} & \cD \sma \cD \ar@{~>}[dl]^{\smaD'} \\
 \cC \ar@{~>}[r]_{\Psi} & \cD. & 
}
\]
\end{enumerate}
If $\Psi$ is a strict $\sV_\bpt$-2-functor, we refer to $(\Psi,\mu)$ as a symmetric monoidal 2-functor.
\end{defn}

 \section{The multicategory of $\sD$-algebras and the multifunctor $\bR$}\label{sec:DAlg}

 Nonequivariantly, categories of operators were introduced on the space level in order to mediate the passage from algebras over an $E_\infty$-operad in spaces to special $\sF$-spaces when comparing the operadic and Segalic infinite loop space machines \cite{MT}. Algebras over a category of operators are a generalization of both $\sF$-spaces (aka $\Gamma$-spaces) and algebras over an $E_\infty$-operad. 
 In this section we discuss their categorical analogues. Equivariant categories of operators were studied in \cite{Sant, MMO}, and their categorical analogues will be introduced in \autoref{sec:DG}. 

We show that for a category of operators $\sD$ with pseudo-commutative structure, as defined in \autoref{sec:PseudocommCO}, there is a multicategory of algebras over $\sD$. For categories of operators coming from operads, the necessary structure arises from a pseudo-commutative structure on the operad, as in \autoref{pseudocom}. We make all of this precise in this section.

\subsection{Categories of operators over $\sF$}\label{Dalgs} 
The definitions in this subsection are categorical analogues of definitions in \cite{MT}. We give the definitions in the setting of \VtwoCats.

\begin{defn}\label{F} Recall that $\sF$ denotes the category of based sets $\bn =\{0,1,\dots,n\}$ with basepoint $0$ and let $\PI$ denote its subcategory of morphisms $\ph\colon \bm \rtarr \bn$ such that $|\ph^{-1}(j)| = 0$ or $1$ for $1\leq j\leq n$.  We often use the abbreviated notation  $\ph_j  = |\ph^{-1}(j)|$. We regard $\sF$ and $\PI$ as discrete $2$-categories, meaning that they have only identity $2$-cells.   Via \autoref{bUbV}, we then regard them as 
$\sV_*$-$2$-categories.
\end{defn}

\begin{defn}\label{GCO/F} 
A \ourdefn{$\VCat$-category of operators} $\sD$ over $\sF$, abbreviated 
$\VCat$-$\mathbf{CO}$ over $\sF$, is a  $\sV$-$2$-category whose objects are the based sets $\bn$ for $n\geq 0$ together with $\sV$-$2$-functors 
\[\xymatrix@1{ \Pi \ar[r]^-{\io} & \sD \ar[r]^-{\xi} & \sF\\} \] 
such that $\io$ and $\xi$ are the identity on objects and $\xi\com \io$ is the inclusion.  
A morphism 
$\nu\colon \sD\rtarr \sE$ of $\VCat$-$\mathbf{CO}$s over $\sF$ is a $\sV$-2-functor  over $\sF$ and under $\PI$.
\end{defn} 

\begin{defn}\label{reducedGCO/F}
A $\VCat$-$\mathbf{CO}$ $\sD$ over $\sF$ is \ourdefn{reduced} if $\bzo$ is a zero object, and  we then say that $\sD$
is a \ourdefn{$\VCatp$-category of operators} over $\sF$. We shall restrict attention to $\VCatp$-categories of operators over $\sF$.
\end{defn}

\begin{rem}\label{reducedCO/Frem}  By Propositions \ref{ZeroObjEnrCatp} and \ref{ReducedVCatpEnr}, if  $\sD$ is a $\VCatp$-$\mathbf{CO}$ over $\sF$, then 
$\sD$ is a $\sV_\ast$-$2$-category and $\io$ and $\xi$ are  $\sV_\ast$-$2$-functors; that is, $\sD$ is a $\sV_\ast$-$2$-category  over $\sF$ and under $\Pi$.    
A morphism $\sD\rtarr \sE$ of reduced $\VCat$-$\mathbf{CO}$s over $\sF$ is necessarily reduced since it must send $\bzo$ to $\bzo$; thus it is a $\sV_\ast$-2-functor over $\sF$ and under $\PI$.
\end{rem}

Let $\oO$ be an operad in $\VCat$. We can associate to it a category of operators $\sD=\sD(\oO)$ over $\sF$  by letting
\[ \sD(\bm,\bn) = \coprod_{\ph\in \sF(\bm,\bn)}\ \prod_{1\leq j\leq n} \oO( \ph_j). \]
\noindent 
Composition is induced from the structural maps $\ga$ of $\oO$. To write formulas instead of diagrams, we use elementwise  notation, writing $c_i\in \oO(\phi_j)$ for objects and morphisms in $\oO(\phi_j)$.  For $(\ph,c_1,\dots, c_n)\colon \bm\rtarr \bn$ and $(\ps,d_1,\dots, d_p)\colon \bn\rtarr \bp$, define
\begin{equation}\label{Dcomposition} 
(\ps,d_1,\dots, d_p)\com (\ph,c_1,\dots,c_n) = \Big(\ps\com \ph, 
\prod_{1\leq j\leq p}\ga(d_j;\prod_{\ps(i) = j} c_i)\rho_j(\ps,\ph)\Big). 
\end{equation}
 The $c_i$ with $\ps(i) =j$  are ordered by the natural order on their indices $i$,
and $\rho_j(\ps, \ph)$ is that permutation of $|(\ps\com\ph)^{-1}(j)|$ letters which converts 
the natural ordering of $(\ps\com\ph)^{-1}(j)$ as a subset of $\{1,\dots,m\}$ to
its ordering obtained by regarding it as $\coprod_{\ps(i)=j}\ph^{-1}(i)$, so ordered
that elements of $\ph^{-1}(i)$ precede elements of $\ph^{-1}(i')$ if $i< i'$ and 
each $\ph^{-1}(i)$ has its natural ordering as a subset of $\{1,\dots,m\}$.  When it is clear which $\phi$ and  $\psi$  are being composed, we abbreviate the notation for the permutation $\rho_j(\ps, \ph)$ to $\rho_j$.

\begin{prop}\cite[Construction 4.1]{MT}\label{DODefn} The above specification makes $\sD(\oO)$ into a category of operators over $\sF$, and it is reduced if $\oO$ is reduced.
\end{prop}

\begin{proof}
The map $\xi\colon \sD \rtarr \sF$
sends $(\ph,c_1,\dots, c_n)$ to $\ph$.   
Recall that any morphism $\phi$ in $\PI$ satisfies $\phi_j \leq 1$ for $j > 0$.
The inclusion $\io\colon \PI\rtarr \sD$ sends 
$\ph\colon \bm\rtarr \bn$ to $(\ph, c_1,\dots, c_n)$, where $c_j = \oid \in \oO(1)$ if $\ph_j= 1$ and
$c_j = \ast\in \oO(0)$ if $\ph_j= 0$. 
\end{proof}

\subsection{Pseudo-commutative categories of operators over $\sF$}\label{sec:PseudocommCO}

In analogy with our definition of pseudo-commutativity of an operad, we define a compatible notion of
pseudo-commutativity of a category of operators $\sD$ over $\sF$.  
The categories $\PI$ and $\sF$ are permutative 
under the smash product of finite based sets, as we now recall. 
On objects, $\bm\sma \bp$ is defined to be $\bm\bp$. Given $\phi\in \sF(\bm,\bn)$ and $\psi\in \sF(\bp,\bq)$ 
their smash product $\phi \sma \psi \in \sF(\bm\bp,\bn\bq)$ is defined, in parallel to \autoref{otimes}, as the restriction of the map
\[  \phi \times \psi \colon \bm\times \bp \rtarr \bn\times \bq \]
along the (based) inclusion $\bm\bp \into \bm \times \bp$ that is 
given by \autoref{lexico} away from the basepoint.
The symmetry 
isomorphisms $\ta$ are given by 
the permutations $\tau_{m,p}$ of \autoref{tau} which reorder the sets  $\bm \bp$ 
from lexicographic to reversed lexicographic ordering.
The inclusion of $\SI$ in $\sF$ identifies $\spair$ with $\sma$. We will continue to use the symbol $\spair$ for emphasis when dealing with permutations.

Recall the notion of pseudo-permutative $\sV_\bpt$-2-category from \autoref{permv2cat}.

\begin{defn}\label{pairprod}
A \ourdefn{pseudo-commutative structure} 
 on $\sD$ is a pseudo-permutative structure $(\sD,\mathbf{1},\smaD,\ta)$ such that
\begin{enumerate}
 \item\label{smaDrestricts} $\smaD$ restricts  to $\sma$ on 
$\PI\sma \PI$ and projects to $\sma$ on $\sF$ 
(in the sense that $\xi \circ \smaD = \sma \circ(\xi\sma \xi)$); 
\item\label{PPDDcond} $\smaD$ restricts to a strict $\sV_\bpt$-2-functor on $\PI\esma\sD$ and $\sD\esma\PI$;
\item\label{tarestricts} $\ta$ restricts to the symmetry on $\PI$ given in \autoref{tau}.
\end{enumerate}
\end{defn}

We identify the pieces of this definition explicitly.  First note that condition \eqref{smaDrestricts}
implies in particular that $\smaD = \sma$ on objects. 
The fact that $\smaD$ is a $\sV_*$-pseudofunctor means that there is a collection of invertible  $\sV_{\bpt}$-transformations
\begin{equation}\label{beta0}
\xymatrix{
\sD(\bn, \bp)  \sma \sD(\br, \bs) \sma \sD(\bm, \bn) \sma  \sD(\bq, \br) 
\ar[rr]^-{\smaD \sma \smaD} \ar[d]_{\id\sma t\sma  \id}^{\iso}   \ddrrtwocell<\omit>{<0>\     \vartheta} & &
\sD(\bn \sma \br, \bp \sma \bs) \sma \sD(\bm \sma \bq, \bn \sma \br )\ar[dd]^\circ\\
\sD(\bn,\bp )\sma  \sD(\bm, \bn)\sma  \sD(\br, \bs) \sma \sD(\bq, \br) 
\ar[d]_{\circ \sma  \circ}  & & \\
\sD(\bm,\bp) \sma \sD(\bq,\bs) \ar[rr]_-{\smaD}  & & \sD(\bm \sma \bq,\bp\sma \bs ) 
}
\end{equation}
relating  $\smaD$ to composition.

Condition \eqref{PPDDcond}, which is necessary for \autoref{multithm2}, translates to requiring 
$\vartheta$ to be the identity (so that the diagram commutes) when  either both $\sD(\bn,\bp)$ and $\sD(\br,\bs)$  are restricted to $\PI$ or both $\sD(\bm, \bn)$ and $\sD(\bq,\br)$ are restricted to $\PI$. 

Writing this out elementwise, on 1-cells it means that, whenever the composites are defined, 
\[   (c\smaD d)\circ(a \sma b) = (c\circ a)\smaD (d\circ b) \]
and
\[   (a \sma b)\circ(c\smaD d) = (a\circ c)\smaD (b\circ d) \] 
where $c$ and $d$ are morphisms of $\sD$ and $a$ and $b$ are morphisms of $\PI$.  We think of this as saying that the monoidal structure on $\smaD$ is \ourdefn{strict relative  to $\PI$}. 
As the pseudofunctoriality constraint for $\smaD$, the $\sV_{\bpt}$-transformations $\vartheta$ must satisfy a condition with respect to triple composition. The condition on $\smaD$ being strictly associative imposes another set of conditions on 
$\vartheta$.

Condition \eqref{tarestricts} means that
  the 1-cell constraint of $\ta$ at the object $(\bm,\bp)$ is the permutation $\tau_{m,p}$ thought of as a morphism in $\sD$, and 
the pseudonaturality constraint is an invertible  $\sV_{\bpt}$-transformation 
\begin{equation}\label{symmetry}
\xymatrix{
\sD(\bm,\bn) \sma \sD(\bp,\bq)  \ar[r]^-{\opair\circ t} \ar[d]_-{\opair} \drtwocell<\omit>{<0>  \hat{\ta}} & \sD(\bp\bm, \bq\bn) \ar[d]^{(\tau_{m,p})^*}\\
\sD(\bm \bp, \bn\bq)  \ar[r]_{(\tau_{n,q})_*} &   \sD(\bm \bp, \bq\bn)
}
 \end{equation}
that is the identity when restricted to the subcategory $\PI(\bm,\bn) \sma \PI(\bp,\bq)$. 
These pseudonaturality constraints must be compatible with composition in $\sD$ and the pseudofunctoriality constraint $\vartheta$.

\begin{rem}\label{condstarcomp}  
For later use, we emphasize  a particular consequence of the strictness relative to $\PI$ here.
Let $a\in \PI(\bm',\bm)$ and $b\in \PI(\bq',\bq)$. The compatibility of $\vartheta$ with triple composition together with condition \eqref{PPDDcond} implies that
\[
\xymatrix{
\sD(\bn, \bp)  \sma \sD(\br, \bs) \sma \sD(\bm, \bn) \sma  \sD(\bq, \br) 
\ar[rr]^-{\smaD \sma \smaD} \ar[d]_{\id\sma \id \sma a^* \! \sma b^*}   & &
\sD(\bn \sma \br, \bp \sma \bs) \sma \sD(\bm \sma \bq, \bn \sma \br ) \ar[d]^{\id \sma (a\sma b)^*}\\
\sD(\bn, \bp)  \sma \sD(\br, \bs) \sma \sD(\bm', \bn) \sma  \sD(\bq', \br) 
\ar[rr]^-{\smaD \sma \smaD} \ar[d]_{\id\sma t\sma  \id}^{\iso}   \ddrrtwocell<\omit>{<0>\     \vartheta} & &
\sD(\bn \sma \br, \bp \sma \bs) \sma \sD(\bm' \sma \bq', \bn \sma \br )\ar[dd]^\circ\\
\sD(\bn,\bp )\sma  \sD(\bm', \bn)\sma  \sD(\br, \bs) \sma \sD(\bq', \br) 
\ar[d]_{\circ \sma  \circ}  & & \\
\sD(\bm',\bp) \sma \sD(\bq',\bs) \ar[rr]_-{\smaD}  & & \sD(\bm' \sma \bq',\bp\sma \bs ) 
}
\]
is equal to
\[
\xymatrix{
\sD(\bn, \bp)  \sma \sD(\br, \bs) \sma \sD(\bm, \bn) \sma  \sD(\bq, \br) 
\ar[rr]^-{\smaD \sma \smaD} \ar[d]_{\id\sma t\sma  \id}^{\iso}   \ddrrtwocell<\omit>{<0>\     \vartheta} & &
\sD(\bn \sma \br, \bp \sma \bs) \sma \sD(\bm \sma \bq, \bn \sma \br )\ar[dd]^\circ\\
\sD(\bn,\bp )\sma  \sD(\bm, \bn)\sma  \sD(\br, \bs) \sma \sD(\bq, \br) 
\ar[d]_{\circ \sma  \circ}  & & \\
\sD(\bm,\bp) \sma \sD(\bq,\bs) \ar[rr]_-{\smaD} \ar[d]_{a^* \!\sma b^*}  & & \sD(\bm \sma \bq,\bp\sma \bs ) \ar[d]^{(a\sma b)^*} \\
\sD(\bm',\bp) \sma \sD(\bq',\bs) \ar[rr]_-{\smaD}  & & \sD(\bm' \sma \bq',\bp\sma \bs )
}\]
The two unlabeled squares have instances of $\vartheta$ that are the identity because of condition \eqref{PPDDcond}. The boundaries on both diagrams are equal since composition in $\sD$ is strict. This equality expresses a condition on $\vartheta$ for when the first terms of a triple composition come from $\PI$. There are similar conditions for when the middle and the last terms come from $\PI$.
\end{rem}

\begin{defn}\label{pseudocommap}
We define a map $(\Psi,\mu)\colon \sD \rtarr \sE$ of pseudo-commutative categories of operators to be a symmetric monoidal 2-functor (\autoref{permfunctor}) such that $\Psi$ is a map of $\VCatp$-$\mathbf{CO}$s over $\sF$ and the restriction of $\mu$ to the subcategory $\PI\sma \PI$ is the identity transformation.
 \end{defn}

We defer the proof  of the following theorem  to \autoref{Rpf}.   It ensures that our definitions of pseudo-commutativity for operads and for their associated categories of operators are compatible.  The verification is essentially combinatorial bookkeeping and is painstaking rather than hard. 

\begin{thm}\label{multithm} Let $\oO$ be a pseudo-commutative operad in $\VCat$. Then $\sD= \sD(\oO)$ is a pseudo-commutative category of operators.  
\end{thm}

\begin{rem}  The construction is functorial.  With the appropriate definition of a pseudo-commutative morphism $\oO\rtarr \sP$  of pseudo-commutative operads,  
the map $\sD(\oO)\rtarr \sD(\sP)$ is pseudo-commutative.  In analogy with \autoref{ChaoticPseudoCom}, when $\oO$ and $\sP$ are chaotic, any morphism of operads 
between them is necessarily pseudo-commutative.
\end{rem}

\subsection{The multicategory of $\sD$-algebras}\label{Multiop1}\label{MdpaDef} 
For a category of operators $\sD$ over $\sF$, we consider $\sD$-algebras as defined in \autoref{strict}.
As indicated in \autoref{NoPseudoAlgs}, there is a more general notion of $\sD$-pseudoalgebra, but 
we defer discussion of that to \cite{AddCat2}. Recall that we have the notions of $(\sD,\Pi)$-pseudomorphism and $\sD$-transformation from Definitions \ref{Cpseudo} and \ref{Ctran}.

\begin{notn}\label{defn:DAlg} For a category of operators $\sD$ over $\sF$,  we denote by 
$\DAlg$ the 2-category of strict $\sD$-algebras, $(\sD,\Pi)$-pseudo\-morphisms, and $\sD$-transformations.  This 2-category was denoted by ${\DAlg}_\PI$ in \autoref{CatCAlg}, but we now fix $\cB =\PI$ and drop it from the notation.
\end{notn} 

This notion, with its strictness with respect to $\PI$, is {\em essential} for the construction of $\bP$ in the left column of \autoref{RoadMap}, as we explain in \autoref{keyG}.  

Let $\sD$ be a reduced pseudo-commutative category of operators over $\sF$.  
We define the multicategory $\Mult(\sD)$ of $\sD$-algebras, which amounts to 
defining the $k$-ary morphisms. As said before, we set it up to have its objects be $\sD$-algebras, although with only slightly more work we could equally well 
have set it up to have its objects be $\sD$-pseudoalgebras.

Recall from \autoref{sec:Cpseudo} that a $\sD$-algebra is given by a $\sV_\bpt$-2-functor $\aX \colon \sD \rtarr \CatVp$, which can be expressed in adjoint form 
as in \autoref{strict}.  Thus the action of $\sD$ on $\aX$ is given by $\sV_\bpt$-functors
\[ \tha\colon \sD(\bm,\bn)\sma \aX(\bm) \rtarr \aX(\bn). \] 

Let $\sD^{\esma k}$ denote the $k$-fold smash power. Following \autoref{PairsPairs}, given $\sD$-algebras $\aX_1, \dots, \aX_k$, we have the external smash product $\aX_1 \overline{\sma} \dots \overline{\sma} \aX_k$.  It sends an object $(\bn_1,\dots,\bn_k)$ of $\sD^{\esma k}$ to $\aX_1(\bn_1)\sma \cdots \sma \aX_k(\bn_k)$, with action map $\tha^k$ given by the composite
 \[
 \xymatrix{\bigwedge\limits_i \sD(\bm_i,\bn_i)\sma \bigwedge\limits_i \aX(\bm_i) \ar[r]^-{t}_-{\cong} & \bigwedge\limits_i \sD(\bm_i,\bn_i)\sma \aX(\bm_i) \ar[r]^-{\bigwedge \tha} & \bigwedge\limits_i \aX(\bn_i),
 }
 \]
 where the first map is the appropriate shuffle. For a $\sD$-algebra $\aY$, we consider the $\sD^{\esma k}$-pseudoalgebra 
  \[
\xymatrix{
 \sD^{\esma k} \ar@{~>}[r]^-{\smaD^k} & \sD \ar[r]^-{\aY} & \CatVp.
 }
 \]
 The conditions in \autoref{pairprod} imply that $\aY \circ \smaD^k$ restricts to a strict $\PI^{\esma k}$-algebra. Since $\PI^{\esma k}$ is discrete, this is a 
functor from $\PI^{\sma k}$ to the underlying $1$-category of $\VCatp$. 

\begin{defn}\label{MultiD}
Let $\sD$ be a reduced pseudo-commutative category of operators over $\sF$. 
 We define a (symmetric) multicategory $\Mult(\sD)$ of $\sD$-algebras as follows. The objects are $\sD$-algebras. 
For objects $\aX_i$, $1\leq i\leq k$, and $\aY$, a $k$-ary morphism $\ul{\aX}\rtarr \aY$ consists of a $(\sD^{\esma k},\Pi^{\esma k})$-pseudomorphism
\[F\colon \squiggly{\aX_1 \overline{\sma} \dots \overline{\sma} \aX_k }{\aY \circ \smaD^k}.\]
Recall that this is the same as saying that $F$ is a $\sV_\bpt$-pseudotransformation
\[\xymatrixcolsep{4pc}\xymatrix{
\sD^{\sma k} \ar[r]^-{\aX_1 \sma \cdots \sma \aX_k} \ar@{~>}[d]_{\smaD^k} \drtwocell<\omit>{<0>\quad   F} & \CatVp ^{\sma k}\ar[d]^{\sma^k}\\
\sD \ar[r]_-{\aY} & \CatVp
}
\]
that is strict when restricted to $\Pi^{\esma k}$.

Given a $j_i$-ary morphism $E_i\colon (\aX_{i,1},\dots,\aX_{i,j_i}) \rtarr \aY_i$ for $i=1,\dots k$, and a $k$-ary morphism $F\colon (\aY_1,\dots,\aY_k) \rtarr \aZ$, the composite is defined by the pasting diagram below, where the right hand 2-cell is the associativity isomorphism for $\sma$ on $\CatVp$.
\begin{equation}\label{compMultiD}
 \xymatrix @C=4em{
 \sD^{\sma j} \ar[rrr]^-{\aX_{1,1} \sma \cdots \sma \aX_{k,j_k}} \ar@{~>}[dd]_{\smaD^j}  \ar@{~>}[dr]^{\bigwedge\limits _i \smaD^{j_i}}    & \drtwocell<\omit>{<0>\qquad  \quad  E_1 \sma \cdots \sma E_k} & \ddtwocell<\omit>{<-12> } & \CatVp^{\esma j}  \ar[dl]_{\hspace{2em}\bigwedge\limits _i \sma^{j_i}} \ar[dd]^{\sma^j}  \\
 & \sD^{\esma k}  \ar[r]^-{\aY_1\sma \cdots \sma \aY_k} \drtwocell<\omit>{F} \ar@{~>}[dl]^{\smaD^k} & \CatVp^{\esma k} \ar[dr]_{\sma ^k} & \\
 \sD \ar[rrr]_\aZ & & & \CatVp.
 }
\end{equation}

Finally, we specify the symmetric structure on the multicategory $\Mult(\sD)$.
Given a permutation $\sigma \in \SI_k$ and a $k$-ary morphism $F\colon (\aX_1 , \dots ,\aX_k) \rtarr \aY$, the 
$k$-ary morphism \mbox{$F\sigma \colon (\aX_{\si(1)},\dots, \aX_{\si(k)}) \rtarr \aY$} is defined by the pasting diagram
\[\scalebox{0.95}{
 \xymatrix @C=4em{
 \sD^{\esma k} \ar@{~>}[dd]_{\smaD^k} \ar[dr]^{t_\si} \ddtwocell<\omit>{<-4>\tau_\si^{-1}} \ar[rrr]^{\aX_{\si(1)} \sma \cdots \sma \aX_{\si(k)}} & & \ddtwocell<\omit>{<-12> t_{\si }} & \CatVp^{\esma k}  \ar[dl]^{t_\si} \ar[dd]^{\sma^k}  \\
 & \sD^{\esma k}  \ar[r]^-{\aX_1\sma \cdots \sma \aX_k} \drtwocell<\omit>{F} \ar@{~>}[dl]^{\smaD^k} & \CatVp^{\esma k} \ar[dr]_{\sma ^k} & \\
 \sD \ar[rrr]_\aY & & & \CatVp.
}}\]
Here the different maps called $t_\si$ send a $k$-tuple $(a_1,\dots,a_k)$ to $(a_{\si^{-1}(1)},\dots,a_{\si^{-1}(k)})$, and $\ta_\si$ is the invertible $\sV_\bpt$-pseudotransformation of \autoref{likepermutative}.
\end{defn}

We now unpack this definition. In what follows, given a $k$-tuple $(n_1,\dots,n_k)$ of natural numbers, we write $n = n_1\cdot \dots \cdot n_k$. 

A $k$-ary morphism
$F=(F,\de)\colon (\aX_1,\dots,\aX_k) \rtarr \aY$ consists of $\sV_\bpt$-functors 
\[ F\colon \aX_1(\bn_1) \sma \cdots \sma \aX_k(\bn_k) \rtarr \aY(\bn), \] 
together with invertible $\sV_\bpt$-transformations $\de$ in the following diagrams, in which $1\leq i\leq k$.
\begin{equation}\label{delta}
\xymatrix @R=5ex{
*-<1ex>{\bigwedge\limits_{i} \sD(\bm_i,\bn_i) \sma \bigwedge\limits_i \aX_i(\bm_i)} 
\ar[]!<13ex,0pt>;[rr(0.47)]^-{\id\sma F} 
\ar[]!<0ex,-1ex>;[d(0.75)]_(.4){t}^(0.4)\iso
\ddrrtwocell<\omit>{<0>\quad  \  \delta} & & 
 *- {\bigwedge\limits_{i} \sD(\bm_i, \bn_i)\sma  \aY(\bm)} \ar[d]^(0.35){\smaD^k\sma \id} \\
*-{\raisebox{-3ex}{$\bigwedge\limits_{i} \sD(\bm_i,\bn_i) \sma \aX_i(\bm_i)$} }
\ar[d(0.7)]_(0.65){\bigwedge\limits_i \tha}  &  &  
\raisebox{1pt}{$\sD(\bm,\bn)\sma \aY(\bm)\ar[d(0.8)]^(0.5){\tha}$}  
\\
*+{\bigwedge\limits_{i} \aX_i(\bn_i)} \ar[rr(0.9)]_{F}  & & *+{\raisebox{4pt}{$\aY(\bn)$}}. 
\\}
\end{equation} %
We  require $\de$ to be the identity when restricted to $\PI^{\esma k}$.
The $\de$ must satisfy coherence diagrams related to composition and identities in $\sD^{\esma k}$. The latter are subsumed in the conditions on $\Pi$. We defer writing out the details of the required conditions for composition to \autoref{cohMultD}.

Unpacking the action of $\si$,  the component of $(F,\de)\si=(F\si,\de\si)$ at an object $\bm_1,\dots, \bm_k$ is defined by the following commutative diagram.

\begin{equation}\label{ObjSym}
 \xymatrix{
\aX_{\si(1)}(\bm_1) \sma \cdots \sma \aX_{\si(k)}(\bm_{k}) \ar[r]^-{F\si} \ar[d]_{t_\si} &  \aY(\bm) \\
\aX_1(\bm_{\si^{-1}(1)}) \sma \cdots \sma \aX_k(\bm_{\si^{-1}(k)}) \ar[r]_-{F} & \aY(\bm) \ar[u]_{\aY(\ta_{\si}^{-1})} \\} 
\end{equation} 

The invertible $\sV_\bpt$-transformation $\de\si$ is obtained by whiskering the $\de$ of \autoref{delta},
but using the pseudocommutativity of $\sD$.  Precisely, we construct $\delta{\si}$ by the following pasting diagram,
where we write $\siginv$ for $\sigma^{-1}$.
Here the inner hexagon is \autoref{delta}, and the outer hexagon is the corresponding diagram for $F\si$.
On $\sD(\bm,\bn)$ we denote by $c_\si$ the
pre- and postcomposition with 
\[\tau_\si\colon \bm=\bigwedge_i \bm_i  \rtarr \bigwedge_i \bm_{\siginv(i)} \qquad \text{and} \qquad
\tau_\si^{-1}\colon \bigwedge_i \bn_{\siginv(i)} \rtarr \bigwedge_i \bn_i = \bn,\]
 respectively.

\begin{equation}\label{schematic}  
\scalebox{0.6}{
\xymatrix{  
 \bigwedge\limits_{i} \sD(\bm_{i},\bn_{i}) \sma \bigwedge\limits_i \aX_{\si(i)}(\bm_{i}) \ar[rrrr]^-{\id\sma F\si}  \ar[dr]^-{t_\si\sma t_\si} \ar[dd]^{t} 
 & & &  
  \ddrtwocell<\omit>{<4>  \quad \  \hat{\ta}\sma \id } &  \bigwedge\limits_{i} \sD(\bm_i, \bn_i)\sma  \aY(\bm) \ar[dd]^{\smaD\sma\id} \\
 & \bigwedge\limits_{i}  \sD(\bm_{\siginv(i)},\bn_{\siginv(i)}) \sma \bigwedge\limits_i \aX_i(\bm_{\siginv(i)}) 
  \ar[d]_{t} \ar[rr]^-{\id\sma F}   \ddrrtwocell<\omit>{<0> \quad \  \delta}  
  & & \bigwedge\limits_{i} \sD(\bm_{\siginv(i)}, \bn_{\siginv(i)})\sma  \aY(\bigwedge\limits_i \bm_{\siginv(i)}) \ar[d]^{\smaD\sma \id}
   \ar[ur]^{t_\si^{-1}\sma \aY(\ta_{\si}^{-1})} & \\ 
 \bigwedge\limits_{i} \sD(\bm_{i},\bn_{i}) \sma \aX_{\si(i)}(\bm_i) \ar[dd]_{\bigwedge\limits_i \tha}   \ar[r]^-{t_\si} 
   & \bigwedge\limits_{i} \sD(\bm_{\siginv(i)},\bn_{\sigma^{-1}(i)}) \sma \aX_i(\bm_{\siginv(i)})    \ar[d]_{\bigwedge\limits_i \tha}  
& &    \sD(\bigwedge\limits_i \bm_{\siginv(i)},\bigwedge\limits_i \bn_{\siginv(i)})\sma  \aY(\bigwedge\limits_i \bm_{\siginv(i)}) \ar[r]^-{c_{\si}\sma \aY(\ta_{\si}^{-1})}  \ar[d]^{\tha}  &  \sD(\bm,\bn)\sma \aY(\bm)  \ar[dd]^{\tha}\\
& \bigwedge\limits_{i} \aX_i(\bn_{\siginv(i)}) \ar[rr]_-{F}  &  &  \aY(\bigwedge\limits_i \bn_{\siginv(i)}) \ar[dr]_{\aY(\ta_{\si}^{-1})} & \\
\bigwedge\limits_{i} \aX_{\si(i)}(\bn_{i}) \ar[rrrr]_-{F\si} \ar[ur]_-{t_\si}  & & & &  \aY(\bn) \\ }}
\end{equation}

The top and bottom trapezoids commute by the definition of $F{\si}$.  The left two trapezoids commute trivially.  The bottom right trapezoid commutes 
since $\aY \in \Mult(\sD)$ is a (strict) $\sD$-algebra.
The top right trapezoid is filled by an invertible $\sV_\bpt$-transformation $\hat{\ta}$ given by the pseudonaturality constraint of the appropriate (unique) composition of instances of the pseudocommutativity of $\sD$. 

 Recall \autoref{notn:pullingback}.

\begin{thm}\label{pullingback}
 Let $(\Psi,\mu) \colon \sD \rtarr \sE$ be a map of pseudo-commutative categories of operators (\autoref{pseudocommap}). Then pulling back along $\Psi$ induces a (symmetric) multifunctor
 \[\Psi^* \colon \Mult(\sE) \rtarr \Mult(\sD).\] 
 \end{thm}
 
 \begin{proof}
Given a $k$-ary morphism $F\colon (\aX_1,\dots,\aX_k) \rtarr \aY$ in $\Mult(\sE)$, the multimorphism $\Psi^*(F)$ is defined as the pasting
 \[\xymatrixcolsep{4pc}\xymatrix{
\sE^{\sma k} \ar[r]^-{\Psi^{\esma k}} \ar@{~>}[d]_{\smaD^k} \drtwocell<\omit>{<0>\quad   \mu_k}  & \sD^{\sma k} \ar[r]^-{\aX_1 \sma \cdots \sma \aX_k} \ar@{~>}[d]_{\smaD^k} \drtwocell<\omit>{<0>\quad   F} & \CatVp ^{\sma k}\ar[d]^{\sma^k}\\
\sE\ar[r]_{\Psi} & \sD \ar[r]_-{\aY} & \CatVp,
}
\]
 where the 2-cell $\mu_k$ denotes an appropriate composite of instances of $\mu$, which is unique by the associativity of $\mu$. Compatibility with the identity, composition, and the symmetric group action follows from the axioms in \autoref{pseudocommap}.
\end{proof}

\subsection{Definition of the functor $\bR$}\label{Rdef} 
Let $\oO$ be a reduced operad in $\VCat$ with associated $\VCatp$-$\mathbf{CO}$ $\sD$ over $\sF$. We define a $2$-functor $\bR\colon \OAlg \rtarr \DAlg$ with the property that for an $\oO$-algebra $\aA$, the resulting $\sD$-algebra is defined on objects by $\bn \mapsto \aA^n$, and we show in \autoref{Multiop2} that $\bR$ extends to a multifunctor.

We have the $2$-category $\PAlg$ of  
$\PI$-algebras, $\PI$-morphisms, and $\PI$-trans\-formations, and we have the evident $2$-functor $\bR\colon \VCatp\rtarr \PAlg$ that sends a $\sV_\bpt$-category $\aA$ to the $\PI$-algebra $\PI\rtarr  \VCatp$  whose value at $\bn$ is $\aA^n$. The injections, projections, and permutations of $\PI$ are sent to basepoint inclusions, projections, and permutations of the $\aA^n$.
For a $\sV_\bpt$-functor $F\colon \aA\rtarr \aB$,  $\bR F$ has component $F^n\colon \aA^n \rtarr \aB^n$ at $\bn$.   

The notation $\bR$ records that $\bR$ is right adjoint to the $2$-functor $\bL$ that sends a $\PI$-algebra to its first $\sV_{\bpt}$-category, $\bL(\aX)=\aX(\mathbf{1})$ \cite[\S1]{Rant2}. 
We claim that $\bR$ extends to a $2$-functor  $\bR\colon \OAlg \rtarr \DAlg$.  When starting operadically, it is convenient to use $\times$ instead of $\sma$. 
For an $\oO$-algebra $\aA$ in $\VCat$, we give $\bR\aA$ a $\sD$-algebra structure via a $\sV$-functor
\[\tha\colon\sD(\bm,\bn) \times \aA^m \rtarr \aA^n \] 
that is compatible with basepoints and therefore descends to a $\sV_\bpt$-functor on the smash product.
Writing $\phi_j = |\phi^{-1}(j)|$ again, this $\sV$-functor can be expressed as a composition
\[ \xymatrix@1{\coprod\limits_{ \ph\colon \bm\to \bn} \Big( \prod\limits_{1\leq j\leq n} \oO(\phi_j) \Big)\times \aA^m   \ar[r] 
& \coprod\limits_{\ph\colon \bm\to \bn} \,\prod\limits_{1\leq j\leq n} \big(\oO(\phi_j)\times \aA^{\phi_j}\big) \ar[r]^-{\tha} & \aA^n.\\}  \]
On each component $\phi\colon \bm\rtarr \bn$, 
 the first map reorders $\aA^m\iso \aA^{\phi_0} \times \aA^{\phi_1}\times \dots \times \aA^{\phi_n}$ and projects away $\aA^{\phi_0}$, while the second map
 is the product of $n$ algebra structure maps $\tha(\ph_j) \colon \oO(\phi_j)\times \aA^{\phi_j} \rtarr \aA$. 

We next define $\bR$ on morphisms. Thus let $\squiggly{(F,\pa_*)\colon\aA}{\aB}$ be a pseudomorphism 
of $\oO$-algebras. 
We define a $(\sD,\Pi)$-pseudomorphism 
\[\squiggly{\bR(F,\pa_*)=(\bR F , \de)\colon\bR\aA}{\bR\aB}.\]  
The required $\sV_{\ast}$-transformation 
\[\xymatrix{ 
\sD(\bm,\bn) \sma \aA^m \ar[r]^{\id \sma F^m} \ar[d]_{\tha} \drtwocell<\omit>{ \hspace{2em} \de_{\bm,\bn}} & \sD(\bm,\bn) \sma \aB^m \ar[d]^\tha \\
\aA^n\ar[r]_{F^n} & \aB^n
}\]
is obtained by passage to smash products from a coproduct of whiskerings of
\[\xymatrixcolsep{4em}\xymatrix{ 
\Prod_{1\leq j\leq n} \big(\oO(\phi_j)\times \aA^{\phi_j}\big) \ar[r]^-{\Prod \id\times F^{\phi_j}} \ar[d]_{\prod \tha} \drtwocell<\omit>{ \hspace{1.5em} \prod \pa }
& \Prod_{1\leq j\leq n} \big(\oO(\phi_j)\times \aB^{\phi_j}\big) \ar[d]^{\prod \tha} \\
\aA^n \ar[r]_-{F^n} & \aB^n
}\]
along the reordering morphisms $\Big( \Prod_{1\leq j\leq n} \oO(\phi_j) \Big)\times \aA^m \rtarr \Prod_{1\leq j\leq n} \big(\oO(\phi_j)\times \aA^{\phi_j}\big)$. 

For an $\oO$-transformation $\om\colon E \Longrightarrow F$, we define the component of the $\sD$-transformation $\bR \om$ at $\bn$ as $\om^n$.
We leave it to the reader to fill in the details of the proof of the following result.

\begin{prop} The above data specifies a $2$-functor $\bR\colon \OAlg \rtarr \DAlg$.
\end{prop}

\subsection{The proof that $\bR$ is a symmetric multifunctor}\label{Multiop2}

Now let $\oO$ be a reduced pseudo-commutative operad in $\VCat$
with associated pseudo-commutative category of operators $\sD$.  

\begin{thm}\label{ORExtend} The $2$-functor $\bR\colon \OAlg \rtarr \DAlg$ 
extends to a symmetric multifunctor $\Mult(\oO) \rtarr \Mult(\sD)$. 
\end{thm}

\begin{proof}   Let 
$(F,\de_i)\colon  
 (\aA_1, \dots, \aA_k) \rtarr \aB$,
be a $k$-ary morphism in $\Mult(\oO)$.  Here $F$ is a $\sV_\bpt$-functor $ \aA_1 \sma \cdots \sma \aA_k \rtarr \aB$ and $\de_i$ is given by $\sV$-transformations $\de_i(n)$ as in \autoref{deltan}. 
We must construct a $k$-ary morphism 
$\bR(F,\de_i)=(\bR F, \bR \de)\colon (\bR\aA_1,\dots, \bR\aA_k) \rtarr \bR\aB$ as in \autoref{MultiD}.

Writing $n=n_1\cdots n_k$ as before, the component
\[ \bR\aA_1(\bn_1)\sma \cdots \sma \bR\aA_k(\bn_k) \rtarr \bR\aB(\bn)\]
of $\bR F$ is
\[ 
\aA_1^{n_1}\sma  \cdots \sma  \aA_k^{n_k} \xrtarr{\ell} (\aA_1\sma \cdots \sma \aA_k)^n \xrtarr{F^n} \aB^n
,\]
where $\ell$ is a based version of the map defined in \autoref{omdefn}, using lexicographic ordering. Next, we specify the $\sV_\bpt$-transformations $\bR\de$ in the following specialization of diagram \autoref{delta}. 
\begin{equation}\label{delta2}
\xymatrix{
*-<1ex>{ \bigwedge\limits_{i} \sD(\bm_i,\bn_i) \sma \bigwedge\limits_i {\aA_i}^{m_i}}
 \ar[rr(0.65)]^-{\id \sma \bR{F}} \ar[d(0.85)]_{t}^\iso
\ddrrtwocell<\omit>{<0>\quad  \  \bR\delta} & & 
 *- {\bigwedge\limits_{i} \sD(\bm_i, \bn_i)\sma  \aB^{m }}
  \ar[d]^(0.4){\smaD \sma \id} \\
*-{\raisebox{-4ex}{$\bigwedge\limits_{i} \sD(\bm_i,\bn_i) \sma \aA_i^{m_i}$}}
 \ar[d]_(0.57){\bigwedge\limits_i \tha}  &  & 
 \sD(\bm,\bn)\sma \aB^{m} \ar[d]^{\tha}\\
\bigwedge\limits_{i} \aA_i^{n_i} \ar[rr]_{\bR F}  & & \aB^{n}. 
\\}
\end{equation}

Before passage to smash products, these transformations are constructed as disjoint unions of products of compositions of the $\de_i(n)$'s. 
To see this, consider, for instance, the case in which $k=2$ and $\bn_1=\bn_2=\mathbf{1}$. We restrict further to the components $\oO(m_1)\subseteq \sD(\bm_1,\mathbf{1})$ and $\oO(m_2)\subseteq \sD(\bm_2,\mathbf{1})$ corresponding to the maps $\phi\colon \bm_1\rtarr {\bf 1}$ and $\psi\colon \bm_2\rtarr {\bf 1}$ that send all non-basepoint elements to 1.  Then the $\sV_\bpt$-transformation $\bR\de$ is obtained by passage to smash products from the 2-cell
\[\scalebox{.85}{\xymatrix @C=4em{
\oO(m_1) \times \oO(m_2) \times \aA_1^{m_1}\times \aA_2^{m_2} \ar[d]_\iso \ar[r]^{\id\times\bR F} \ddrtwocell<\omit>{<-1>} &  \oO(m_1) \times  \oO(m_2)  \times \aB^{m_1m_2} \ar[d]^{\opair \times \id}
\\
\oO(m_1) \times \aA_1^{m_1} \times  \oO(m_2) \times  \aA_2^{m_2} \ar[d]_{\tha\times \tha} & \oO(m_1m_2)  \times \aB^{m_1m_2} \ar[d]^\tha
\\
 \aA_1 \times \aA_2 \ar[r]_{F}&  \aB  
}}\]
of \eqref{DeiDejCommute}, the axiom for commutation of cells $\de_i$ and $\de_j$, in \autoref{cohMultO}, where \autoref{MultiO} is completed.  
Now consider the general case of $k=2$, with arbitrary $\bn_1$ and $\bn_2$. For the component of $\sD(\bm_1, \bn_1)\times \sD(\bm_2, \bn_2)$ indexed by maps $\phi\colon \bm_1\rtarr \bn_1$ and $\psi\colon\bm_2\rtarr \bn_2$ in $\sF$, the required 2-cell is of the form 
\[\xymatrix@C=4em@R=5em{
\prod\limits_{1\leq j\leq n_1} \oO(\phi_j) \times \prod\limits_{1\leq k\leq n_2} \oO(\psi_k) \times \aA_1^{m_1}\times \aA_2^{m_2} \ar[d]_-{(\PI \theta\times \PI \theta) \circ t} \ar[r]^-{\opair\times\bR F} \drtwocell<\omit>{<-1>} &  \prod\limits_{j,k} \oO((\phi\sma \psi)_{(j,k)})  \times \aB^{m_1m_2} \ar[d]^-{\Pi \theta}
\\
 \aA_1^{n_1} \times \aA_2^{n_2} \ar[r]_-{\bR F}&  \aB^{n_1n_2}  
}\]
and is a product of 2-cells of the previous type. We note that $\bR \de$ is the identity when $\sD^k$ is restricted to $\Pi^k$,  by axioms \eqref{OperUnitAxiom} and \eqref{OperIdentAxiom} in \autoref{MultiO}.  When $k=1$, the construction above recovers that of \autoref{Rdef}.
Axiom \eqref{DeiDejCommute} of  \autoref{MultiO} implies that $\bR$ preserves composition.

We prove that $\bR$ is symmetric by a comparison of the definitions here with those of \autoref{Multiop}.  
Remembering the lexicographic reordering, it is straightforward to check by comparison
of \autoref{OObSym} with \autoref{ObjSym}  that $\bR (F\si) = (\bR F)\si$ for a $k$-ary morphism $(F,\de_i)$ of $\Mult(\oO)$.   
The equality of pasting diagrams required to ensure that the $\sV_\bpt$-transformations in $(\bR(F, \de_i)) \si$ and $ \bR ((F,\de_i)\si)$ are equal follows from 
axiom \eqref{DeiDejCommute} of  \autoref{MultiO}. 
\end{proof}

\section{The multicategory of $\sD_G$-algebras and the multifunctor $\bP$}
\label{sec:DG}

We introduced $\VCat$-categories of operators  in \autoref{sec:DAlg}, as well as the multicategory $\Mult(\sD)$ associated to any pseudo-commutative category of operators $\sD$. 
In this section, we finally bring in equivariance, starting in \autoref{GCats}, where we specialize the content of the previous sections to the category $G\sV$ of $G$-objects in $\sV$. In \autoref{COFG} we introduce $\GVCat$-categories of operators over $\sF_G$, the category of finite based $G$-sets.
A key idea here is \autoref{DtoDG}, which allows us to prolong from equivariant categories of operators over $\sF$ to equivariant categories of operators over $\sF_G$.
We introduce the multicategory $\Mult(\sD_G)$ in \autoref{sec:MultDG} and extend the prolongation functor to a symmetric multifunctor in \autoref{subsec:prolong}.
Much of this section is precisely  parallel to the previous one. 

\subsection{$G\sV$-categories and $G\sV_{\bpt}$-categories}\label{GCats}
So far, equivariance has not entered into the picture and yet everything we have done applies equally well equivariantly, as we now explain. Start again with a category $\sV$ satisfying Assumptions \ref{ass1} and \ref{ass2}, such as the category $\sU$ of spaces, and let $G$ be a finite group.  An action of $G$ on an object $X$ of $\sV$ can be specified in several equivalent ways.  One is to regard $G$  as a group in $\sV$ via \autoref{bUbV} and to require a map $G\times X\rtarr X$ in $\sV$ that satisfies the evident unit and associativity properties, expressed diagrammatically.  Another is to regard $G$ as a category with one object and to require a functor $G\rtarr \sV$ that sends the one object to $X$.  We  have the evident notion of a $G$-map $X\rtarr Y$.  

Let $G\sV$ denote the category of $G$-objects in $\sV$ and $G$-maps between them.  Then $G\sV$ is bicomplete, with limits and colimits created in $\sV$ and given the induced actions by $G$.  With the second description, this is a standard fact about functor categories.   Therefore $G\sV$  satisfies \autoref{ass1}.  Similarly, we have the $2$-category $\GVCat$ of categories internal to $G\sV$, which can be identified with the 2-category of $G$-objects in $\VCat$.  Hence $G\sV$ satisfies \autoref{ass2} as well.  Thus we can replace $\sV$ by $G\sV$ and everything we have said so far applies verbatim.   

\begin{rem}\label{GVvsV} 
As just noted, we can think of a $G\sV$-category as a $\sV$-category $\cC$ together with $\sV$-functors $g\colon \cC\rtarr \cC$ for $g\in G$.  Then a $G\sV$-functor 
$F \colon \cB\rtarr \cC$ is a $\sV$-functor $F$ such that the diagrams
\[  \xymatrix{   
\cB\ar[r]^-g \ar[d]_{F} & \cB  \ar[d]^{F}\\
\cC \ar[r]_-{g} & \cC\\} \]
commute.  A $G\sV$-transformation $\nu\colon E \Longrightarrow F$  is a $\sV$-transformation such that 
$\nu\colon \ob\cB \rtarr \mor{\cC}$ is $G$-equivariant.  This is equivalent to having the following equality of pasting diagrams.
\[  \xymatrix{   
\cB\ar[r]^-g \dtwocell^E_F{\nu} & \cB  \ar[d]^{E}\\
\cC \ar[r]_-{g} & \cC} \ \ \ \ 
\raisebox{-1.75em}{=}
\ \ \ \ 
\xymatrix{   
\cB\ar[r]^-g \ar[d]_{F} & \cB  \dtwocell^E_F{\nu}\\
\cC \ar[r]_-{g} & \cC} \]
\end{rem}  

The terminal object of $\sV$, with trivial $G$-action, is terminal in $G\sV$, and we have the category $G\sV_{\bpt}$ of based objects in $G\sV$, which can be identified with the category of $G$-objects in $\sV_\bpt$.  We also have the $2$-category $\GVCatp$ of categories internal to $G\sV_{\bpt}$, which can be identified with the 2-category of $G$-objects in $\VCatp$; an analogue of \autoref{GVvsV} applies in this case as well.

We also have $G\sV$-$2$-categories, which are defined to be categories enriched  in $\GVCat$, using cartesian products,
and $G\sV_{\bpt}$-$2$-categories, which  are categories enriched in $\GVCatp$, defined using smash products.  
We emphasize that $G$ does not act on the collection of objects of a $G\sV$-2-category. It acts on the $\sV$-categories of morphisms.

What changes is that we now build finite $G$-sets into the picture.  We work with categories of operators $\sD_G$ over $\sF_G$, our chosen permutative model of the category of based finite $G$-sets. We define these in \autoref{COFG},  we define pseudo-commutativity for them in \autoref{sec:PseuodocomG}, and we define algebras and pseudoalgebras over them in \autoref{DGpseudo}.   When $\sD_G$ is pseudo-commutative, we define a multicategory with underlying category $\DGAlg$ in \autoref{sec:MultDG}.    We define the prolongation functor $\bP$ from $\sD$-algebras to $\sD_G$-algebras and show that it extends to a symmetric multifunctor $\bP\colon \Mult(\sD) \rtarr \Mult(\sD_G)$  in \autoref{subsec:prolong}.  

\subsection{Categories of operators over $\sF_G$}\label{COFG} 
The definition of a category of operators over $\sF_G$ in this section  is parallel to that of a category of operators over  $\sF$ in \autoref{Dalgs}, and it is the categorical analogue of the definition given in \cite{MMO} on the space level. After giving the relevant definitions, we show how we can go back and forth between categories of operators over $\sF$ and categories of operators over $\sF_G$.

\begin{defn} Let $\sF_G$ be the following model of the $G$-category of based finite $G$-sets.  An object in $\sF_G$ consists of a based set $\bn$ together with a $G$-action prescribed by a homomorphism $\al\colon G\rtarr \SI_n$.  We denote this object by $\bn^\al$.The morphisms are defined to be all based functions, not just the equivariant ones, and we let $G$ act by conjugation on the set of morphisms. Thus $\sF_G$ can be viewed as a $G$-category where $G$ acts trivially on objects.

Let $\PI_G\subset \sF_G$ be the sub $G$-category of morphisms 
$\ph\colon \bm^\al \rtarr \bn^\be$ such that $\ph_j = 0$ or $1$ for $1\leq j\leq n$.   Write $\bn$ for $\bn^\epz$, where $\epz$ is the trivial homomorphism. 
That fixes compatible embeddings of $\PI$ in $\PI_G$ and $\sF$ in $\sF_G$. 
Note that $\bzo$ is a zero object in $\PI_G$ and $\sF_G$.
\end{defn}

\begin{rem}
\label{keyG}  
We think of $\SI_n$ as the subset of isomorphisms in $\PI(\bn, \bn)$.  For a based $G$-set $\bn^\al$, the homomorphism $\al$ thus maps $G$ to $\PI(\bn, \bn)$, which is contained in all categories of operators of either type.  We have built strictness with respect to $\PI$ into all of our structures, and the strictness with respect to permutations is crucial in dealing with equivariance, in particular in constructing the prolongation functor $\bP$.  
\end{rem}

\begin{defn}\label{GCO/FG} 
A \ourdefn{$\GVCat$-category of operators} $\sD_G$ over $\sF_G$, which we abbreviate to
$\GVCat$-$\mathbf{CO}$ over $\sF_G$, is a 
$G\sV$-$2$-category whose objects are the based $G$-sets $\bn^\al$ for $n\geq 0$, together with $G\sV$-$2$-functors 
\[\xymatrix@1{ \Pi_G \ar[r]^-{\io_G} & \sD_G \ar[r]^-{\xi_G} & \sF_G \\} \] 
such that $\io_G$ and $\xi_G$ are the identity on objects and $\xi_G\com \io_G$ is the inclusion.  
 A morphism $\nu_G\colon \sD_G\rtarr \sE_G$ of $\GVCat$-$\mathbf{CO}$s over $\sF_G$ is a $G\sV$-$2$-functor over $\sF_G$ and under $\PI_G$.
 
 For a $\GVCat$-$\mathbf{CO}$ $\sD_G$ over $\sF_G$, let $\sD$ denote the full subcategory on the objects $\bn=\bn^\varepsilon$ with trivial $G$-action. This is the underlying $\GVCat$-$\mathbf{CO}$ over $\sF$ of $\sD_G$.
\end{defn}

\begin{defn}\label{reducedGCO/FG}
A \ourdefn{$\GVCat$-category of operators} $\sD_G$ is \ourdefn{reduced} if $\bzo$ is a zero object, and  we then say that $\sD$
is a \ourdefn{$\GVCatp$-category of operators} over $\sF_G$. We shall restrict attention to $\GVCatp$-categories of operators over $\sF_G$. 
\end{defn}

\begin{rem}\label{reducedCO/FGrem}
As in \autoref{reducedCO/Frem},  a $\GVCatp$-$\mathbf{CO}$ $\sD_G$ over $\sF_G$ is a $G\sV_\ast$-$2$-category, with $\io_G$ and $\xi_G$  $G\sV_\ast$-$2$-functors, and  a morphism of  reduced $\GVCat$-$\mathbf{CO}$s over $\sF_G$ is reduced and is thus a $G\sV_\ast$-2-functor  over $\sF_G$ and under $\PI_G$.
\end{rem}

\begin{con}\label{DtoDG}
We construct a \ourdefn{prolongation functor} $\bP$  from the category of $\GVCatp$-$\mathbf{CO}$s over $\sF$ to the category of $\GVCatp$-$\mathbf{CO}$s over $\sF_G$. Let $\sD$ be a $\GVCatp$-$\mathbf{CO}$  over $\sF$. 
 Define the morphism $G\sV_\bpt$-category $\bP\sD(\bm^\al, \bn^\be)$ to be a copy of $\sD(\bm, \bn)$, but with $G$-action induced by conjugation and the original given  $G$-action on $\sD(\bm,\bn)$.  Explicitly, the action of $g\in G$ on $\bP\sD(\bm^\al, \bn^\be)$, which we shall call $\bP g$ when $\al$ and $\be$ are understood, is the composite
\begin{equation}\label{GD} 
 \xymatrix@1{\bP g:= \sD(\bm,\bn) \ar[rr]^-{\al(g^{-1})^*} & &  \sD(\bm,\bn) \ar[r]^-{g} &  \sD(\bm,\bn) \ar[rr]^-{\be(g)_*}  & & \sD(\bm,\bn). \\} 
 \end{equation}
Here $\al(g^{-1})^*$ and $\be(g)_*$ are defined by precomposition with $\al(g^{-1})$ and postcomposition with $\be(g)$; we think of them as prewhiskerings and postwhiskerings. Composition is inherited from $\sD$ and is equivariant. 
Observing that $\PI_G$ and $\sF_G$ are the prolongations of $\PI$ and $\sF$, the inclusion $\io_G$ and projection  $\xi_G$ are inherited from $\sD$ as $\bP\io$ and $\bP\epz$, This uses the functoriality of $\bP$, which we now explain. For a map of  $\GVCatp$-$\mathbf{CO}$s $\nu\colon\sD\to \sE$, we define 
\[\bP\nu \colon \bP\sD(\bm^\al,\bn^\be)\rtarr\bP\sE(\bm^\al,\bn^\be)\]
to just be $\nu$; it is equivariant with respect to the new action because $\nu$ is a $G\sV_\bpt$-2-functor and thus is compatible with the $G$-action and with precomposition and postcomposition with maps in $\Pi$. 
\end{con}

\begin{prop}
If $\sD_G$ is a $\GVCatp$-$\mathbf{CO}$ over $\sF_G$, then $\sD_G \cong \bP\sD$, where $\sD$ is the underlying $\GVCatp$-$\mathbf{CO}$ over $\sF$  of $\sD_G$.
\end{prop}
\begin{proof}
Let $\sD_G$ be a $\GVCatp$-$\mathbf{CO}$ over $\sF_G$. Let $\id^\al \in \Pi_G(\bm,\bm^\al)$ and $\id_\al \in \Pi_G(\bm^\al, \bm)$ be the morphisms given by the identity map on the set $\bm$. They are not identity morphisms but rather are mutual inverses in $\Pi_G$, and hence in $\sD_G$. Since $G$ acts by conjugation on $\Pi_G(\bm,\bm^\al)$ and $\Pi_G(\bm^\al, \bm)$, the action of $g$ sends $\id^\al$ and  $\id_\al$ to the maps on $\bm$ given by $\alpha(g)$ and $\alpha(g)^{-1}$, respectively. 

Precomposition with $\id^\al$ and postcomposition with $\id_\be$ induce an isomorphism of $\sV_\bpt$-categories
\begin{equation}\label{DGtoD}
\sD_G(\bm^\al,\bn^\be) \rtarr \sD_G(\bm,\bn)=\sD(\bm,\bn).
\end{equation}

The above observations and the fact that composition in $\sD_G$ is $G$-equivariant imply that this map becomes $G$-equivariant when we endow the target with the action defined on $\bP\sD(\bm^\al,\bn^\be)$, giving the desired isomorphism $\sD_G \cong \bP\sD$.
\end{proof}

\begin{rem}
The map of $\sV_\bpt$-categories in \autoref{DGtoD} induces a $\sV_\bpt$-2-functor
\[\sD_G \rtarr \sD.\]
It is an inverse up to invertible $\sV_\bpt$-2-natural transformation to the inclusion $\sD \rtarr \sD_G$. Thus $\sD$ and $\sD_G$ are equivalent as $\sV_\bpt$-2-categories, but not as $G\sV_\bpt$-2-categories.
\end{rem}
 
\begin{defn} For a reduced operad $\oO$ in $\GVCat$, define the associated category of operators $\sD_G(\oO)$ over $\sF_G$ to be the prolongation  $\bP(\sD(\oO))$. 
\end{defn}

A more explicit description is given on the space level in \cite[Definition 5.1]{MMO}.\footnote{Numbered references to \cite{MMO} refer to the first version posted on ArXiv;  they will be changed when the published version appears.}

\subsection{Pseudo-commutative categories of operators over $\sF_G$}
\label{sec:PseuodocomG}

Observe that $\PI_G$ and $\sF_G$ are permutative  under the smash product of finite based $G$-sets.  On underlying sets, the smash product and its symmetry isomorphism are defined just as for $\PI$ and $\sF$.  Recall that when we restrict to $\SI$, we denote the smash product $\spair$. This can be thought of as a collection of maps $\SI_m\times\SI_n \rtarr \SI_{mn}$. 
Then  homomorphisms $\al\colon G\rtarr \SI_m$ and $\be\colon G\rtarr \SI_n$ have the product homomorphism $\al\otimes \be$ given by applying $\spair$ elementwise; that is, $(\al\spair \be)(g)=\al(g)\spair\be(g)$. 

Then $\bm^\al\sma\bn^\be=\bm\bn^{\al \otimes \be}$, and the smash product 
\[  \sF_G(\bm^{\al}, \bp^{\ga}) \sma \sF_G(\bn^{\be}, \bq^{\de}) \rtarr \sF_G(\bm\bn^{\al\otimes \be}, \bp\bq^{\ga\otimes \de}) 
\]
is $G$-equivariant.   
 
The following definition is precisely analogous to  \autoref{pairprod}.  

\begin{defn}\label{pseudocomDG}
A \ourdefn{pseudo-commutative structure}
on a category of operators $\sD_G$ over $\sF_G$ is a pseudo-permutative structure $(\sD_G,\mathbf{1},\smaD,\ta)$ such that
\begin{enumerate}
 \item $\smaD$ restricts  to $\sma$ on 
$\PI_G\sma \PI_G$ and projects to $\sma$ on $\sF_G$ 
(in the sense that $\xi_G \circ \smaD = \sma \circ(\xi_G\sma \xi_G)$); 
\item $\smaD$ restricts to a strict $G\sV_\bpt$-2-functor on $\PI_G\esma\sD_G$ and $\sD_G\esma\PI_G$;
\item $\ta$ restricts to the symmetry on $\PI_G$.
\end{enumerate}
Define a map of pseudo-commutative categories of operators over $\sF_G$ as in \autoref{pseudocommap}.
\end{defn}

The following theorem states that the prolongation of a pseudo-commutative category of operators is again pseudo-commutative.

\begin{thm}\label{multithm2} Let $\sD$ be a pseudo-commutative $\GVCatp$-category of operators over $\sF$. Then $\bP\sD$ is a 
pseudo-commutative $\GVCatp$-category of operators over $\sF_G$ and the inclusion $(\sD,\PI)\rtarr (\bP\sD,\PI_G)$ preserves the pseudo-commutative structure.
\end{thm}

\begin{proof}
We define the $G\sV_\bpt$-pseudofunctor $\squiggly{\smaD\colon \bP\sD \sma \bP\sD}{\bP\sD}$ as follows. On objects, $\smaD$ is just $\sma$; that is, 
$\bm^\al \smaD \bn^\be = \bm\bn^{\al \spair \be}$. On $G\sV_\bpt$-categories of morphisms, 
\[  \smaD\colon \bP\sD(\bm^\al, \bp^\ga) \sma \bP\sD(\bn^\be, \bq^\de) \rtarr \bP\sD(\bm\bn^{\al\otimes \be}, \bp\bq^{\ga\otimes \de}) \]
is just
\[  \smaD\colon \sD(\bm, \bp) \sma \sD(\bn, \bq) \rtarr \sD(\bm\bn, \bp\bq). \]

We need to show that $\smaD$ is equivariant with respect to the action of $G$ on $\bP\sD$ (see \autoref{GVvsV}). The equivariance  is encoded in the following commutative diagram. 

\begin{equation}\label{DGsmaD}
\scalebox{0.85}{\xymatrix{ 
\sD(\bm,\bp)\sma \sD(\bn,\bq) \ar[ddd]_{\smaD}   \ar[dr]^-{\ \ \al(g^{-1})^*\sma \be(g^{-1})^*} \ar[rrr]^-{\bP g\sma \bP g} &  & &    \sD(\bm,\bp)\sma \sD(\bn, \bq)  \ar[ddd]^{\smaD} \\
&  \sD(\bm,\bp)\sma \sD(\bn,\bq)  \ar[d] _{\smaD} \ar[r]^{g\sma g} & \sD(\bm,\bp)\sma \sD(\bn, \bq) \ar[d]^{\smaD}  \ar[ur]^-{\ga(g)_*\sma \de(g)_* } &  \\
 & \sD(\bm\bn, \bp\bq)  \ar[r]_-{g}  & \sD(\bm\bn, \bp\bq)  \ar[dr]_-{(\ga\spair \de)(g)_* }  & \\
 \sD(\bm\bn, \bp\bq) \ar[ur]_{(\al\spair \be)(g^{-1})^*}  \ar[rrr]_{\bP g}   &&& \sD(\bm\bn,\bp\bq)) \\} }
\end{equation} 
The central square commutes 
since $\sD$ is a $\GVCatp$-category of operators, making the displayed functors $\smaD$ internal to $G\sV_\bpt$ and therefore equivariant.
The left and right trapezoids commute because $\smaD$ is strict when 
composing with morphisms in $\PI$ according to condition \eqref{PPDDcond} of \autoref{pairprod}.
The upper and lower trapezoids commute by definition, as in \autoref{GD}.  

The pseudofunctoriality constraint 
\[\scalebox{0.8}{\xymatrix{
\bP\sD(\bn^\be, \bp^\ga)  \sma \bP\sD(\br^\zeta, \bs^\eta) \sma \bP\sD(\bm^\al, \bn^\be) \sma \bP \sD(\bq^\de, \br^\zeta) 
\ar[rr]^-{\smaD \sma \smaD} \ar[d]_{\id\sma t\sma  \id}^{\iso}   \ddrrtwocell<\omit>{<0>\    \vartheta} & &
\bP\sD(\bn  \br^{\be\spair\zeta}, \bp  \bs^{\ga\spair\eta}) \sma \bP\sD(\bm  \bq^{\al\spair\de}, \bn  \br ^{\be\spair\zeta})\ar[dd]^\circ\\
\bP\sD(\bn^\be,\bp^\ga )\sma  \bP\sD(\bm^\al, \bn^\be)\sma \bP \sD(\br^\zeta, \bs^\eta) \sma \bP\sD(\bq^\de, \br^\zeta) 
\ar[d]_{\circ \sma  \circ}  & & \\
\bP\sD(\bm^\al,\bp^\ga) \sma \bP\sD(\bq^\de,\bs^\eta) \ar[rr]_-{\smaD}  & & \bP\sD(\bm\bq^{\al\spair\de},\bp\bs^{\ga\spair\eta} ) }}
\]
is just that of $\sD$ at the corresponding objects with trivial action (see \autoref{beta0}).  The equivariance of $\vartheta$ with respect to the original action on $\sD$ and the conditions on $\vartheta$ from \autoref{condstarcomp} combine to show that $\vartheta$ prewhiskered with $\bP g$ is equal to $\vartheta$ postwhiskered with $\bP g$, as needed. 

The symmetry $\ta$ is prolonged similarly. The 1-cell at the object $(\bm^\al,\bp^\be)$ is the permutation $\ta_{m,p}$ considered as a 1-cell in $\bP\sD$. The pseudonaturality constraint $\hat{\tau}$ (see \autoref{symmetry})
is given by that on $\sD$. One needs to check that this $\sV_\bpt$-transformation is $G$-equivariant with respect to the prolonged action. This follows from the equivariance of $\hat{\ta}$ with respect to the original action on $\sD$, the compatibility of $\hat{\ta}$ with $\vartheta$, condition \eqref{PPDDcond} of \autoref{pairprod}, and the fact that $\hat{\ta}$ restricted to $\PI$ is the identity. 
\end{proof} 

The construction is functorial with respect to pseudo-commutative morphisms of pseudo-commutative categories of operators over $\sF$.  Theorems \ref{multithm}  and \ref{multithm2} have the following corollary.

\begin{cor}  If $\oO$ is a pseudo-commutative operad, then $\sD_G(\oO)$ is  pseudo-commutative category of operators over $\sF_G$.
\end{cor}

\subsection{Algebras and pseudoalgebras over categories of operators over $\sF_G$}
\label{DGpseudo}

Just as for categories of operators over $\sF$, Definitions \ref{strict}, 
\ref{Cpseudo}, and \ref{Ctran} specialize to 
define a 2-category of algebras, pseudomorphisms, and transformations for categories of operators $\sD_G$ over $\sF_G$.

\begin{notn}\label{defn:DGAlg} 
Let $\sD_G$ be a $\GVCatp$-$\mathbf{CO}$ over $\sF_G$. We denote by 
$\DGAlg$ the 2-category 
of strict $\sD_G$-algebras, $(\sD_G,\Pi_G)$-pseudo\-morphisms, and $\sD_G$-transformations. This 2-category was denoted by ${\DGAlg}_{\PI_G}$ in \autoref{CatCAlg}. Just as in \autoref{defn:DAlg} we fix $\cB=\PI_G$ and drop the subscript from the notation.
\end{notn} 

Again, we do not discuss general $\sD_G$-pseudoalgebras here, leaving such consideration for \cite{AddCat2}. However, we will need pseudoalgebras in the special case of $\sD_G=\sF_G$ starting in \autoref{sec:section}. Recall from \autoref{CatCAlg} that we have the 2-category $\psFGAlg$ of weak $\sF_G$-pseudoalgebras, weak $\sF_G$-pseudomorphisms, and $\sF_G$-transformations. 

\begin{rem}\label{ExplainWeak}
We comment on the choice of weak pseudoalgebras.
This choice is already essential nonequivariantly. Look back at \autoref{RoadMap}, but take $G=e$.   
With more effort, we could have started with $\oO$-pseudoalgebras, as defined in \cite{AddCat1}; the functor $\bR$ 
would then land in $\sD$-pseudoalgebras that are strict over $\PI$.
However, the section $\oursection \colon \sF\rtarr \sD$ (see \autoref{TheSection}) loses strictness with respect to $\PI$, as explained in \autoref{warn}, so that whether the 
domain of $\oursection^*$ is taken to be $\DAlg$ or some 2-category of $\sD$-pseudoalgebras, its target must still be $\psFAlg$.  It takes strict $\sD$-algebras only to weak $\sF$-pseudoalgebras.
\end{rem}

\subsection{The symmetric multicategory of $\sD_G$-algebras}
\label{sec:MultDG}

Let $\sD_G$ be a pseudo-commutative category of operators over $\sF_G$.  
We define the multicategory $\Mult(\sD_G)$ of $\sD_G$-algebras, which amounts to defining the $k$-ary morphisms.
The definition  is exactly like \autoref{MultiD}, hence we refer the reader there for details.  Again, we set it up to have its objects be $\sD_G$-algebras, although with only slightly more work we could equally well have set it up to have its objects be $\sD_G$-pseudoalgebras.  Remember that we write 
\[ \tha\colon \sD_G(\bm^\al,\bn^\be)\sma \aX(\bm^\al) \rtarr \aX(\bn^\be) \] 
for the action of $\sD_G$ on a $\sD_G$-algebra $\aX$.  For a $k$-tuple of finite $G$-sets $\bm_i^{\al_i}$,
we write $\bm^\al$ for the finite $G$-set with  $m = m_1\cdots m_k$ and $\al = \al_1\otimes \cdots \otimes \al_k$. 

Recall from \autoref{PairsPairs} that given $\sD_G$-algebras $\aX_1,\dots,\aX_k$, we have the external smash product $\aX_1\overline{\sma} \dots \overline{\sma} \aX_k$, which  is a $\sD_G^{\esma k}$-algebra. For a $\sD_G$-algebra $\aY$, we have the $\sD_G^{\esma k}$-pseudoalgebra $\sY \circ \smaD^k$, which is strict over 
$\PI^{\esma k}_G$.

\begin{defn}\label{MultiDG} 
We define a symmetric multicategory $\Mult(\sD_G)$ of $\sD_G$-algebras. The objects are the $\sD_G$-algebras. 
For $\sD_G$-algebras $\aX_i$, $1\leq i\leq k$, and $\aY$, a $k$-ary morphism
$\ul{\aX} \rtarr \aY$ consists of a $(\sD_G^{\esma k},\PI_G^{\esma k})$-pseudomorphism
\[\squiggly{F\colon \aX_1\overline{\sma} \dots \overline{\sma} \aX_k}{\aY\circ \smaD^k}.\] 
Composition and the symmetric action are specified as in \autoref{MultiD}.
\end{defn}

Unpacking the definition, a $k$-ary morphism $F=(F,\de)$ consists of $G\sV_\bpt$-functors 
\[ F\colon \aX_1(\bm_1^{\al_1}) \sma \cdots \sma \aX_k(\bm_k^{\al_k}) \rtarr \aY(\bm^\al), \]
together with invertible $G\sV_\bpt$-transformations $\de$ as  in the following diagram, in which $1\leq i\leq k$.

\begin{equation}\label{deltaG}
\scalebox{0.9}{\xymatrix @R=5ex{
*-<1ex>{\bigwedge\limits_{i} \sD_G(\bm_i^{\al_i},\bn_i^{\be_i})  \sma \bigwedge\limits_i \aX_i(\bm_i^{\al_i})} 
\ar[]!<15.5ex,0pt>;[rr(0.43)]^-{\id\sma F} 
\ar[]!<0ex,-1ex>;[d(0.65)]_(.45){t}^(0.45)\iso
\ddrrtwocell<\omit>{<0>\quad  \  \delta} & & 
*-{\bigwedge\limits_{i}  \sD_G(\bm_i^{\al_i}, \bn_i^{\be_i})\sma  \aY(\bm^{\al})}
 \ar[d]^(0.35){\smaD\sma \id} \\
*-{\raisebox{-3ex}{$\bigwedge\limits_{i}  \sD_G(\bm_i^{\al_i}, \bn_i^{\be_i}) \sma \aX_i(\bm_i^{\al_i})$}} \ar[d(0.7)]_(0.7){\bigwedge\limits_i \tha}  &  &
\raisebox{1pt}{$ \sD_G(\bm^{\al},\bn^{\be})\sma \aY(\bm^\al)$}  \ar[d]^(0.45){\tha}\\
*+{\bigwedge\limits_{i} \aX_i(\bn_i^{\be_i})} \ar[rr]_{F}  & & \aY(\bn^{\be}). 
\\} }
\end{equation}

The strictness over $\Pi_G^{\esma k}$ is encoded by requiring $\de$ to be the identity if 
all factors $\sD_G$ are restricted to $\PI_G$.
These transformations must satisfy coherence with respect to composition in $\sD_G^{\esma k}$, details of which can be found in \autoref{cohMultD}.

We have the following analogue of \autoref{pullingback}; its proof is essentially the same.

\begin{thm}\label{pullingbackG}
 Let $(\Psi,\mu) \colon \sD_G \rtarr \sE_G$ be a map of pseudo-commutative categories of operators over $\sF_G$.  Then pulling back along $\Psi$ induces a (symmetric) multifunctor
 \[\Psi^* \colon \Mult(\sE_G) \rtarr \Mult(\sD_G).\] 
 \end{thm}
 
\subsection{The symmetric multifunctor $\bP$}
\label{subsec:prolong}

Let $\sD$ be a $\GVCatp$-category of operators over $\sF$ and
$\sD_G =  \bP\sD$ be its prolonged category of operators over $\sF_G$, as defined in \autoref{DtoDG}.
Define 
$\bU\colon \DGAlg \rtarr \DAlg$
by restricting along the inclusion $\sD\subset \sD_G$. Then $\bU$ has a left adjoint 
prolongation functor on the level of algebras,
\[\bP: \DAlg \rtarr \DGAlg.\]
For the subcategory of {\it strict} maps, this is the categorical analogue of the prolongation functor from
\cite[\S4.3]{MMO} or \cite{Shim}, and it gives an equivalence of categories.
We will discuss the extension of $\bP$ to $(\sD,\Pi)$-pseudo\-morphisms 
in \autoref{PMultiFun} below.
  
 On objects, 
 $\bP(\aX)(\bn^\al)\in \GVCatp$ is defined  by letting $\bP(\aX)(\bm^\al)$ be a copy of $\aX(\bm)$, but with the action of 
$g\in G$, denoted $\bP g$ when $\al$ is understood, defined to be the composite
\[
\xymatrix@1{ \aX(\bm) \ar[r]^-{g} & \aX(\bm) \ar[r]^-{\al(g)_*} & \aX(\bm).\\} 
\]
Here $\al(g)_*\colon \aX(\bm)\rtarr \aX(\bm)$ is given by the action of $\PI$ on $\aX$.  
The enriched functor $\aX$ takes the morphisms of $\Pi$, which are $G$-fixed, to $G$-equivariant functors. It follows that 
 we can equivalently write $\bP g$ as the composite
\begin{equation}\label{GDX2} 
\xymatrix@1{ \aX(\bm) \ar[r]^-{\al(g)_*}   & \aX(\bm)\ar[r]^-{g} & X(\bm).\\} 
\end{equation}

The action $G\sV_{\bpt}$--functor 
\[  \bP\tha \colon  \sD_G (\bm^\al, \bn^\be)\sma \bP \aX(\bm^\al) \rtarr \bP\aX(\bn^\be) \] 
is defined to be 
\[ \tha\colon \sD(\bm,\bn) \sma \aX(\bm) \rtarr  \aX(\bn). \]
  The following diagram shows that 
$\bP\tha$ is equivariant because $\tha$ is equivariant, as displayed in the middle square.  
\[
\scalebox{0.9}{\xymatrix{ 
\sD(\bm,\bn)\sma \aX(\bm) \ar[ddd]_{\tha}   \ar[dr]^-{\al(g^{-1})^*\sma \al(g)_*} \ar[rrr]^-{\bP g\sma \bP g} &  & &    \sD(\bm,\bn)\sma \aX(\bm)  \ar[ddd]^{\tha} \\
&  \sD(\bm,\bn)\sma \aX(\bm)  \ar[d] _{\tha} \ar[r]^{g\sma g} & \sD(\bm,\bn)\sma \aX(\bm) \ar[d]^{\tha}  \ar[ur]^-{\be(g)_*\sma\id} &  \\
 & \aX(\bn)  \ar[r]_-{g}  & \aX(\bn)  \ar[dr]_-{\be(g)_*}  & \\
 \aX(\bn) \ar@{=}[ur]   \ar[rrr]_{\bP g}   &&& \aX(\bn) \\}}
\]
Since the action $\tha$ is compatible with composition in $\sD$, the left and right trapezoids commute, the left one using that $\al(g^{-1})_*\com \al(g)_*= \id$.
 
With these preliminaries, we have the following result.

\begin{thm}\label{PMultiFun} The functor $\bP \colon \DAlg \rtarr \DGAlg$ extends to a (symmetric) multifunctor
\[ \bP \colon \Mult(\sD) \rtarr \Mult(\sD_G).\]
\end{thm}

\begin{proof}  
Since the values of $\bP \aX$ at objects in $\sD_G$ are the values of $\aX$ but with a new $G$-action, the idea of the proof is to show that the data of a map of $\sD$-algebras remains $G$-equivariant with respect to the new action. 

Thus let  $(F,\de)\colon (\aX_1,\dots,\aX_k) \rtarr \aY$ be a $k$-ary morphism of $\sD$-algebras. This means that we are given
$G\sV_{\bpt}$-functors 
\[ F\colon \aX_1(\bm_1) \sma \cdots \sma  \aX_k(\bm_k)  \rtarr \aY(\bm) \]
and invertible $G\sV_\bpt$-transformations $\de$ is as in \autoref{delta}.  We define $\bP(F,\de)=(\bP F,\bP\de)$ as follows. For $\bm_1^{\al_1},\dots, \bm_k^{\al_k}$, with $m=m_1\cdots m_k$ and $\al = \al_1\otimes \cdots \otimes\al_k$, the map
\[  \bP F\colon \bP\aX_1(\bm_1^{\al_1}) \sma \cdots \sma \bP\aX_k(\bm_k^{\al_k}) \rtarr  \bP\aY(\bm^\al),\] 
is given by $F$. Similarly, we define $\bP\de$ in the diagram \autoref{deltaG} to be $\de$ in the underlying
diagram \autoref{delta}. 

We must check that these give $G\sV_\bpt$-functors and $G\sV_\bpt$-transformations, respectively, that is, that they are equivariant with respect to the prolonged action.
When $k=1$, this gives the promised definition of $\bP$ on pseudomorphisms of $\sD$-algebras.
We first check that $\bP F$ is equivariant. For ease of notation, we consider the case $k=2$. 
Recall the $G$-action on $\bP \aX(\bm^\al)$ given in \autoref{GDX2}.
Then to check that
\[ \bP\aX_1(\bm_1^{\al_1}) \sma \bP\aX_2 (\bm_2^{\al_2}) \xrtarr{\bP F} \bP\aY(\bm^\al)\]
is $G$-equivariant, it suffices to show that the diagram
\[ \xymatrix{
\aX_1(\bm_1) \sma \aX_2(\bm_2) \ar[r]^-F \ar[d]_{\al_1(g)_*\sma \al_2(g)_*} & \aY(\bm) \ar[d]^{\al(g)_*} \\
\aX_1(\bm_1) \sma \aX_2(\bm_2) \ar[r]^-F \ar[d]_{g\sma g} & \aY(\bm) \ar[d]^{g} \\
\aX_1(\bm_1) \sma \aX_2(\bm_2) \ar[r]_-F  & \aY(\bm) 
}\]
commutes.
The top square commutes since $F$ restricts to a strict transformation of $\Pi$-functors (see \autoref{Cpseudo}). The bottom square commutes since $F$ is equivariant with respect to the original action.

It remains to check that $\bP\de$ is equivariant with respect to the prolonged action. By \autoref{GVvsV}, this is done by proving that prewhiskering $\de$ with $\bP g$ is equal to postwhiskering it.
We illustrate by taking $k=1$.
Thus we return to \autoref{Cpseudo}.
Let  $\aX$ and $\aY$ be $\sD$-algebras and let  $F\colon \squiggly{\aX}{\aY}$ be a $(\sD,\Pi)$-pseudo\-morphism.  
The 2-cell
\[  \scalebox{0.9}{
\xymatrix{
\sD(\bm,\bn)\sma \aX(\bm) \ar[r]^{1\sma F} \ar[d]_{\be(g)_*\al(g^{-1})^*}^{\sma \al(g)_*} & \sD(\bm,\bn)\sma \aY(\bm)  \ar[d]_{\be(g)_*\al(g^{-1})^*}^{\sma \al(g)_*} \\
\sD(\bm,\bn)\sma \aX(\bm) \ar[r]^{1\sma F} \ar[d]_{g\sma g} & \sD(\bm,\bn)\sma \aY(\bm) \ar[d]^{g\sma g} \\
\sD(\bm,\bn)\sma \aX(\bm) \ar[r]^{1\sma F} \ar[d]_\tha \drtwocell<\omit>{ \hspace{.5em} \de}
& \sD(\bm,\bn)\sma \aY(\bm) \ar[d]^\tha \\
\aX(\bn) \ar[r]_F & \aY(\bn)
} 
\quad
\raisebox{-15ex}{\text{is}}
\xymatrix{
\sD(\bm,\bn)\sma \aX(\bm) \ar[r]^{1\sma F} \ar[d]_{\be(g)_*\al(g^{-1})^*}^{\sma \al(g)_*} & \sD(\bm,\bn)\sma \aY(\bm)  \ar[d]_{\be(g)_*\al(g^{-1})^*}^{\sma \al(g)_*} \\
\sD(\bm,\bn)\sma \aX(\bm) \ar[r]^{1\sma F} \ar[d]_{\tha} \drtwocell<\omit>{ \hspace{.5em} \de} & \sD(\bm,\bn)\sma \aY(\bm) \ar[d]^{\tha} \\
 \aX(\bn) \ar[r]_{ F} \ar[d]_g 
&  \aY(\bn) \ar[d]^g \\
\aX(\bn) \ar[r]_F & \aY(\bn),
}
}
\]
by the equivariance of $\de$ with respect to the original action, and this also agrees with the 2-cell
\[  \scalebox{0.9}{
\xymatrix{
\sD(\bm,\bn)\sma \aX(\bm) \ar[r]^{1\sma F} \ar[d]_{\tha} \drtwocell<\omit>{ \hspace{.5em} \de}& \sD(\bm,\bn)\sma \aY(\bm)  \ar[d]^{\tha} \\
 \aX(\bn) \ar[r]_{ F} \ar[d]_{\be(g)_*}  & \aY(\bn) \ar[d]^{\be(g)_*} \\
 \aX(\bn) \ar[r]_{ F} \ar[d]_g 
&  \aY(\bn) \ar[d]^g \\
\aX(\bn) \ar[r]_F & \aY(\bn)
}
}
\]
by the compatibility of $\de$ with composition and the strictness of $\de$ with respect to $\Pi$.
The equivariance of $\bP\de$ for $k>1$ is deduced from the equivariance of $\de$ in a similar way. The diagrams are larger, but the verification is essentially the same.

Since $\bP F$ and $\bP \de$ are just $F$ and $\de$ on the underlying $\sV_\bpt$-categories, it follows that $\bP$ respects composition, the identity, and the $\SI$-action.
\end{proof}

\begin{defn}
We define $\bR_G$ to be the composite
\[\xymatrix{ \OAlg \ar[dr]_{\bR_G}  \ar[r]^\bR & \DAlg \ar[d]^{\bP} \\
 & \DGAlg }\]
\end{defn}

\begin{cor}\label{RGMulti}
The functor $\bR_G\colon \OAlg \rtarr \DGAlg$ extends to a symmetric multifunctor.
\end{cor}

\section{From $\sD_G$-algebras to $\sF_G$-pseudoalgebras} 
\label{sec:section}

We recall the notion of an $E_\infty$ $G$-operad in $G\Top$ and in $\Cat(G\Top)$ in \autoref{sec:Einfty}. For a category of operators $\sD_G$ arising from a chaotic $E_\infty$ $G$-operad $\oO$, we produce a pseudofunctor $\oursectionG\colon \squiggly{\sF_G }{\sD_G}$ that is a section to $\xi_G$  in \autoref{TheSection}. 
Finally, in \autoref{sectMulti}, we show that pulling back along $\oursectionG$ defines a symmetric multifunctor from $\sD_G$-algebras to $\sF_G$-pseudoalgebras.

Throughout this section, we restrict attention to the topological case $\sV=\Top$ and thus $G\sV = G\Top$.  

\subsection{$E_{\infty}$ $G$-operads}\label{sec:Einfty}

So far, our chaotic operads have been quite general.  We now restrict attention to a chaotic $E_{\infty}$ $G$-operad $\oO$ of $G\Top$-categories and its associated category of operators $\sD_G = \sD_G(\oO)$ over $\sF_G$.

\begin{defn}\label{einfty}
An operad $\oO$ in $G\Top$ is an \ourdefn{$E_\infty$ $G$-operad} if for all $n\geq 0$ and all subgroups $\LA$ of $G\times \SI_n$, the fixed point space $\oO(n)^\LA$ is contractible if $\LA \cap \SI_n=\{e\}$ and is empty otherwise. 

An operad $\oO$ in $\Cat(G\Top)$ is an \ourdefn{$E_\infty$ $G$-operad} if the operad $B\oO$ obtained by applying the classifying space functor levelwise is an $E_\infty$ $G$-operad of $G$-spaces.
\end{defn}

The condition on fixed-points implies that for an $E_\infty$ $G$-operad in $G\Top$, the space $\oO(n)$ is a universal principal $(G,\SI_n)$-bundle. Algebras over $E_\infty$ $G$-operads, are, up to group completion, equivariant infinite loop spaces with deloopings with respect to all finite-dimensional $G$-representations, and thus give rise to genuine $G$-spectra. For more background and examples we refer the reader to \cite[Section~2.1]{GMPerm}. 

The following result shows how chaotic categories are useful in this context.

\begin{prop}
 Let $\oO$ be a chaotic operad in $\Cat(G\Top)$. Then $\oO$ is an $E_\infty$ $G$-operad if and only if for all $n\geq 0$ and all subgroups $\LA$ of $G\times \SI_n$, the fixed point object space $(\ob\oO(n))^\LA$ is non-empty if $\LA \cap \SI_n=\{e\}$ and is empty otherwise.
\end{prop}

\begin{proof}
 As noted in \cite[Remark 1.15]{AddCat1}, the classifying space of a non-empty chaotic $\sU$-category is contractible. Thus, the statement follows by noting that if $\oO$ is chaotic, then $\oO(n)^\LA$ is also chaotic.
\end{proof}

\subsection{The section map $\oursectionG$ from $\sF_G$ to $\sD_G$} 

Recall that  $\sD_G$ comes equipped with functors 
\[\iota_G\colon \PI_G\rtarr \sD_G \ \ \text{and} \ \  \xi_G\colon \sD_G\rtarr \sF_G\]
such that $\xi_G\circ\iota_G$ is the inclusion. We here define an (equivariant) section 
\[\xymatrix@1{\oursectionG \colon \sF_G\ar@{~>}[r] &\sD_G}\]
to $\xi_G$.

\begin{defn}\label{pseudoCO}
A \ourdefn{pseudomorphism}
$\squiggly{\nu\colon \sD_G}{\sE_G}$ of $\mathbf{CO}$s over $\sF_G$ is a $G\sU_\bpt$-pseudofunctor  over $\sF_G$ and under $\PI_G$.
\end{defn}

\begin{prop}\label{TheSection}
Let $\oO$ be a chaotic $E_{\infty}$ $G$-operad in $\GUCat$ and let $\sD_G = \sD_G(\oO)$.
Then there exists a pseudomorphism 
\[\xymatrix@1{\oursectionG \colon \sF_G\ar@{~>}[r] &\sD_G}\]
of $\mathbf{CO}$s over $\sF_G$ and an invertible $G\sU_\bpt$-pseudotransformation 
\[ \xymatrix{
\sD_G  \ar[dr]_{\xi_G} \ar[rr]^{\id}  \rrtwocell<\omit>{<3>  \chi}  & & \sD_G. \\
& \sF_G  \ar@{~>}[ur]_{\oursectionG}& \\} \]
\end{prop}

\begin{proof}
On objects, we have no choice: $\oursectionG$ is the identity. More generally, on $\PI_G$ we must take $\oursectionG = \iota_G$. Given finite $G$-sets 
$\bm^\al$ and $\bn^\be$, we must specify a based $G$-equivariant function 
\[ \oursectionG \colon  \sF_G(\bm^\al,\bn^\be) \rtarr \ob\sD_G(\bm^\al,\bn^\be).\]
To define an equivariant function, it suffices to specify the function on each $G$-orbit. Moreover, an equivariant function out of an orbit is completely determined by its value at any point in the orbit. We thus {\it choose, for each $\bm^\al$ and $\bn^\be$, a point of each $G$-orbit of the $G$-set $\sF_G(\bm^\al,\bn^\be)$}. 

Let $f\in\sF_G(\bm^\al,\bn^\be)$ be such a chosen element. Let $H\leq G$ be the stabilizer of $f$. The section $\oursectionG$ must send $f$ to an $H$-fixed object of the $G$-category $\sD_G(\bm^\al,\bn^\be)$ that is in the component $\prod_{1\leq j\leq k} \oO(f^{-1}(j))$ of $f$. Since $\oO$ is an $E_{\infty}$ $G$-operad, the $H$-fixed point subset of this component has contractible classifying space, by \cite[Theorem 5.4]{MMO}, hence is nonempty. Thus we can choose an $H$-fixed object $\oursectionG(f)$ in the component of $f$. The only exception to such use of choices is that we already know the definition of $\oursectionG$ on $\PI_G$, so these choices only apply to morphisms of $\sF_G$ that are not in $\PI_G$.

The claim is that these choices specify an equivariant pseudofunctor  $\oursectionG$. The equivariance has been forced by the definition of $\oursectionG$: if $f$ is one of our distinguished points, then $\oursectionG(g\cdot f) = g\cdot \oursectionG(f)$. To see the pseudofunctor structure, we must specify a $G$-equivariant natural isomorphism 
\[ \xymatrix{
\sF_G(\bn^\be,\bp^\ga)\sma  \sF_G(\bm^\al,\bn^\be)\ar[d]_{\com} \ar[r]^{\oursectionG \sma \oursectionG} \drtwocell<\omit>{\varphi}
& \sD_G(\bn^\be,\bp^\ga) \sma \sD_G(\bm^\al,\bn^\be) \ar[d]^{\circ} \\
\sF_G(\bm^\al,\bp^\ga) \ar[r]_{\oursectionG} & \sD_G(\bm^\al,\bp^\ga).
}\]
The component of $\varphi$ at $(h,f)$ must be a morphism in $\sD_G(\bm^\al,\bp^\ga)$ of the form
\[ \varphi_{h,f}\colon \oursectionG(h)\circ \oursectionG(f) \rtarr \oursectionG(h\circ f).\]
Since both of these points in $\sD_G(\bm^\al,\bp^\ga)$ live over $h\circ f\in \sF_G(\bm^\al,\bp^\ga)$ and $\oO$ is chaotic, there is a unique morphism with the required source and target. 

We claim that $\varphi$ is $G$-equivariant. This means that $\varphi_{g\cdot h,g\cdot f} = g\cdot \varphi_{h,f}$ for $g\in G$. This again holds because $\oO$ is chaotic: there is a unique morphism between any two objects, so these two morphisms are necessarily the same.  The compatibility of $\varphi$ with triple composition follows again from the uniqueness of these morphisms.

For an object $\bm^\al$, the 1-cell component $\chi_{\bm^\al}\colon \bm^\al \rtarr \oursectionG \circ \xi_G(\bm^\al)=\bm^\al$ is the identity map. We need to construct the pseudonaturality constraint, which is an invertible $G\sU_\bpt$-transformation 
\[ \xymatrix{
\sD_G(\bm^\al,\bn^\be)\ar[d]_{\id} \ar[r]^{\oursectionG \circ \xi_G} \drtwocell<\omit>{\chi}
&  \sD_G(\bm^\al,\bn^\be) \ar[d]^{(\chi_{\bm^\al})^*} \\
\sD_G(\bm^\al,\bn^\be) \ar[r]_{(\chi_{\bn^\be})_*} & \sD_G(\bm^\al,\bn^\be).
}\]

For a $1$-cell $d\in \sD_G(\bm^\al,\bn^\be)$,   the object $\oursectionG\circ \xi_G(d)$ is in the same component of the same $G$-fixed point subset as $d$, hence there is a unique morphism $\chi_d\colon  \oursectionG\circ \xi_G(d)\rtarr d$, and this assignment is continuous and equivariant. The uniqueness implies naturality and the required compatibility with $\varphi$.  That and the evident inverse isomorphism lead to the conclusion.
\end{proof} 

\begin{rem}\label{warn}  
Although the section $\oursectionG$ is strictly functorial when restricted to $\PI_G$, it is not when only one of the morphisms of a composite in $\sF_G$ is in $\PI_G$, even when $G = e$.  Let $p_n\colon \bn\rtarr \bf{1}$ be the based function that sends $j$ to $1$ for $1\leq j\leq n$.  Then $\oursection(p_n)=(p_n;d)$ for some $d\in \oO(n)$. For a permutation $\si\in \SI_n$,  we have $p_n\com \si =p_n$, while $\oursection(p_n) \circ \oursection(\si) = (p_n ;d\cdot \si)$. In the cases of interest $\SI_n$ acts freely on $\oO(n)$ and we cannot have  $\oursection(p_n) \circ \oursection(\si)= \oursection(p_n)$.
It is this fact that led us to the distinctions highlighted in 
\autoref{ExplainWeak}.  
 \end{rem}
 
Recall the definition of a map between pseudo-commutative categories of operators over $\sF_G$ from 
Definitions \ref{pseudocommap} and \ref{pseudocomDG}. We adapt that definition to pseudomorphisms of $\mathbf{CO}$s over $\sF_G$. 

\begin{defn}
A \ourdefn{pseudomorphism of pseudo-commutative $\mathbf{CO}$s over $\sF_G$ }  
consists of a pseudomorphism $\squiggly{\nu\colon \sD} {\sE}$
of $\mathbf{CO}$s over $\sF_G$ 
and a $G\sU_\bpt$-pseudotransformation
\[ \xymatrix{
\sD \sma \sD \ar@{~>}[r]^{\nu \sma \nu} \ar@{~>}[d]_{\opair_\sD} \drtwocell<\omit>{\mu}
& \sE \sma \sE \ar@{~>}[d]^{\opair_\sE} \\
\sD \ar@{~>}[r]_{\nu} & \sE
}\]
satisfying analogues of all the axioms in \autoref{pseudocommap}.
\end{defn}

\begin{prop}\label{DGFGsmaD}
Let $\oO$ be a chaotic $G$-operad and let $\sD_G = \sD_G(\oO)$.
Then the section 
$\xymatrix@1{\oursectionG \colon \sF_G\ar@{~>}[r] &\sD_G}$
 is a pseudomorphism of pseudo-commutative $\mathbf{CO}$s. 
 \end{prop} 
\begin{proof} We must produce a $G\sU_{\bpt}$-pseudotransformation $\mu$ as above.
Since $\oursectionG$ is the identity on objects and $\mu$ must restrict to the identity on $\PI_G$, we take the $1$-cell component of $\mu$ to be the identity. Given $1$-cells $f$ and $h$ in $\sF_G$, we must produce an invertible two-cell
\[ \mu_{f,h} \colon \oursectionG(f\sma h) \iso \oursectionG(f) \smaD \oursectionG(h),\] 
that is the component of a $G\sU_\bpt$-transformation.
Since the specified source and target  live over $f\sma h$ in $\sF_G$, the fact that $\oO$ is chaotic  implies that there is a unique choice for $\mu_{f,h}$; that $\mu$ is natural and equivariant, and satisfies all the axioms follows for the same reason.
\end{proof}

\subsection{The symmetric multifunctor $\oursectionG^*$}
We have elected to work with (strict) $\sD_G$-algebras but, due to \autoref{warn}, when we precompose with our section $\oursectionG$, we only produce weak 
$\sF_G$-pseudoalgebras; recall from \autoref{CatCAlg} that we defined the objects and morphisms of $\psFGAlg$ 
as enriched pseudofunctors from $\sF_G$ to $\CatGVp$ and pseudotransformations, with no strictness conditions over $\PI_G$. The following proposition follows immediately from the definitions.

\begin{prop} Pullback along the section $\oursectionG$ defines a 2-functor
\[ \oursectionG^*\colon\! \DGAlg \rtarr \psFGAlg.\]
\end{prop}

\begin{rem}
Since $\oursectionG$ restricts to the inclusion $\io_G$ on $\Pi_G$, for any $\sD_G$-algebra $\aX$, the underlying $\Pi_G$-algebra of $\oursectionG^*\aX$ is the underlying $\Pi_G$-algebra of $\aX$. 
\end{rem}

\begin{cor}\label{StrictEquiv}
Let $\aY$ be a $\sD_G$-algebra. Then the pseudonatural isomorphism $\chi$ of \autoref{TheSection} induces an invertible pseudomorphism
\[ \chi\colon \aY \iso \xi_G^* \oursectionG^* \aY\]
of weak $\sD_G$-pseudoalgebras (see \autoref{Cpseudo}) whose components are identity maps.
\end{cor}

We define multimorphisms of weak $\sF_G$-pseudoalgebras following \autoref{MultiDG}, but deleting its strictness conditions with respect to $\PI_G$. Note that $\smaD$ is just $\sma$ on $\sF_G$.  Recall that we are working in $\GUCatp$ in this section.

\begin{defn}\label{MultiF}
We define a (symmetric) multicategory $\Mult(\psFGAlg)$ of (weak) $\sF_G$-pseudoalgebras whose underlying category is $\psFGAlg$, as follows.
For weak $\sF_G$-pseudoalgebras $\aX_i$, $1\leq i\leq k$, and $\aY$, a $k$-ary morphism
$\ul{\aX} \rtarr \aY$ consists of a $\sF_G^{\esma k}$-pseudomorphism
\[\squiggly{F\colon \aX_1\overline{\sma} \dots \overline{\sma} \aX_k}{\aY\circ \sma^k}.\]
Composition and the symmetric group action are given by the corresponding pasting diagrams, as done explicitly in \autoref{MultiD}.
\end{defn} 

The unpacking of this definition is similar to the unpackings given for Definitions \ref{MultiD} and \ref{MultiDG}, with the caveat that the coherence diagrams in \autoref{cohMultD} must account for the pseudofunctoriality constraints of the weak $\sF_G$-pseudoalgebras.

\begin{rem}  We could define an analogous multicategory $\psFAlg$ of weak structures, but we would not have a
prolongation multifunctor $\bP\colon \psFAlg\rtarr \psFGAlg$ since we would no longer have the compatibility with $\PI$
that we used in the proof of \autoref{PMultiFun}.  However, naturally occurring examples of $\sF$-pseudoalgebras that do not arise from use of the section 
often do have such compatibility with $\PI$ and thus can be prolonged to $\sF_G$-pseudoalgebras.
\end{rem}

We have the following theorem, whose proof is essentially the same as that of \autoref{pullingback}. Recall \autoref{notn:pullingback}.
\begin{thm}\label{sectMulti} Pullback along $\oursectionG$ induces a symmetric multifunctor 
\[\oursectionG^* \colon \Mult(\sD_G) \rtarr \Mult(\psFGAlg).\]
\end{thm}

\section{Strictification of pseudoalgebras}\label{sec:PowerLack}

For clarity about what is general and what is special and also for simplicity of notation, we revert to a general $\sV$ satisfying our standard assumptions in this section.  The reader may prefer to focus on $\sV=\sU$ or $\sV=G\sU$, but equivariance and topology play no role in this section.

\subsection{Power-Lack strictification}\label{sec:PowerLackFG} 

We here specialize a general result of Power and Lack \cite{Lack,Power} (see also \cite[Theorem 0.1]{AddCat1})
about strictification of pseudoalgebras over a 2-monad.  Let $\cC$ be a $\sV_\bpt$-2-category, and recall the 2-categories $\CAlg$  and $\CPsAlg$  (\autoref{CatCAlg}) of $\cC$-algebras and weak $\cC$-pseudoalgebras. We prove the existence of a strictification 2-functor $\St\colon \CPsAlg\rtarr \CAlg$.
While the general result specializes to give everything we need, 
we prefer to give an independent account that translates out from the theory of $2$-monads and includes an explicit construction.  We are implicitly
using the existence of an enhanced factorization system on $\VCatp$ \cite[\S4]{AddCat1} to apply Power and Lack's result to our context.

\begin{thm}\label{ConjOut1} 

Let $\cC$ be a $\sV_\bpt$-2-category.
 The inclusion of $2$-categories 
\[ \bJ\colon \CAlg \rtarr \CPsAlg \]
has a left 2-adjoint 
\[\St_{\cC} \colon \CPsAlg \rtarr \CAlg.\] 
The component of the unit $i$ of the $2$-adjunction is a $\cC$-pseudomorphism which is a levelwise equivalence (\autoref{defn:leveleqv}).
\end{thm}
For brevity, we will omit the subscript $\cC$ and simply write $\St$ for the strictification functor when the category $\cC$ is clear from the context.

\begin{pf}
For a weak $\cC$-pseudoalgebra $(\aX,\tha,\varphi)$, we define the (strict) $\cC$-algebra $\St \aX$ as follows.  

For $c$ in $\cC$, the $\sV_{\bpt}$-category  $\St \aX(c)$  has $\sV_\bpt$-object of objects
\begin{equation}\label{ObSt}
 \ob (\St \aX(c)) = \bigvee_{b}\ob\cC(b,c) \sma \ob (\aX(b)),
 \end{equation}
where $b$ ranges over all objects of $\cC$.

The action of $\cC$ on $\aX$ defines a morphism
\begin{equation}\label{mOb}
 \tha\colon \ob (\St \aX(c)) \rtarr \ob (\aX(c)),
  \end{equation}
and this allows us to define the 
 $\sV_\bpt$-object 
of morphisms of $\St \aX(c)$ as the pullback displayed in the diagram
\begin{equation}\label{MorSt}
\xymatrix{
\mor (\St \aX(c)) \ar[r] \ar[d]_{(T,S)} & \mor (\aX(c)) \ar[d]^{(T,S)}\\
 \ob (\St \aX(c)) \times  \ob (\St \aX(c)) \ar[r]_-{\tha \times \tha} & \ob (\aX(c) ) \times \ob (\aX(c)).
 }
 \end{equation}
Composition is induced by composition in $\aX(c)$. 

To describe the pullback more explicitly, writing elementwise for the sake of exposition, let 
\[ (f,x) \in \ob\cC(b,c) \sma \ob (\aX(b)) \ \ \text{and} \ \ (f',x') \in \ob\cC(b',c) \sma \ob (\aX(b')).\] 
Then 
\[ \mor(\St\aX(c))\Big((f,x),(f',x')\Big) = \mor(\aX(c))\Big(\tha(f,x),\tha(f',x')\Big). \]
The action map
\[\St \tha \colon \cC(b,c)\sma \St \aX(b) \rtarr \St \aX(c)\]
descends to the smash product from the map on the product given on objects by $(\St \tha) (h, (f,x)) = (h\com f, x)$ for $h\in \cC(b,c)$.

For a morphism $\al:(f,x) \rtarr (f',x')$ in $\St\aX(b)$ and $\la\colon h\rtarr h'$ in $\cC(b,c)$, the corresponding morphism $\St\tha(h,(f,x)) \rtarr \St\tha(h',(f',x'))$ is defined to be 
\[ \tha(h\com f,x)\xrtarr{\varphi^{-1}}\tha(h,\tha(f,x)) \xrtarr{\tha(\la,\al)} \tha(h',\tha(f',x')) \xrtarr{\varphi} \tha (h'\circ f,x).\]
Then $\St\tha$ gives $\St\aX$ a strict $\cC$-algebra structure by the strict functoriality of composition in $\cC$.

For a $\cC$-pseudomorphism $\squiggly{(F,\de)\colon \aX} { \aY}$, 
we define $\St (F,\de)\colon \St \aX \rtarr \St \aY$ by letting $\St(F,\de)_c:\St \aX(c) \rtarr \St \aY(c)$ be the functor sending $(f,x)$ to $(f,F x)$ and $\alpha\colon (f,x)\rtarr (g,y)$ to the composite
\[ \tha_Y(f,F x) \xrtarr{\de^{-1}} F(\tha_X(f,x)) \xrtarr{F \alpha}F(\tha_X(g,y)) \xrtarr{\de} \tha_Y(g,F y).\]
It is straightforward to check that these form the components of a strict $\cC$-morphism.
We can similarly define the action of $\St$ on  $\cC$-transformations.

For a weak $\cC$-pseudoalgebra $(\aY,\tha,\varphi)$, 
we define $\cC$-pseudomorphisms
\[\squiggly{i\colon \aY }{\, \St\aY } \qquad \text{and} \qquad \squiggly{m \colon \St \aY }{\, \aY}\]
as follows. The component $\sV_\bpt$-functor $i_c\colon \aY(c)\rtarr \St\aY(c)$  is given by $i_c(y) = (\id_c, y)$ on objects $y$ of $\aY(c)$ and $i_c(\ga) = \ga$ on morphisms $\ga$ of $\aY(c)$. The latter makes sense since $\tha(\id_c,y)=y$.  The component of the pseudonaturality $\sV_\bpt$-transformation
\[ \xymatrix{
\cC(b,c) \sma \aY(b) \ar[r]^{\id\sma i_b} \ar[d]_\tha 
\drtwocell<\omit>{ \hspace{1em}  i_{b,c}} & \cC(b,c) \sma \St \aY(b) \ar[d]^{\St \tha} \\
\aY(c) \ar[r]_{i_c} & \St \aY(c)
}\]
at $(f,y)$ is the morphism $(f,y) \rtarr (\id_c,\tha(f,y))$ in $\St \aY(c)$ corresponding to the identity map of $\tha(f,y)$ in $\aY(c)$.

The component $\sV_\bpt$-functors $m_c\colon \St\aY(c)\rtarr \aY(c)$ are given by  \autoref{mOb} on the $\sV_\bpt$-object of objects  and by the top horizontal arrow in \autoref{MorSt} on the $\sV_\bpt$-object of morphisms.  The pseudonaturality constraint
\[ \xymatrix{
\cC(b, c) \sma \St\aY(b) \ar[r]^-{\id\sma m_b} \ar[d]_{\St\tha}   \drtwocell<\omit>{ \hspace{.5em} \tilde{\varphi}}  &  \cC(b,c) \sma \aY(b) \ar[d]^{\tha} \\
\St\aY(c) \ar[r]_-{m_c} & \aY(c)
} \]
is induced by the invertible $\sV_\bpt$-transformation $\varphi$ that witnesses the $\cC$-pseudoalgebra structure of $\aY$.

As is easily checked directly, $i$ and $m$ are inverse equivalences in $\CPsAlg$, so in particular, $i$ is a level equivalence. The map $i$ is 2-natural with respect to $\cC$-pseudomorphisms and $\cC$-transformations and is the unit of the adjunction. If $\aY$ is a strict $\cC$-algebra and thus of the form $\aY=\bJ\aX$,
then $m$ is 
a strict $\cC$-morphism, and moreover, it is 2-natural with respect to strict $\cC$-morphisms and $\cC$-transformations; it is the counit of the adjunction.
\end{pf}

\begin{rem}\label{mCounit}
 The $\cC$-pseudomorphism $\squiggly{m\colon \St \aY }{\, \aY}$ of the proof of \autoref{ConjOut1} is not strictly natural with respect to $\cC$-pseudomorphisms of weak $\cC$-pseudoalgebras. With slightly extra work, one can prove that it gives the components of a pseudonatural transformation $m\colon \bJ \St \Longrightarrow \id$ of endo-2-functors of $\CPsAlg$. As it is not necessary for our work, we shall not pursue this route. However, we will use the fact, noted in the proof, that $m_c\colon \St \aY (c) \rtarr \aY(c)$ is an equivalence of $\sV_\bpt$-categories.
\end{rem}

Recall \autoref{notn:pullingback}. We shall apply the following result to an iterated monoidal product $ \cC^{\sma k}\rtarr \cC$ in \autoref{sec:StrictMult} and to $\xi_G\colon \sD_G \rtarr \sF_G$ in \autoref{HtpySect}.

\begin{lem}\label{zeta1}  Let $\cD$ and $\cC$ be $\sV_\bpt$-2-categories and let $\xi\colon \cD\rtarr \cC$ be a $\sV_\bpt$-2-functor. Given a 
$\cC$-pseudoalgebra $\aY$,
there is a $\cD$-morphism  
\[ \ps=\ps_\xi\colon \St_\cD (\xi^*\aY) \rtarr \xi^*(\St_\cC \aY) \] 
that is 2-natural with respect to $\cC$-pseudomorphisms and $\cC$-transformations.  Moreover, the diagram
\[  \xymatrix{ & \xi^*\aY \ar@{~>}[dl]_{i_{\cD}} \ar@{~>}[dr]^{\xi^*i_{\cC}}  & \\
\St_\cD (\xi^*\aY) \ar[rr]_{\ps} & &  \xi^*(\St_\cC \aY)\\}  \]
commutes, hence $\ps$ is a levelwise equivalence.  
\end{lem}

\begin{proof}  
We can specify $\psi$ as the $\cD$-morphism corresponding to the $\cD$-pseudomorphism
\[ \xi^* i_\cC \squiggly{\colon \xi^* \aY}{\xi^*(\St_\cC \aY)}\]
under the adjunction of \autoref{ConjOut1}.
The claim about 2-naturality then follows from the 2-naturality of $i$. The commutativity of the diagram can be verified directly from the definition. 
\end{proof}

\begin{rem} For later use, we give an explicit description of $\ps$ in terms of elements. For an object $d\in \cD$, $\psi_d \colon \St_\cD(\xi^* \aY) (d) \rtarr \xi^*(\St_\cC \aY)(d)$ sends an object $(f,y)$, with $f\colon d'\to d$ in $\cD$ and $y \in \aY(\xi(d'))$, to the object $(\xi(f),y)$.
\end{rem}

The following lemma records the compatibility of the morphism $\psi$ of \autoref{zeta1} with composition of $\sV_\bpt$-2-functors; it follows directly from the definitions.

\begin{lem}\label{psi-comp}
Let $\nu\colon \cE \rtarr \cD$ and $\xi \colon \cD \rtarr \cC$ be $\sV_\bpt$-2-functors. Then the diagram 
\[
 \xymatrix{
 \St_\cE((\xi\nu)^* \aY ) \ar[rr]^-{\psi_{(\xi\nu)}} \ar@{=}[d] && (\xi\nu)^*(\St_\cC \aY) \ar@{=}[d]\\
 \St_\cE(\nu^*\xi^* \aY ) \ar[r]_-{\psi_{\nu}} & \nu^* (\St_\cD (\xi^* \aY)) \ar[r]_-{\nu^* \psi_{\xi}} & \nu^*\xi^*(\St_\cC \aY)
 }
\]
commutes for every $\cC$-pseudoalgebra $\aY$.
\end{lem}

The following observation about the interaction of strictification with the external smash product (\autoref{PairsPairs}) generalizes \cite[Lemma 3.5]{MayPair}.

\begin{lem}\label{Sprod}
Let $\cC$ and $\cD$ be $\sV_\bpt$-2-categories, $\aX$ a $\cC$-pseudoalgebra and $\aY$ a $\cD$-pseudoalgebra. Then there is a canonical isomorphism 
\[  \St_{\cC\esma \cD} (\aX \bsma\aY)\cong (\St_\cC \aX) \bsma (\St_\cD \aY),\] 
which is $2$-natural with respect to the respective pseudomorphisms and pseudo\-transformations. 
In particular, up to composing with this canonical isomorphism, for a $\cC$-pseudomorphism $\squiggly{E\colon \aX}{\aX'}$ and an
$\cD$-pseudomorphism $\squiggly{F\colon \aY}{\aY'}$, the $\cC\esma \cD$-morphism $\St_{\cC\esma \cD} (E\bsma F)$  corresponds to $(\St_\cC E) \bsma( \St _\cD F)$.
\end{lem}

\begin{proof}  On the level of objects, the identification
follows from \autoref{ObCommutesSma} and \autoref{ObSt}.
Again writing elementwise, 
it is just the twist that sends an object $((e,f),(x,y))$ of $\St_{\cC\esma \cD} (\aX\bsma\aY)(c,d)$ to the object
$((e,x), (f,y))$ of  $(\St_\cC \aX)(c) \sma (\St_\cD \aY)(d)$.  

Note that since the $\sV_\bpt$-functor
\[m_\cC \sma m_\cD \colon (\St_\cC \aX(c))\sma (\St_\cD \aY(d)) \to \aX(c)\sma \aY(d)\]
is an equivalence (\autoref{mCounit}), it is fully faithful \cite[Lemma~4.17]{InternalCats}, 
in the sense  that the diagram
\[
\xymatrix{
\mor \Big((\St_\cC \aX(c))\sma (\St_\cD \aY(d))\Big) \ar[r] \ar[d]_{(T,S)} & \mor (\aX(c)\sma \aY(d)) \ar[d]^{(T,S)}\\
 \ob \Big((\St_\cC \aX(c)) \sma (\St_\cD \aY(d))\Big)^{\times 2} \ar[r] & \ob (\aX(c)\sma \aY(d) )^{\times 2} 
 }
 \]
is a pullback square which is  isomorphic to the pullback square in \autoref{MorSt} that defines 
$\mor(\St_{\cC\esma \cD}(\aX \bsma\aY))$.   
This shows that the $\sV_*$-categories $\St_{\cC\esma\cD} (X\overline{\sma}Y)(c,d)$ and $(\St_\cC X)(c)\sma (\St_\cD Y)(d)$ are isomorphic. One can check that these isomorphisms respect the action of $\cC \sma \cD$, thus proving the result.
\end{proof}

We record the relationship between the 2-natural transformation $\ps$ of \autoref{zeta1} and the canonical isomorphism 
of \autoref{Sprod}.

\begin{lem}\label{ZEandtimes}
 Let $\cC$, $\cC'$, $\cD$, and $\cD'$,  be $\sV_\bpt$-2-categories,  let
  $\xi \colon \cC'\rtarr \cC$ and $\ze\colon \cD'\rtarr \cD$ be $\sV_\bpt$-2-functors, and let
  $\aX$ be a $\cC$-pseudoalgebra and $\aY$ be a $\cD$-pseudoalgebra.  
 Then the following diagram commutes, where the unnamed isomorphisms are those of \autoref{Sprod}.
 
 \[\xymatrix{
 \St_{\cC'\esma\cD'} \Big((\xi \sma \ze )^*(\aX \bsma\aY) \Big)\ar[r]^-{\ps_{\xi \sma \ze}} \ar@{=}[d] &  (\xi\sma \ze)^* \St_{\cC\esma\cD} (\aX \bsma \aY) \ar[d]^{\cong}\\
  \St_{\cC'\esma\cD'} \Big((\xi ^* \aX )\bsma(\ze^*\aY) \Big)  \ar[d]_{\cong} & (\xi\sma \ze)^* (\St_\cC \aX \bsma \St_\cD\aY) \ar@{=}[d]\\
 (\St_{\cC'} \xi^*\aX) \bsma (\St_{\cD'} \ze^*\aY) \ar[r]_-{\ps_{\xi} \bsma \ps_{\ze}} & (\xi^*\St_\cC \aX)\bsma (\ze^*\St_\cD \aY)\\} \] 
 \end{lem}

\subsection{The extension of strictification to a multifunctor}\label{sec:StrictMult}

We now assume that $(\cC,I,\smaD,\ta)$ is a 
permutative $\sV_\bpt$-2-category (\autoref{permv2cat}).  We extend $\St$  to a multifunctor from the multicategory $\Mult(\CPsAlg)$ of (weak) $\cC$-pseudoalgebras to the multicategory $\Mult(\CAlg)$ of strict $\cC$-algebras and strict multilinear maps.  The former is defined by replacing $\sF_G$ by $\cC$ in \autoref{MultiF},  and we now define the latter.

\begin{defn}
The multicategory $\Mult(\CAlg)$ is constructed by using the symmetric monoidal structure on $\CAlg$ given by Day convolution along iterations of the monoidal product $\smaD$ on $\cC$. Reinterpreting this externally, via the universal property of Day convolution,  we see that for strict
$\cC$-algebras $\aX_1, \dots, \aX_k$ and $\aY$, a $k$-ary morphism is given by a strict $\cC^{\sma k}$-morphism  
\[F \colon \aX_1 \bsma \cdots \bsma \aX_k \rtarr  \aY\com \smaD^k.\]
\end{defn}

We remark that the definition above differs from Definitions \ref{MultiD} and \ref{MultiDG} in that the multimorphisms are pseudotransformations in those cases, and strict here. 

\begin{thm}\label{StMulti} The strictification functor $\St $ induces a (non-symmetric) multifunctor 
\[\St \colon \Mult(\CPsAlg) \rtarr \Mult(\CAlg).\]
\end{thm}

\begin{proof}  
For a $\cC$-pseudoalgebra $\aX$, 
we set
$\St\aX=\St_\cC\aX$. For multimorphisms, we use the 2-functor $\St_{\cC^{\esma k}}$, which for brevity we denote by $\St_k$, to strictify $\cC^{\esma k}$-pseudomorphisms. More precisely, recall that a $k$-ary morphism in $\Mult(\CPsAlg)$ is given by a $\cC^{\esma k}$-pseudomorphism
\[ \squiggly{F\colon \aX_1 \bsma \cdots \bsma \aX_k }{\, \aY\circ \smaD^k.} \]
We define the  $k$-ary morphism $\St (F)$ in $\Mult(\cC)$  as the composite
\[\xymatrix{
\St\aX_1\bsma \cdots \bsma \St \aX_k \ar[r]^-{\cong} & \St_k (\aX_1\bsma \cdots \bsma \aX_k) \ar[r]^-{\St_k(F)} &  \St_k (\aY\com \smaD^k) \ar[r]^-{\ps} & (\St \aY)\com \smaD^k,}
\]
where the unnamed isomorphism is that of \autoref{Sprod}, and the map $\ps$ is the one defined in \autoref{zeta1}.

Note that on 1-ary morphisms, $\St$ is just the 2-functor $\St_\cC$, so in particular this assignment sends the identity to itself.

It remains to prove that $\St$ preserves multicomposition. Let 
\[ F\colon \squiggly{\aY_1\bsma \cdots \bsma \aY_k }{\, \aZ\com \smaD^k } \] 
be a $\cC^{\esma k}$-pseudomorphism
and,   for  $1\leq r\leq k$, let
\[ E_r\colon \squiggly{\aX_{r,1} \bsma \cdots \bsma \aX_{r,j_r} }{\, \aY_{r}\com \smaD^{j_r} } \] 
be a $\cC^{\esma j_r}$-pseudomorphism.  
The composite in $\Mult(\CPsAlg)$ is given by the pasting diagram in \autoref{compMultiD}. In terms of the external smash product of \autoref{PairsPairs} and the pullback of \autoref{notn:pullingback}, this $\cC^{\sma j}$-pseudomorphism can be expressed as the composite of $\cC^j$-pseudomorphisms
\[
\xymatrixcolsep{1.5cm}\xymatrix{
\overline{\bigwedge\limits_{r,i}} \aX_{r,i} \ar@{~>}[r]^-{\overline{\bigwedge\limits_{r}} E_r} & \overline{\bigwedge\limits_{r}} (\smaD^{j_r})^*\aY_{r} = (\bigwedge\limits_r \smaD^{j_r})^* (\overline{\bigwedge\limits_r} \aY_r) \ar@{~>}[r]^-{(\bigwedge\limits_r \smaD^{j_r})^*F} & (\bigwedge\limits_r \smaD^{j_r})^*(\smaD^k)^*\aZ=(\smaD^j)^*\aZ.
}
\]

Consider the following diagram of $\cC^{\sma j}$-morphisms.
\[
\xymatrixcolsep{1.3cm}\xymatrix{
 \overline{\bigwedge\limits_{r,i}}\St \aX_{r,i}  \ar[r]^-{\cong}  \ar[d]_-{\cong} 
 & \overline{\bigwedge\limits_r}\St_{j_r}\Big(\overline{\bigwedge\limits_i}\aX_{r,i} \Big) \ar[r]^-{\overline{\bigwedge\limits_r}\St_{j_r}(E_r)} \ar[dl]^-{\cong}
 & \overline{\bigwedge\limits_r} \St_{j_r}\Big((\smaD^{j_r})^*\aY_r\Big) \ar[d]^-{\overline{\bigwedge\limits_r}\ps_{\smaD^{j_r}} } \ar[ddll]^-{\cong}\\
  \St_j\Big( \overline{\bigwedge\limits_{r,i}}\aX_{r,i}\Big) \ar[d]_-{\St_j\Big(\overline{\bigwedge\limits_r}E_r\Big)}
  && \overline{\bigwedge\limits_r} (\smaD^{j_r})^*\St\aY_r \ar@{=}[d]\\
 \St_j\Big(\overline{\bigwedge\limits_r} (\smaD^{j_r})^* \aY_r \Big)\ar@{=}[d]
  && (\bigwedge\limits_r \smaD^{j_r})^* \Big(\overline{\bigwedge\limits_r} \St \aY_r\Big)\ar[d]^-{\cong}\\
  \St_j\Big((\bigwedge\limits_r \smaD^{j_r})^*(\overline{\bigwedge\limits_r}  \aY_r) \Big)\ar[d]_-{\St_j\big((\bigwedge\limits_r \smaD^{j_r})^*F\big)} \ar[rr]^-{\ps_{(\bigwedge \smaD^{j_r})}}
  && (\bigwedge\limits_r \smaD^{j_r})^* \St_k\Big(\overline{\bigwedge\limits_r}  \aY_r\Big)\ar[d]^-{(\bigwedge\limits_r \smaD^{j_r})^*\St_k F}\\
  \St_j\Big( (\bigwedge\limits_r \smaD^{j_r})^*(\smaD^k)^*\aZ\Big) \ar@{=}[d] \ar[rr]_-{\ps_{(\bigwedge \smaD^{j_r})}}
  &&  (\bigwedge\limits_r \smaD^{j_r})^*\St_k\Big((\smaD^k)^*\aZ\Big) \ar[d]^-{(\bigwedge\limits_r \smaD^{j_r})^*\ps_{\smaD^k}}\\
  \St_j\Big((\smaD^j)^*\aZ\Big) \ar[r]_{\ps_{\smaD^j}}
  & (\smaD^j)^*\St\aZ \ar@{=}[r]
  & (\bigwedge\limits_r \smaD^{j_r})^*(\smaD^k)^*\St\aZ
}
\]
Strictifying $E_1,\dots,E_r$ and $F$ and then composing is equal to going around clockwise. Using that $\St_j$ is a 2-functor, we get that going around counter-clockwise is equal to strictifying the composite. The diagram commutes; indeed, going from top to bottom, the regions commute by associativity of $\sma$ in $\CatVp$, \autoref{Sprod}, \autoref{ZEandtimes}, naturality of $\ps$, and \autoref{psi-comp}, respectively. 
\end{proof}

\begin{rem}  This proof depends crucially on the fact that the monoidal product on $\cC$ is a strict $\sV_\bpt$-2-functor and not just a pseudofunctor.  It does not work for  our pseudo-commutative categories of operators $\sD$ or $\sD_G$.  This is the crux of why the route in this paper is less categorically intensive than the monadic route of \cite{MayToBe}, which simultaneously strictifies and transfers structure from $\sD_G$ to $\sF_G$.
\end{rem}

\subsection{$\St$ is not a symmetric multifunctor}\label{sec:nono}
As stated in \autoref{StMulti}, $\St$ is not a  {\em symmetric} multifunctor.  We explain why in this parenthetical subsection. 
We consider the case when $k =2$, with $\si$ the non-trivial element of $\SI_2$.   Since the problem already appears nonequivariantly, we take $G=e$ and specialize to $\cC=\sF$.  Thus let 
$F\colon \squiggly{\aX_1 \sma\aX_2}{\, \aY}$ be a (weak) $\sF$-pseudomorphism between (weak) $\sF$-pseudoalgebras.  Following \autoref{ObjSym}, for an object $(\bm,\bp)$ of $\sF \esma \sF$, the 1-cell component of $F\si$ is the composite
\[
\xymatrix{
\aX_2(\bm)\sma \aX_1(\bp) \ar[r]^-{t} & \aX_1(\bp)\sma \aX_2(\bm)  \ar[r]^-{F} & \aY(\bp\bm) \ar[r]^-{\aY(\ta_{p,m})} & \aY(\bm\bp).
}
\] 

We claim that the 1-cell components of $\St(F\si)$ and  $(\St F)\si$ do not agree. For an object $(\bm,\bp)$ of $\sF \esma \sF$, these are $\sV_\bpt$-functors
\[\St \aX_2(\bm) \sma \St \aX_1(\bp)\rtarr \St\aY(\bm\bp).\]
We compare them at the level of objects, writing elementwise. 
An object of the source has the form
\[ \Big( (f_2, x_2), (f_1, x_1)\Big) \]
where $f_1\colon \bq\rtarr \bp$ and $f_2\colon \bn \rtarr \bm$ are morphisms of $\sF$ and  $x_1$ and $x_2$ are objects of 
$\aX_1(\bq)$ and $\aX_2(\bn)$, respectively.  

Then $\St (F)\si$ sends $\Big((f_2, x_2), (f_1, x_1)\Big)$ to $\Big((\ta_{p,m} \circ (f_1\sma f_2),F(x_1,x_2)\Big)$.
We can rewrite 
the output as 
$$\Big((f_2\sma f_1)\circ \tau_{q,n},F(x_1,x_2)\Big).$$
On the other hand, $\St(F\si)$ sends $\Big((f_2, x_2), (f_1, x_1)\Big)$ to 
$$\Big( (f_2\sma f_1),\tha(\ta_{q,n},F(x_1,x_2))\Big).$$
We conclude that the multifunctor $\St$ is not symmetric.

\begin{rem}\label{StfNotSt}  This failure of symmetry is forced by our need to use weak pseudo\-structure in the target of $\oursectionG^*$. If we instead use pseudofunctors which are strict relative to $\PI$ when we strictify,
{for example using the generalized strictification theorem given in \cite{AddCat2}},
 then we do have symmetry. 
Symmetry is also studied in the $2$-monadic context in \cite{MayToBe}, where the problem is entirely different.
\end{rem}

\section{From $\sF_G$-algebras in Cat to  $G$-spectra} 
\label{sec:CattoTop}

In this section, we describe how we pass from categorical to topological $\sF_G$-algebras and $\sD_G$-algebras, and then to $G$-spectra, keeping track of multiplicative structure.   Nonequivariantly, $\sF$-spaces (aka $\Gamma$-spaces) were introduced by Segal in his treatment of infinite loop space theory. These generalize to $\sF_G$-$G$-spaces, which are the input of the equivariant version of the Segal infinite loop space machine. For a detailed treatment of $\sF_G$-$G$-spaces, we refer the reader to \cite{MMO}, and for a treatment of its symmetric monoidal structure to \cite{GMMO}.  Topological categories of operators $\sD$ and $\sD$-spaces were introduced in \cite{MT} as an intermediary between $\sF$-spaces and operadic algebras in the proof of the uniqueness of infinite loop space machines. The topological equivariant analogues, $\sD_G$-spaces, are treated extensively in \cite{MMO} in the comparison of equivariant infinite loop space machines. 

In \autoref{sec:HaveB}, we discuss the classifying space functor multiplicatively. We recall the equivariant Segal machine in \autoref{sec:HaveSegal}. Using that the classifying space functor and the Segal machine are both lax monoidal, we restate and prove  \autoref{IntroMultiKG} as \autoref{KGMultiFun}.  In effect, it gives a multiplicative equivariant infinite loop space machine starting from operadic categorical input. 

Some technicalities ensuring that our passage from categorical to space-level input is homotopically well-behaved are postponed to \autoref{nondegen}. The point is just to give conditions on the categorical input that ensure that the output $\sF_G$-$G$-spaces have nondegenerate basepoints.  We briefly discuss a related open question about Day convolution in  \autoref{DayTop}. The brief \autoref{homotopy} shows how to obtain homotopies between maps of $G$-spectra from operadic categorical input.

\subsection{The multifunctor $B$}\label{sec:HaveB}
In order to construct equivariant spectra from $\oO$-algebras in a multiplicative way, we need to understand the multiplicative properties of the classifying space functor $B$.

The classifying space functor $B$ does not commute with  smash products in general.  However, we have the following result, which allows us to use $B$ to change enrichment.  Recall \autoref{Catstar}.

\begin{prop}\label{BLaxMon}
 The classifying space functor $B\colon \GUCatp \rtarr G\Top_\bpt$ is lax symmetric monoidal.
\end{prop}
\begin{proof}
The map
\[B\cC\times B\cD \iso B(\cC\times \cD) \rtarr B(\cC\sma \cD)\]
sends the subspace $B\cC \vee B\cD$ to the basepoint and therefore induces a based map 
\[B\cC \sma B\cD \rtarr B(\cC \sma \cD).   \qedhere\]
\end{proof}

\begin{defn}\label{DGGU}
Let $\sD_G$ be a $\GUCatp$-category of operators over $\sF_G$, as defined in \autoref{reducedGCO/FG}. Let $\sD_G^{top}$ denote the category enriched in $G\Top_*$ obtained by applying $B$ to morphism based categories to change the enrichment.  When  $\sD_G = \sF_G$, the morphism categories are discrete (identity morphisms only).  Since the classifying space of a discrete category is isomorphic to itself, we can identify $\sF_G^{top}$ with $\sF_G$.  It follows that $\sD_G^{top}$ is a category of operators over $\sF_G$ in the sense of \cite[\S 4.2]{MMO}.
\end{defn}

\begin{notn}
Since the category $G\Top_*$ is closed monoidal, it is enriched over itself. We denote 
this enriched category by $\enGU$. Its based $G$-spaces of morphisms are given by the spaces of all nonequivariant based maps,  based at the constant maps at the basepoint, with $G$ acting by conjugation.  We denote by  $\GDU$ the category of $\GTop_*$-enriched functors $\aX \colon \sD_G^{top} \rtarr \enGU$. The enrichment over based $G$-spaces implies that $\aX(0) = \ast$ \cite[Lemma 1.13]{MMO}. 
To emphasize that these are just (enriched) functors to $G$-spaces, we call them $\sD_G^{top}$-$G$-spaces.   In particular, 
an \ourdefn{$\sF_G$-$G$-space} will mean an object of $\GFU$.

\end{notn}

Recall that we write $\sD_G$-$\mathbf{Alg}$ as shorthand for the category of $\sD_G$-algebras and strict maps in $\GUCatp$. 

\begin{prop}\label{Blevelwise}
Applying the classifying space functor levelwise induces a functor
\[B\colon \sD_G\text{-}\mathbf{Alg} \rtarr \GDU.\]
\end{prop}

\begin{proof}
By \autoref{BLaxMon} the classifying space functor $B$ is lax symmetric monoidal. It follows formally that it induces a map on functor categories.
Explicitly, if $\aX$ is a $\sD_G$-algebra in $\GUCatp$, we obtain the $\sD_G^{top}$-$G$-space $B\aX$ by applying $B$ levelwise, with action maps given by the composites
\[ B\sD_G(\bm^\al,\bn^\be) \sma B\aX(\bm^\al) \rtarr B(\sD_G(\bm^\al,\bn^\be) \sma \aX(\bm^\al)) \xrightarrow{B\tha} B\aX(\bn^\be), \]
where the first map is the monoidal constraint for $B$.  The commutativity of the composition and unit diagrams follows from their analogs for $\aX$ (see \autoref{strict}) and the axioms for a lax monoidal functor.   The functoriality of $B$ on strict algebra maps is obtained by applying $B$ levelwise.
\end{proof}

We now concentrate on the case of $\sF_G$. The categories $\GFU$ and $\FGAlg$ are symmetric monoidal via Day convolution.

\begin{prop}\label{BMultiFun}
The functor
\[ B \colon \FGAlg \rtarr \GFU\]
of \autoref{Blevelwise} is lax symmetric monoidal.
\end{prop}
\begin{proof}
This follows formally from \autoref{BLaxMon},  but we sketch the argument. The induced functor $B$ on our categories of algebras preserves the monoidal unit, which in both the source and the target is given by the representable functor $\sF_G(\bf{1},-)$. Given $\aX$ and $\aY$ in $\FAlg$, we construct a map
\[B\aX \sma B\aY \rtarr B(\aX \sma \aY)\]
in $\GFU$ by applying the universal property of Day convolution to the map of $(\sF_G\esma \sF_G)$-$G$-spaces with components given by the composites
\[B\aX(\bm^\al) \sma B\aY(\bn^\be) \rtarr B(\aX(\bm^\al)  \sma \aY(\bn^\be)) \rtarr B(\aX \sma \aY)(\bm\bn^{\al \otimes \be}).
\]
Here the first map is the lax monoidal constraint for $B$ and the second map is obtained by applying $B$ to the components of the unit of the Day convolution adjunction. It is routine to check that this map satisfies the required compatibilities with the unit, associativity and symmetry isomorphisms. 
\end{proof}

\begin{rem}
Note that if $\sD_G$ is a $\GUCatp$-category of operators over $\sF_G$ equipped with a pseudo-commutative structure (\autoref{pseudocomDG}), this does {\it not} give rise to a symmetric monoidal structure on $\sD_G^{top}$. As a result, we do not have a monoidal structure on the category of $\sD_G^{top}$-algebras, and so we cannot expect an analogue of \autoref{BMultiFun} for $\sD_G$-algebras.
\end{rem}

\subsection{From $\sF_G$-$G$-spaces to $G$-spectra}
\label{sec:HaveSegal}

In this section, we first recall the properties of the equivariant Segal machine, whose construction is given in detail in \cite{MMO}. A treatment that deals with  multiplicative properties can be found in \cite{GMMO}. In this paper, we treat the Segal machine as a black box, and we  refer the reader to those sources for details. 

All homotopical versions of the Segal machine come in the form of bar constructions, which are only homotopically well-behaved when the input functors
 $\aX\colon \sF_G\rtarr  \enGU$ take values in nondegenerately based $G$-spaces.   However, in the previous subsection, we concentrated on formal properties of our constructions.   Write  $G\sT$ and $\sT_G$ for the full subcategories of nondegenerately based $G$-spaces in $G\sU_{\ast}$ and in $\enGU$.  Since these categories are not bicomplete, they are less useful for formal purposes.  We introduce notations and definitions to help deal with the resulting dichotomy. 
 
 \begin{notn}  
Let $\sD_G^{top}$  be a $G\Top_*$-category of operators over $\sF_G$, such as the one in \autoref{DGGU}.   A $\sD_G^{top}$-$G$-space is  \ourdefn{levelwise nondegenerately based} if each $\aX(\bn^{\al})$ is nondegenerately based.  We write $\GDT$ for the full subcategory of $\GDU$ whose objects are levelwise nondegenerately based.  In particular, starting with the commutativity operad, whose terms are one-point $G$-spaces, this defines the full subcategory  $\GFT$ of $\GFU$.
 \end{notn}
   
\begin{defn}  Define $\bT$ to be the composite functor
\[\bT =  \St_{\sF_G} \circ \oursectionG^* \circ  \bR_G: \OAlg \rtarr \FGAlg. \] 
Then define  $\OAT$ to be the full subcategory of  $\OAlg$ consisting of those $\oO$-algebras $\cA$ such that $B \bT \cA$ is in $\GFT$.
Thus, by definition, the composite  $B\bT$ restricts to a functor $\OAT \rtarr \GFT$.  
\end{defn}

The functor $\bT$ collates the categorical functors studied in previous sections, $B$ passes from categorical data to space level data, and the Segal machine passes from there to spectra.  
 That machine will be well-behaved when we restrict it to $\GFT$, and $\OAT$ specifies those $\oO$ algebras that feed into $\GFT$.
We will show in the next subsection that most $\oO$-algebras of interest are in $\OAT$.

To define the notion of a Segal machine $\bS_G$, we need the key notion of a \emph{special}  \ourdefn{$\sF_G$-$G$-space}.  To give a conceptual setting for this notion, observe first that, just as we had on categories, we have a composite functor $\bR_G^{top} = \bP^{top} \bR^{top}$ from $\oO^{top}$-algebras in $G\sT$ 
to  $\sD_G^{top}$-$G$-spaces, where $\oO^{top}$ is a operad in $\GTop$ with associated category of operators $\sD_G^{top}$.
We specialize this to the initial operad $\oO^{top}$, which has $\oO^{top}(0)=\oO^{top}(1)=\ast$ and all other $\oO^{top}(j)=\emptyset$.  Its associated category of operators is $\PI_G$.  Applying $\bR_G^{top}$ to this operad,  we obtain a functor $\bR_G^{top}\colon G\sT \rtarr \GPT$.   

For a based $G$-space $X$, $\bR_G^{top}(X)$ sends $\bn^\al$ to the $G$-space $X^{\bn^{\al}} = G\sT(\bn^{\al}, X)$. More explicitly, this is 
$X^n$ with $G$-action given by
\begin{equation}
\label{RGaction}
  g(a_1,\dots,a_n) = (ga_{\alpha(g^{-1})(1)},\dots, ga_{\alpha(g^{-1})(n)}).
  \end{equation}
The functor $\bR_G^{top}$ is right adjoint to the functor $\bL_G^{top}$ that evaluates 
a $\PI_G$-$G$-space at $\bf{1}$. For a $\PI_G$-$G$-space $\aY$, 
the unit 
\[\de\colon \aY\rtarr \bR_G^{top}(\aY(\mb{1}))\]
of the adjunction is  a map of $\PI_G$-$G$-spaces 
 called the \ourdefn{Segal map}. At level $\bn^\al$, it is induced
by the  $n$ projections $\bn^\al\to \mathbf{1}$ \cite[Definition 2.28]{MMO}. 

\begin{defn}\label{defnSpecial}
We say that  a $\PI_G$-$G$-space  $\aY$ is  \ourdefn{special} if $\de$ is a levelwise weak $G$-homotopy equivalence.  
 We say that 
a $\sD^{top}_G$-$G$-space, and in particular an $\sF_G$-$G$-space,
is special if  
 its underlying $\PI_G$-$G$-space 
is special. 
\end{defn}

Recall that an orthogonal $G$-spectrum $E$ is a positive $\Omega$-$G$-spectrum if its adjoint structure maps 
\[ E_V \rtarr \Omega^W E_{V\oplus W}  \]
are weak $G$-equivalences when $V^G\neq 0$ and is connective if the negative homotopy groups of its fixed point spectra are all zero.

\begin{defn}\label{segalmachine} 
A \ourdefn{Segal machine} is a functor  $\bS_G \colon \GFU \rtarr \SpG$ together with a natural map of $G$-spaces 
\[\nu\colon \aX({\bf 1}) \rtarr (\bS_G\aX)_0\] 
such that the following properties hold when $\aX$ is a special $\sF$-$G$-space in $\GFT$. 
\begin{enumerate}[(i)]
\item $\bS_G\aX$ is a connective positive $\Omega$-$G$-spectrum. 

\item  The composite of $\nu$ with the adjoint structure map
\[  (\bS_G\aX)_0 \rtarr  \Omega^{V}(\bS_G\aX)_V\]  
is a group completion for all $V$ such that  $V^G\neq 0$.
\end{enumerate}
\end{defn}

\begin{rem}  It is equivalent to replace general $V$ by $V = \bR$ in (ii).
\end{rem}

The notion of a group completion of a Hopf $G$-space is defined as a group completion on all fixed point maps (see \cite[Definition 1.9]{GMPerm}).  The nonequivariant construction of the Segal machine was introduced in \cite{Seg}. The equivariant construction is due to Shimakawa \cite{Shim}, who started from  an unpublished version that is also due to Segal.  It is given a {self-contained} modernized  treatment in \cite{MMO}. A multiplicative version is given in \cite{GMMO}. We  refer the reader to those sources for details.

From now on, we set $\bS_G$ to be the Segal machine from \cite{MMO}, which is lax monoidal by \cite{GMMO}. 
We could just as well use the equivalent symmetric monoidal version from \cite{GMMO}, but that would not be of any benefit since we lost symmetry with the multifunctor 
$\St_{\sF_G}$.   Moreover, using the machine from \cite{MMO} will be convenient in \autoref{sec:BPQ}, where we will use several results from \cite{MMO}. 
We repeat that we mostly treat the Segal machine $\bS_G$ as a black box.  The only detail from \cite{MMO} that we will need to make explicit is a partial description of the construction that allows us to define the natural map {$\nu \colon \aX(1) \rtarr  (\bS_G\aX)_0$} required in \autoref{segalmachine}.   That will be given where it is used  in \autoref{sec:BPQ}.

We now restate and prove \autoref{IntroMultiKG}.  

\begin{thm}\label{KGMultiFun} 
Let $\oO$ be a chaotic $E_{\infty}$ $G$-operad in $\GUCat$. The functor
\begin{equation}\label{roadkill}
\bK_G = \bS_G \com B \com  \St_{\sF_G} \com \oursectionG^* \com \bR_G\colon \OAlg \rtarr  \SpG.
\end{equation}
from \autoref{RoadMap} extends to a multifunctor
\[ \bK_G\colon  \Mult(\oO) \rtarr {\Mult}(\SpG).\]
For an $\oO$-algebra $\cA\in \OAT$, $\bK_G \cA$ is a connective positive $\Omega$-$G$-spectrum 
with a group completion map 
\[B\cA \rtarr \Omega^V(\bK_G\cA)_V \] 
for all $V$ such that $V^G\neq 0$.
\end{thm}

\begin{proof} By \autoref{RGMulti}, \autoref{sectMulti}, \autoref{StMulti}, \autoref{BMultiFun}, and \cite[Section~5.2]{GMMO}, $\bK_G$ is a composition of multifunctors and is thus a multifunctor.  
When $\cA$ is in $\OAT$, $B\bT\cA$ is in $\GFT$, and we claim that it is special.  That  will imply the second statement.
Since $\bT\cA$  is level $G$-equivalent to $\bR_G\cA$, by  \autoref{ConjOut1}, the claim follows from the fact that $B$ takes equivalences of $G$-categories to 
homotopy equivalences of $G$-spaces and commutes with $\bR_G$, in the sense that  $B\bR_G\cong \bR_G^{top}B$.
\end{proof}

\subsection{The identification of objects in $\OAT$}
\label{nondegen}

When the operad $\oO$ and an $\oO$-algebra  $\cA$ are topologically discrete, in the sense that they are categories internal to $G\text{Set}$, $\cA$ is in $\OAT$  since all of our categorical constructions retain discreteness and the geometric realization of a based simplicial set is nondegenerately based. We show here that many topologically non-trivial examples, such as those that appear in \autoref{sec:BPQ}, are also in $\OAT$.  

We require the following definition.  Nonequivariantly, its use goes back at least to Milnor's classical paper \cite{MilCW}, and it was studied in more detail by Dyer and Eilenberg \cite{DE} and later Lewis \cite{LewisOMSIX}. Details of equivariant cofibrations are in \cite[Section A.2]{BVbook}.

\begin{defn}
 A $G$-space $X$ is \ourdefn{$G$-locally equiconnected} ($G$-LEC for short) if the diagonal map $\Delta \colon X \rtarr X \times X$ is a $G$-cofibration.
\end{defn}

Examples of $G$-LEC $G$-spaces include $G$-CW-complexes \cite{DE, LewisOMSIX}. Every basepoint of a $G$-LEC $G$-space is nondegenerate \cite[Corollary II.8]{DE}.
The following lemma gives sufficient conditions for the classifying space of a chaotic category to be nondegenerately based.

\begin{lem}
\label{ChaoticLEC}
Suppose that $\aC\in \GUCatp$ is chaotic and that $\ob\aC$ is $G$-LEC. Then $B\aC$ has a nondegenerate basepoint.
\end{lem}

\begin{proof}
Since $\ob\aC$ is $G$-LEC and $\aC$ is chaotic,  the nerve of $\aC$ is levelwise $G$-LEC. 
Then $B\aC$ is $G$-LEC by \cite[Corollary~2.4(b)]{LewisOMSIX}, and in particular it has a nondegenerate basepoint.
\end{proof}

We also need the following two general results about $G$-LEC $G$-spaces. 

\begin{lem}\label{RLEC}
 Let $X$ be a $G$-LEC based $G$-space and 
 $\bn^\al$ be a finite based $G$-set. Then $X^{\bn^\al} = \enGU(\bn^\al, X)$ is $G$-LEC. 
 
\end{lem}
\begin{proof}
 The $G$-space $X^{\bn^\al}$ can be viewed as the restriction along the homomorphism $G\rtarr G\wr \Sigma_n$ of the $G\wr \Sigma_n$-space $X^n$. 
It then follows from \cite[Proposition~A.2.6]{BVbook} that $X^{\bn^\al}$ is $G$-LEC.
\end{proof}

\begin{lem}\label{HtoGLEC}
Let $H$ be a subgroup of $G$, and let $Y$ be an $H$-LEC space. Then $G\times_H Y$ is $G$-LEC.
\end{lem}

\begin{proof}
The diagonal on $G\times_H Y$ factors as
\[ \xymatrix{
G\times_H Y \ar[r]^(0.35)\Delta \ar[dr]_{\id\times \Delta} & (G\times_H Y) \times (G\times_H Y) \ar[r]^\iso & (G\times G)\times_{H\times H} (Y\times Y) \\
 & G\times_H (Y\times Y). \ar[ur]_{\Delta\times \id}
}\]
The map $\id\times \Delta$ is the induction from $H$ to $G$ of the $H$-cofibration $\Delta_Y$, and it follows that it is a $G$-cofibration. On the other hand, we claim that the map $\Delta\times \id$ is the inclusion of a coproduct summand and is therefore a $G$-cofibration. To see this, note that the subset $\{ (g,gh) \mid g\in G, h\in H\} \subset G\times G$ is a $(G,H\times H)$-invariant subset, and it is precisely the image under the right $(H\times H)$-action of $\Delta(G)\subset G\times G$. Since $G$ is discrete, it follows that we may decompose $G\times G$ as a $(G,H\times H)$-equivariant disjoint union of this subset and its complement. Crossing with $Y\times Y$ and passing to $(H\times H)$-orbits gives a $G$-equivariant decomposition of $(G\times G)\times_{H\times H}(Y\times Y)$ into the image of $\Delta\times \id$ and its complement. 
\end{proof}  

In the remainder of this section, we let $\oO$ be an operad in $\Cat(G\Top)$, and $\cA$ be an $\oO$-algebra.

\begin{prop}\label{LECprop}
 If  $\ob\cA$ is $G$-LEC and has a disjoint basepoint, then $\cA$ is in $\OAT$. 
\end{prop}

\begin{proof}
We first prove that each object $G$-space $\bT\cA(\bn^\al)$ of $\bT\cA$ is $G$-LEC.  Write $\aW$ for $\oursectionG^* \bR_G \cA$, so that $\bT \cA=  \St_{\sF_G}\aW$.  The $G$-space $\ob\aW(\bn^\al)$ can be identified with $(\ob\cA)^{\bn^\al}$, and is thus $G$-LEC by \autoref{RLEC}.  Moreover, it has a disjoint basepoint.  Recall from \autoref{ConjOut1} that 
\[ \ob \St_{\sF_G} \aW(\bn^\al) = \bigvee_{\bk^\be}  \sF_G(\bk^\be,\bn^\al) \sma  \ob\aW(\bk^\be).\]
Since $\ob\aW(\bk^\be)$ is $G$-LEC, it follows that  $ \sF_G(\bk^\be,\bn^\al) \sma \ob\aW(\bk^\be)$ is $G$-LEC.
Since the basepoint of  $ \sF_G(\bk^\be,\bn^\al) \sma \ob\aW(\bk^\be)$ is disjoint, the infinite wedge is in fact an infinite disjoint union, with an adjoined disjoint basepoint. 
Since an arbitrary coproduct of $G$-LEC $G$-spaces is again $G$-LEC,  $\ob\bT\cA(\bn^\al)$ is $G$-LEC.

In the proof of \autoref{ConjOut1}, we defined a pseudomorphism 
\[\squiggly{m\colon \bT\cA=\St_{\sF_G} \aW}{\aW},\]
 each of whose components is an equivalence of categories.
Let $\aY(\bn^\al) \subset \bT\cA(\bn^\al)$ be the subcategory $\aY(\bn^\al) = m_{\bn^\al}^{-1}(\ast)$. Since $m_{\bn^\al}$ is an equivalence of categories, $\aY(\bn^\al)$ is equivalent to the trivial category and is therefore chaotic. 
The fact that the basepoint splits off of $\ob\aW(\bn^\al)$ implies that $\ob\aY(\bn^\al)$ splits off from $\ob\bT\cA(\bn^\al)$ and is therefore $G$-LEC since $\ob\bT\cA(\bn^\al)$ is $G$-LEC.
Since the basepoint of $B \bT\cA(\bn^\al)$ lies in $B\aY(\bn^\al)$ and $B\aY(\bn^\al)$ has a nondegenerate basepoint by  \autoref{ChaoticLEC}, this gives the conclusion.
\end{proof}

The following example will be used in \autoref{BPQtoo}.   

\begin{rem}\label{embed}  We embed $G\sU$ in $\GUCat$ by regarding an unbased $G$-space $X$ as an object of $\GUCat$ with $X$ as both the object and the morphism $G$-space and with the source, target, identity and composition maps all the identity.  Similarly, we regard $X_+$ as an object of  $\GUCatp$.  The free 
$\oO$-algebra generated by $X_+$ is  the disjoint union of the categories $\oO(j)\times_{\Sigma_j} X^j$ with base object $\ast$ the $0$th term.
\end{rem}

\begin{prop}\label{gloat}
If $X\in G\sU$ is $G$-LEC and $\ob\oO(j)$ is a (discrete) free $\Sigma_j$-set for each $j$, then $\bOp(X)$ is in $\OAT$.
\end{prop}

\begin{proof}
By \autoref{LECprop}, it suffices to show that $\ob\bOp(X)$ is $G$-LEC.  This holds if each $G$-space 
$$\ob\Big( \oO(j)\times_{\Sigma_j} X^j \Big) \iso \ob\oO(j)\times_{\Sigma_j} X^j$$ 
is $G$-LEC.  Since $\ob\oO(j)$ is discrete with free $\Sigma_j$-action, we can write it as a disjoint union of $(G\times\Sigma_j)$-sets $(G\times\Sigma_j)/\Lambda$, where $\Lambda \subset G\times \Sigma_j$ is a subgroup such that $\Lambda \cap e\times \Sigma_j = e$. In other words, the subgroup $\Lambda$ is the graph of a homomorphism $\al\colon H \rtarr \Sigma_j$ for some subgroup $H\leq G$.  For each such $\LA$, we have an isomorphism of $G$-spaces
\[ \big((G\times\Sigma_j)/\Lambda\big) \times_{\Sigma_j} X^j \iso G\times_H X^{j^\al}.
\]
Since $X$ is $G$-LEC, \autoref{RLEC} implies that $X^{j^\al}$ is $H$-LEC. Then by \autoref{HtoGLEC}, we have that $G\times_H X^{j^\al}$ is $G$-LEC as wanted.
\end{proof}

Consider the category of operators $\sD_G = \sD_G(\oO)$. The following analogue of \autoref{LECprop}, with $\sF_G$ replaced by $\sD_G$, will be needed 
in \autoref{HtpySect}. 
 There we use comparisons between infinite loop space machines $\bS_G$ defined on $\sF_G$-$G$-spaces and $\bS_G^{\sD_G}$ defined on $\sD_G^{top}$-$G$-spaces, where $\sD^{top}_G$ is the topological version of $\sD_G$, as specified in \autoref{DGGU}.   The machine $\bS_G^{\sD_G}$ has good properties when its domain is restricted to $\GDT$.  Just as for $\sF_G$, we have categories $\DGAlg$ of strict $\sD_G$-algebras and pseudomorphisms and a subcategory $\DGA$ of strict $\sD_G$-algebras and strict morphisms. \autoref{sec:PowerLack}
specializes to give a strictification functor
\[\St_{\sD_G} \colon {\DGAlg} \rtarr \DGA.\] 
The composite 
\[ \St_{\sD_G} \bR_G\colon  \OAlg  \rtarr \DGA\]
plays a role analogous to that of $\bT$ in the earlier results of this subsection, and we let $\OADT$ be the full subcategory of  $\OAlg$ consisting of those $\oO$-algebras 
$\cA$ such that $B  \St_{\sD_G} \bR_G \cA$ is in $\GDT$.
Thus, by definition, $B \St_{\sD_G} \bR_G$ restricts to a functor $\OADT \rtarr \GDT$.  

\begin{prop}\label{LECDprop}
If $\ob\oO(j)$ for each $j$ and $\ob\cA$ are $G$-LEC, and  $\ob\cA$ has a disjoint basepoint, then $\cA$ is in $\OADT$.  
\end{prop}

\begin{proof}
We need a modification of the first step of the proof of \autoref{LECprop}  to account for strictification over $\sD_G$ rather than $\sF_G$. Writing $\aZ = \bR_G \cA$, we have
\[ \ob \St_{\sD_G} \aZ(\bn^\al) = \bigvee_{\bk^\be}  \ob\sD_G(\bk^\be,\bn^\al) \sma \ob\aZ(\bk^\be).\]
But $\ob\sD_G(\bk^\be,\bn^\al)$ is a finite coproduct of finite products of $G$-spaces $\ob \oO(j)$, each of which has a disjoint basepoint and is assumed to be $G$-LEC. Therefore  $\ob \St_{\sD_G} \aZ(\bn^\al)$ is $G$-LEC.  The rest of the proof of \autoref{LECprop} goes through unchanged.
\end{proof}

The proof of \autoref{gloat} applies directly to give the following analog.

\begin{prop}
\label{OXgoodD}  If $X\in G\sU$ is $G$-LEC and $\ob\oO(j)$ is a (discrete) free $\Sigma_j$-set for each $j$, then $\bO_+(X)$ is in $\OADT$.
\end{prop}

\subsection{Nondegenerate basepoints and Day convolution}\label{DayTop}
This brief parenthetical section highlights a question that seems to have been overlooked in all previous papers dealing with the use of Day convolution in topology, even nonequivariantly, whether for spectra or for categories of operators.  We concentrate on the latter and restrict attention to $\sF$, thinking nonequivariantly for simplicity. 

Of course, smash products are constructed as quotient spaces   $X\times Y/X\vee Y$ in $\sU_{\ast}$.    It is essential to be working in compactly generated spaces since otherwise the smash product is not even associative \cite[Theorem 1.7.1]{MaySig}. It follows from Lillig's union theorem \cite{Lillig} that $X\sma Y$ is nondegenerately based if $X$ and $Y$ are.  Therefore both $\sU_{\ast}$ and its full subcategory $\sT$ are symmetric monoidal under the smash product.   By \autoref{symmonmulti}, we have associated multicategories $\Mult(\sU_*)$ and $\Mult(\sT)$. 

An $\sF$-space is an (enriched) functor $\sF\rtarr \sU_{\ast}$ and the category $\sF$-$\sU_{\ast}$ 
of $\sF$-spaces  is symmetric monoidal under the internal smash product given by Day convolution.   By \autoref{symmonmulti}, it also has an associated multicategory $\Mult(\sF\text{-}\sU_{\ast})$. That can be defined using either the internal smash product as in \autoref{symmonmulti} or using the external smash product as in \autoref{MultiD}.  These definitions give isomorphic multicategories by the universal property of Day convolution.

Now consider the category $\sF$-$\sT$ of (enriched) functors $\sF\rtarr \sT$.  It has been asserted in many places, including our own \cite{GMMO}, that $\sF$-$\sT$ is symmetric monoidal under the internal smash product.  We do not know whether or not that is true, and we believe that it is not.  The external smash product
$X\barwedge Y \colon \sF\sma \sF\rtarr \sU_*$ of functors $X, Y\colon \sF\rtarr \sT$ clearly takes values in $\sT$, but it does not follow that the internal smash product
$X\sma Y\colon \sF\rtarr \sU_*$ takes values in $\sT$.  That is, we do not believe that Day convolution preserves levelwise nondegeneracy of basepoints.  We cannot use the universal property to prove that it does, and we have not succeeded in proving that it does by direct inspection of the construction.

That problem does not affect applications since, {\em using the external smash product} as in  \autoref{MultiD}, we have the full submulticategory  $\Mult(\sF\text{-}\sT)$ of 
$\Mult(\sF\text{-}\sU_{\ast})$, {whose objects are }
levelwise nondegenerately based functors.  When we reinterpret {\em internally}, using Day convolution, we may leave that world.  The same holds for the functors from $\sF$ or, equivariantly, $\sF_G$ to topological $G$-categories that are the focus of this paper.

\subsection{From $\oO$-transformations to homotopies of maps of $G$-spectra}\label{homotopy}
 
It is classical that the classifying space functor takes $G$-categories, $G$-functors and $G$-natural transformations to $G$-spaces, $G$-maps, and $G$-homotopies. For the last, $G$-natural transformations are functors $\cC\times \sI\rtarr \cD$, where $\sI$ is the category with two objects $[0]$ and $[1]$ and one non-identity morphism $ [0]\rtarr [1]$.   For based $G$-categories, {based $G$-}transformations are given by {based $G$-}functors $\cC \sma \sI_+ \rtarr \cD$.  

Recall the definition of $\oO$-transformations from \autoref{Otran}. 
 
 \begin{prop}\label{Htpies}
 The functor $\bK_G$ takes $\oO$-transformations to homotopies of maps of $G$-spectra.
 \end{prop}
 
\begin{pf}
We showed in \cite[Proposition 6.16]{GMMO} that the topological Segal machine $\bS_G$ preserves homotopies. If we start with 
$\sF_G$-algebras in $\GUCatp$, which of course are themselves $G$-functors, then maps between them are $G$-natural transformations and maps between those are $G$-modifications.   These are given levelwise by $G$-categories, $G$-functors, and $G$-natural transformations.   
Since $B$ commutes with products, it takes $\sF_G$-$G$-algebras in $\GUCatp$, $\sF_G$-functors, and $\sF_G$-transformations between them to
$\GTop_*$-enriched functors $\sF_G \rtarr \enGU$, enriched natural transformations, and homotopies between those.
As $\bR_G$, $\oursectionG^*$, and $\St_{\sF_G}$ are all $2$-functors, with $\St_{\sF_G}$ converting pseudostructure to strict structure, their composite takes $\oO$-transformations to $\sF_G$-transformations, which are levelwise $G$-natural transformations.
 \end{pf}
 
 The result above is used in \cite[Remark 2.9]{GM}, but it will surely find other uses.

\section{The multiplicative Barratt-Priddy-Quillen Theorem}
 \label{sec:BPQ}

In this section, we prove \autoref{IntroBPQ}. We begin by producing the transformation $\alpha$ in \autoref{alSect}. We show that $\alpha_X$ is a  stable equivalence of orthogonal $G$-spectra for all $G$-LEC $G$-spaces $X$ in \autoref{HtpySect} and we finish by showing that  $\alpha$ is monoidal in \autoref{MonSect}.

\subsection{The construction of $\al$} \label{alSect}

We restate and begin the proof of \autoref{IntroBPQ}.  

\begin{thm}\label{BPQtoo} Let $\oO$ be a topologically discrete chaotic $E_\infty$ $G$-operad in $\GUCat$. 
Then there is a lax monoidal natural transformation 
\[ \alpha \colon \Sigma^\infty_{G+}  \rtarr \bK_G \bOp\]
of functors $\GTop\rtarr \SpG$
 such that $\al_X$ is a stable equivalence of orthogonal $G$-spectra for all 
$G$-LEC $G$-spaces $X$.
\end{thm}

Recall that we write $\bT$ for the composite 
\[\bT =  \St_{\sF_G} \circ \oursectionG^* \circ  \bR_G: \OAlg \rtarr \FGAlg, \] 
so that $\bK_G$ is given by $\bK_G = \bS_G \circ B\circ \bT$. We shall exploit the fact that $\Sigma^\infty_{G+}\colon \GTop \rtarr \SpG$ is left adjoint to the zeroth 
$G$-space functor $(-)_0$, with basepoint forgotten, to construct $\al$. 
Therefore,
 to define $\alpha: \Sigma^\infty_{G+}  \rtarr \bK_G  \bOp$ in \autoref{BPQtoo},  it suffices
 to define a map of unbased $G$-spaces
\[ \tilde{\alpha}_X: X \rtarr {\bS}_G\big( B\bT \bOp X \big)_0 \] 
for each unbased $G$-space $X$. 
We define $\tilde\al_X$ to be the composite displayed in the diagram
\[\xymatrix{
X \ar[d]_{\cong} \ar[rr]^{\tilde{\al}_X}  & & {\bS}_G\big( B\bT \bOp X\big)_0 \\ 
B  X \ar[d]_-{B \eta} & & \\
B  \bOp X\ar@{=}[r] & \big{(}B\bR_G\, \bOp X\big{)} (  \mb{1}) \ar[r]_-{B i } 
&  \big( B \bT \bOp X \big)  (  \mb{1}) . \ar[uu]_{\nu} 
\\}\] 
Regarding $X$ as an object of $\GUCat$ as in \autoref{embed},  the top left isomorphism is immediate;
$\et$ on the left is the unit of the monad $\bOp$.  
For the bottom left equality, it is  true by definition that $\cA= (\bR_G \cA)(\mb{1})$ for any $\oO$-algebra $\cA$ in $\GUCat$, such as $\cA = \bOp X$.
Next, for any strict $\sD_G$-algebra $\aY$, such as $\aY = \bR_G\bOp X$, the map
\[  i  \colon \aY(  \mb{1}) \rtarr \St_{\sF_G} \oursectionG^* \aY (  \mb{1})\]
is given by
\[ \aY(  \mb{1}) = \oursectionG^* \aY (  \mb{1}) \xrtarr{ i } \St_{\sF_G} \oursectionG^*\aY (  \mb{1}),\]
where $i$ is a component of the unit of the adjunction of \autoref{ConjOut1} and is an equivalence.  Finally, 
the map $\nu$ is the natural map required by \autoref{segalmachine}; it will be specified in \autoref{HtpySect}.

The adjoints of the maps $\tilde\al_X$ define the natural transformation $\al$,  and we must verify that $\al$  is monoidal 
and homotopical, the latter meaning that $\al$ is a stable $G$-equivalence.

\subsection{The proof that $\al$ is a stable equivalence} \label{HtpySect}

We will show that $\al_X$ is an equivalence by comparing it with the equivariant Barratt-Priddy-Quillen equivalence for the equivariant operadic machine proven in \cite{GMPerm} and the equivalence between the equivariant operadic and Segal machines proven in \cite{MMO}. Consider the following diagram of $G$-spectra, in which $\bOp^{top}$ denotes the monad on unbased $G$-spaces associated to the operad 
$B\oO$, $\bE_G$ denotes the operadic infinite loop space machine of \cite[Definition 2.7]{GMPerm}, $\bE_G^{\sD_G}$ denotes the category of operators infinite loop space machine of \cite[Definition 6.19]{MMO},\footnote{We repeat that numbered references to \cite{MMO} refer to the first version posted on ArXiv;  they will be changed when the published version appears.}  and $\SGD$ denotes the Segal machine on $\sD_G^{top}$-$G$-spaces of \cite[\S4.5]{MMO}.   The dotted arrows signify zig-zags of maps. 

\begin{equation} \label{home} 
\xymatrix{
\SI^\infty_G X_+ \ar[d]_{\fbox{1}} \ar[rr]^-{\al_X}  & &  \bS_G B \bT \bOp X \\
\bE_G\bOp^{top}X \ar[d]_{\fbox{2}}^{\iso} & &  \SGD B\bR_G \bOp X \ar@{-->}[u]_(.47){\fbox{5}} \\
\bE_G^{\sD_G} \bR_G^{top} \bOp^{top}X   \ar@{-->}[rr]_(.47){\fbox{3}}  & &  \SGD\bR_G^{top}\bOp^{top}X \ar[u]^{\iso}_{\fbox4} \\} 
\end{equation}

The  arrows $\fbox{1}$ through $\fbox{5}$ are specified as follows.
\begin{enumerate}
\item The map $\fbox{1}$ is the stable equivalence of orthogonal $G$-spectra given in \cite[Theorem 6.1]{GMPerm}, which is an operadic version of the Barratt-Priddy-Quillen Theorem.

\vspace{1mm}

\item The isomorphism of $\fbox{2}$ is a comparison between operad level and category of operators level infinite loop space machines given by  \cite[Corollary 6.22]{MMO}.

\vspace{1mm}

\item The dashed arrow $\fbox{3}$ is a zig-zag of stable equivalences between the generalized operadic and Segal machines, both defined on $\sD_G^{top}$-$G$-spaces. This is given by  \cite[Theorem 7.1]{MMO}

\vspace{1mm}

\item The isomorphism \fbox{4} is induced by the isomorphism $\bR_G^{top} B\iso B\bR_G$. 

\vspace{1mm}

\item The zig-zag of level equivalences \fbox{5} is described in \autoref{ZigZagFive} below.
\end{enumerate}

Write $\aY$ for the $\sD_G$-algebra $\bR_G\bOp X$.
\autoref{ConjOut1}, \autoref{StrictEquiv}, and
\autoref{zeta1} give a zig-zag of strict maps of $\sD_G$-algebras in $\GUCatp$ that are level equivalences
\[  \xymatrix{ \aY &  \St_{\sD_G}\aY \ar[l]_\sim^{m} \ar[r]^-{\sim}_-{\St_{\sD_G} \chi} & \St_{\sD_G}\xi_G^*  \oursectionG^*\aY \ar[r]^-{\sim}_-{\ps} 
&  \xi_G^* \St_{\sF_G}\oursectionG^*\aY.\\}\]
We apply $B$ to this and use that $\xi^*_G$ commutes with $B$ to obtain a zigzag of level equivalences of
 $\sD_G^{top}$-$G$-spaces. 
 \[  \xymatrix{ B\aY &  B\St_{\sD_G}\aY \ar[l]_-\sim^{Bm} \ar[rr]^-{\sim}_-{B\ps\com B\St_{\sD_G} \chi}   & &  \xi_G^* B\St_{\sF_G}\oursectionG^*\aY.\\}\]
The $\sD_G^{top}$-$G$-space $B\aY$ has levelwise disjoint basepoints, whereas the basepoints of $B\St_{\sD_G}\aY$ and $ \xi_G^* B\St_{\sF_G}\oursectionG^*\aY$ are levelwise nondegenerate according to Propositions \ref{OXgoodD} and \ref{gloat}, respectively.
By \cite[Proposition~4.31]{MMO}, the Segal machine $\SGD$ on $\sD_G^{top}$-$G$-spaces
converts this to a zig-zag of stable equivalences of orthogonal $G$-spectra.  By  \cite[Theorem~4.32]{MMO},
we have a natural stable equivalence $\SGD\xi^*_G \rtarr \bS_G$
relating the Segal machine on $\sD_G^{top}$-$G$-spaces to the Segal machine on $\sF_G$-$G$-spaces.  Applying this
to the last term in the sequence above and remembering that $\bT = \St_{\sF_G} \oursectionG^* \bR_G$ and $\aY=\bR_G\bOp X$,
we obtain the final zig-zag \fbox5 of stable equivalences of orthogonal $G$-spectra:

\begin{equation} \label{ZigZagFive}
\xymatrix{
\SGD  B   \bR_G \bOp X & \SGD B \St_{\sD_G} \bR_G \bOp X  \ar[r]^-{\sim} \ar[l]_-{\sim} & \SGD \xi^*_G B\bT  \bOp X  \ar[r]^-{\sim} &  \bS_G B\bT \bOp X. \\}
\end{equation}

In order to deduce that $\al_X$ is a stable equivalence, we need only show that \autoref{home} yields a commutative diagram in the homotopy category.  As we are mapping out of a suspension spectrum, we may instead consider the diagram on zeroth spaces, by adjunction.
Chasing the required diagrams is tedious but routine.  We give a few details for the skeptical reader.  For this, we only need to know that the operadic and Segal machines are given by two-sided monadic and categorical bar constructions with easily described zeroth spaces and maps between them. We use the notations from \cite{MMO}, and we refer the reader to that source for more details of the definitions of the constructions and maps between them. 
Abbreviating by writing $q=B\ps\com B\St_{\sD_G} \chi$, the adjoint of \autoref{home} is a diagram of  
unbased $G$-spaces that  takes the following form:

\begin{equation}\label{homeadjoint}
\scalebox{0.67}{
\xymatrix @C=1em {
 X   \ar[dd]_{\fbox{1}} \ar[r]^-{B\eta}   \ar@{-->}[ddrrr]^{\fbox8} \ar@{-->}@/^2em/[dddddr]_{\fbox6} \ar@{-->}[ddddrrr]_{\fbox7}
 &(B\bR_G\bOp X)(\mb{1}) \ar[r]^-{B i } & (B \St_{\sF_G}\oursectionG^*\bR_G \bOp X)(  \mb{1}) \ar[r]^-\nu   &  B((S^0)^\bullet,\sF_G, B \bT \bOp  X) \\
& & &  B((S^0)^\bullet,\sD_G^{top}, \xi_G^* B \bT \bOp  X) \ar[u]_\sim \\
\bOp^{top}X \ar[ddd]_{\fbox{2}}^{\iso}  & & & B((S^0)^\bullet,\sD_G^{top}, B \St_{\sD_G}\bR_G\bOp X) \ar[u]_\sim^q \ar[d]^\sim_{Bm} \\
  & & & B((S^0)^\bullet,\sD^{top}_G, B\bR_G \bOp X) \\
  &  & & B((S^0)^\bullet, \sD_G^{top},\bR_G\bOp^{top}X)  \ar[u]^{\iso}_{\fbox4} \ar[d]^{\io_1}  \\
\bOp^{top}X & B({}^\bullet(S^0),\sD_G^{top}, \bR_G\bOp^{top}X) \ar[l]^-{\om_0} \ar[r] & B((S^0)^\bullet, \sD_G^{top},\bR_G\bOp^{top}X) \ar[r]_-{\io_0}     & B(I_+\sma (S^0)^\bullet, \sD_G^{top},\bR_G\bOp^{top}X).
\\} }
\end{equation}
Here the numbers \fbox{n} labelling solid arrows are the induced maps on $0$th $G$-spaces from
the maps \fbox{n} in \autoref{home}.  
The zig-zag of morphisms starting from the bottom left of the diagram and ending at the source of \fbox{4} is the  
zig-zag of maps of $0$th $G$-spaces induced by
\fbox{3} in \autoref{home},  while the top three vertical maps on the right are the zig-zag 
induced by
 \fbox{5}.  The top horizontal composite is the adjoint of $\al_X$.   In the middle left entry, we use an identification for the operadic machine that is explained in \cite[Remark 2.9]{GMPerm}; the map  \fbox{1} is just $\et^{top}\colon X \rtarr \bOp^{top}X \cong B\bOp X$, the isomorphism that of \autoref{BcommutesO}.   At the top,  $(B\bR_G\bOp X)(\mb{1}) =B \bOp X$,  and  $B\et=\et^{top}$.

We will define the dotted arrow maps \fbox6, \fbox7, and \fbox8 and show that each subdiagram commutes, at least up to homotopy.   The maps \fbox6, \fbox7, and \fbox8 each map into the summand labeled by $\mb{1}$ in the space of 0-simplices 
of its corresponding (categorical) bar construction.  These two-sided bar constructions are of the form $B(Y,\sE,Z)$, where $\sE$ is a $G\Top_*$-enriched category, and $Y\colon \sE^{op} \rtarr \enGU$ and $Z\colon \sE \rtarr \enGU$ are $G\Top_*$-enriched functors. The space of 0-simplices is given by\footnote{We here use a choice of bar construction that is slightly modified from but equivalent to the choice in the ArXiv version of \cite{MMO}; it will appear in the published version and in a revision to be posted on ArXiv.  It is the choice labelled  $B^{\bN_G}$ in \cite{GMMO}.}
\[\bigvee_n Y(n) \sma Z(n),\]
where $n$ ranges over the objects of $\sE$.
In all but one of the bar constructions in the diagram (the bottom right corner), $Y({\bf 1})=S^0$, hence the summand 
 $Y({\bf 1}) \sma Z({\mb 1})$
is isomorphic to $Z({\mb 1})$.  Taking $\sE = \sF_G$, the Segal machine $\bS_G Z$ is constructed as 
\[ (\bS_GZ)_V = B((S^V)^\bullet,\sF_G,Z). \]
Taking $V=0$, the map $\nu\colon  Z(\mb 1) \rtarr (\bS_G Z)_0$ is given by the inclusion of $Z(\mb 1)$ in the space of $0$-simplices of the bar construction.  The map $\om_0$ in \autoref{homeadjoint} is given by projection to the relevant component, as in \cite[Section 7.6]{MMO}.

Replacing bar constructions in \autoref{homeadjoint} by the components of their zero simplices that serve as targets for the maps with domain $X$, the diagram can be written as
\begin{equation}
 \label{LevelOneDiagram}
\scalebox{0.9}{
\xymatrix{
 X \ar[dd]_{\eta} \ar[r]^-{B\eta}   \ar@{-->}[ddrrr]^{\fbox8} \ar@{-->}@/^3ex/[dddddr]_{\fbox6} \ar@{-->}[ddddrrr]_{\fbox7}
 &B \bOp X \ar[r]^-{B i } & B \bT \bOp X(  \mb{1}) \ar@{=}[r]^\nu   &  B \bT \bOp X(  \mb{1}) \\
& & \hspace{6em} \fbox{D}  &  \xi_G^* B \bT \bOp X(  \mb{1}) \ar@{=}[u] \\
\bOp^{top}X \ar@{=}[ddd] & & \hspace{4em} \fbox{C} &  B \St_{\sD_G} \bR_G\bOp X(  \mb{1}) \ar[u]^\sim_q \ar[d]_\sim^{Bm} \\
  & & &  B \bR_G\bOp X(  \mb{1}) \\
  & \hspace{-6em} \fbox{A} & \hspace{1em} \fbox{B} &  \bOp^{top}X  \ar[u]^{\iso}_{\fbox4} \ar[d]^{\io_1}  \\
 \bOp^{top}X &  \bOp^{top}X \ar@{=}[l]^{\om{_0}} \ar@{=}[r] &  \bOp^{top}X \ar[r]_-{\io_0}     & I_+ \sma \bOp^{top}X
\\ } }
\end{equation}
We take \fbox6 and \fbox7 both to be $\eta$, making diagram \fbox{A} commute and diagram \fbox{B} commute up to homotopy.
We take \fbox8 to be the composite
\[ X \xrtarr{B\eta} B\bOp X  = B\bR_G \bOp X(  \mb{1}) \xrtarr{B i_{\sD}} B \St_{\sD_G} \bR_G\bOp X(  \mb{1}) .\]
Region \fbox{C} of the diagram commutes by the triangle identity, since $ i $ and $ m$ are the unit and counit for the adjunction in \autoref{ConjOut1}. 
Region \fbox{D} is obtained by applying the classifying space functor $B$ to the following diagram.
\[ 
\xymatrix{
X \ar[d]_{ \eta} \ar[r]^{\eta} &  \bOp X \ar[r]^{ i _\sF} &  \bT \bOp X(  \mb{1}) \ar@{=}[d]
\\
  \bR_G \bOp X(  \mb{1}) 
 \ar[d]_{  i _\sD} \ar[r]^-{\chi} &   \xi_G^* \oursectionG^* \bR_G \bOp X(  \mb{1}) \ar[r]^{ \xi_G^*  i _\sF} \ar[d]_{ i _\sD} &  \xi_G^* \St_{\sF_G} \oursectionG^* \bR_G \bOp X(  \mb{1}) \ar@{=}[u] 
\\
 \St_{\sD_G} \bR_G \bOp X(  \mb{1}) \ar[r]_-{\St_{\sD_G} \chi} &  \St_{\sD_G} \xi_G^* \oursectionG^* \bR_G \bOp X(  \mb{1}) \ar[ur]_{\psi} 
}\]
The top rectangle commutes because the components of $\chi$ (\autoref{StrictEquiv}) are identity maps. The lower rectangle commutes by the naturality of $ i $, and the triangle commutes by \autoref{zeta1}.

\subsection{The proof that $\al$ is monoidal} \label{MonSect} 

Since the adjoint of $\widetilde{\al}_{\{*\}}$ 
is easily seen to be the unit map of the lax monoidal functor $\bK_G\, \bOp$, it suffices to verify that the following diagram commutes for $G$-spaces $X$ and $Y$. Recall again that $(X\times Y)_+\iso X_+\sma Y_+$.
\[ \xymatrix{
\SI^\infty_G X_+\sma \SI^\infty_G Y_+ \ar[r]^-{\al\sma \al} \ar[d]_{\iso} & \bK_G \bOp X\sma \bK_G  \bOp Y\ar[d]^{\varphi} \\
\SI^\infty_G(X\times Y)_+ \ar[r]_-\al & \bK_G \bOp(X\times Y)
}\]
The map $\varphi$ is constructed by applying the composite of the multifunctors from
Theorems \ref{freeinput} and  \ref{roadkill} to the identity map $X\times Y \rtarr X \times Y$ considered as a 2-ary morphism in $\Mult(G\Top)$.

By passage to adjoints, and since the adjunction is monoidal, we  are reduced to showing that the following diagram of maps of $G$-spaces commutes.
 \[ \xymatrix @C=8ex{
X \times Y
\ar[r]^-{\tilde{\al}_X\times \tilde{\al}_Y} \ar[dd]_{=} & (\bK_G \bOp X)_0 \times (\bK_G \bOp Y )_0 \ar[d] \\
& (\bK_G \bOp X \sma \bK_G \bOp Y )_0 \ar[d]^{\varphi_0}\\
X\times Y \ar[r]_-{\tilde{\al}_{X\times Y}} & (\bK_G \bOp(X\times Y))_0 }\] 
Here, the top right map is the lax monoidal constraint for the zeroth space functor.

The proof is concluded by checking that the following diagram commutes.
\[ 
\def\objectstyle{\scriptstyle}
\xymatrix @C=4.5ex{
X \times Y \ar[r]^-{B\eta\times B\eta} \ar[d]_{B\eta }   &  B \bOp X   \times  B \bOp Y \ar[r] ^-{B i \times B i } 
&    B \bT \bOp X(  \mb{1})  \times  B \bT \bOp Y(  \mb{1}) \ar[d]^{\nu\times \nu} \\   
B \bOp(X\times Y)   \ar[d]_{B i } & 
&( {\bS}_G B\bT \bOp X)_0 \times ({\bS}_G B\bT  \bOp Y)_0 \ar[d] \\
 B \bT \bOp(X\times Y)(  \mb{1}) \ar[r]_-{\nu} & (\bS_G B \bT \bOp(X\times Y))_0 &  \big(({\bS}_G B \bT \bOp X )\sma ({\bS}_GB \bT \bOp Y)\big)_0 \ar[l]^-{\varphi_0}
} \]

\section{Coherence axioms}\label{sec:coh}

We collect the coherence axioms we need in this section. Those for pseudo-commutative operads, deferred from \autoref{pseudocom},  appear in \autoref{cohpseudo};
those for  $\Mult(\oO)$, deferred from \autoref{MultiO}, appear in \autoref{cohMultO}; and those for  $\Mult(\sD)$, deferred from the unpacking of \autoref{MultiD} in \autoref{MdpaDef}, are gathered in \autoref{cohMultD}.

\subsection{Coherence axioms for pseudo-commutative operads}\label{cohpseudo}

Here we complete \autoref{pseudocom} by specifying coherence axioms for the $\al_{j,k}$  of diagram \autoref{alphas}.

\begin{enumerate}[(i)]
\item\label{pseudocomid} The component of $\alpha_{1,n}$ at $(\oid,y)$ is the identity map.
\item\label{symmpseudocom} The 
composite
\[
\xymatrix{
\oO(j) \times \oO(k) \ar[r]^-{\opair} \ar[d]_{t}
\drtwocell<\omit>{<0>\qquad \alpha_{j,k}}      & \oO(jk)
 \ar[d]^{\tau_{k,j}}  \\
\oO(k) \times \oO(j) \ar[r]^-{\opair} \ar[d]_{t}
\drtwocell<\omit>{<0>\qquad \alpha_{k,j}}      & \oO(kj)
 \ar[d]^{\tau_{j,k}}  \\
\oO(j) \times \oO(k)  \ar[r]_-{\opair} & \oO(jk).\\}
\]
is the identity $\sV$-transformation. 
\item\label{symmpseudocomequiv} For permutations $\rho_1 \in \Sigma_k$ and $\rho_2\in \Sigma_j$,
\[
\xymatrix{
\oO(j) \times \oO(k) \ar[r]^-{\opair} \ar[d]_{t}
\drtwocell<\omit>{<0>\qquad \alpha_{j,k}}      & \oO(jk)
 \ar[d]^{\tau_{k,j}}  \\
\oO(k) \times \oO(j) \ar[r]^-{\opair} \ar[d]_{\rho_1\times\rho_2}
     & \oO(kj)
 \ar[d]^{\rho_1 \spair \rho_2}  \\
\oO(j) \times \oO(k)  \ar[r]_-{\opair} & \oO(jk).\\}
\ \ \ \ 
\raisebox{-9ex}{=}
\ \ \ \ 
\xymatrix{
\oO(j) \times \oO(k) \ar[r]^-{\opair} \ar[d]_{\rho_2 \times \rho_1}
    & \oO(jk)
 \ar[d]^{\rho_2\spair\rho_1}  \\
\oO(j) \times \oO(k) \ar[r]^-{\opair} \ar[d]_{t}
\drtwocell<\omit>{<0>\qquad \alpha_{j,k}}      & \oO(jk)
 \ar[d]^{\tau_{k, j}}  \\
\oO(k) \times \oO(j)  \ar[r]_-{\opair} & \oO(kj).\\}
\] 
\item\label{pseudocommhex1} 
Let 
\[ 
 \Prod_{i=1}^j \oO(k_i) \times \oO(\ell) \xrtarr{\Delta_\ell} \Prod_{i=1}^j \Big( \oO(k_i) \times \oO(\ell) \Big)
  \]
and
 \[
 \oO(\ell) \times \Prod_{i=1}^j \oO(k_i) \xrtarr{\Delta_\ell'} \Prod_{i=1}^j \Big( \oO(\ell) \times \oO(k_i)  \Big)
 \]
 be the morphisms whose $i$th components are the products $p_i\times \id$ and $\id \times p_i$, respectively.
We require the 2-cell
\[
\scalebox{0.9}{
\xymatrix @C=3em {
\oO(j) \times \Prod_{i=1}^j \oO(k_i) \times \oO(\ell) \ar[d]_{\iso} \ar[r]^-{\id\times \Delta_\ell}
&
\oO(j) \times \Prod_{i=1}^j \Big( \oO(k_i) \times \oO(\ell) \Big)
\ar[r]^-{\id \times \prod \opair} 
\ar[d]_\iso
\drtwocell<\omit>{<0>\qquad\qquad \id \times \prod \alpha_{k_i,\ell}}    
&
\oO(j) \times \Prod_{i=1}^j \oO(k_i\ell) \ar[d]^{\id \times \prod \tau_{\ell,k_i}}
 \\
\oO(j) \times \oO(\ell) \times \Prod_{j=1}^j \oO(k_i) \ar[d]_{\id \times \prod \Delta_\ell}
\ar[r]_-{\id\times \Delta_\ell'} 
&
\oO(j) \times \Prod_{j=1}^j  \Big( \oO(\ell) \times \oO(k_i) \Big)
\ar[r]_-{\id \times \prod \opair}
& 
\oO(j) \times \Prod_{i=1}^j \oO(\ell k_i) \ar[d]^\ga
  \\
\oO(j) \times \oO(\ell) \times \Prod_{j=1}^j \oO(k_i)^\ell  \ar[r]^-{\opair \times \id} 
\ar[d]_\iso
\drtwocell<\omit>{<0>\qquad \alpha_{j,\ell}\times\id }    
& 
\oO(j \ell) \times \Prod_{i=1}^j \oO(k_i)^\ell \ar[r]^-\ga
\ar[d]^{\tau_{\ell,j} \times \id}
&
\oO(\ell k) \ar[dd]^{D_{\ell,k_*}}
\\
\oO(\ell) \times \oO(j) \times  \Prod_{i=1}^j \oO(k_i)^\ell   \ar[r]_-{\opair \times \id} \ar[d]_\iso
&
\oO(\ell j)  \times \Prod_{i=1}^j \oO(k_i)^\ell \ar[d]^\iso
& 
\\
\oO(\ell) \times \oO(j) \times \Big( \Prod_{i=1}^j \oO(k_i) \Big)^\ell   \ar[r]_-{\opair \times \id}
&
\oO(\ell j)  \times \Big( \Prod_{i=1}^j \oO(k_i) \Big)^\ell \ar[r]_-\ga
& 
\oO( \ell k)
} }
\]
to be equal to the 2-cell
\[
\scalebox{0.95}{
\xymatrix{
\oO(j) \times \Prod_{i=1}^j \oO(k_i) \times \oO(\ell) \ar[d]_{\ga\times \id} \ar[r]^-{\id\times \Delta_\ell}
&
\oO(j) \times \Prod_{i=1}^j \Big( \oO(k_i) \times \oO(\ell) \Big)
\ar[r]^-{\id \times \prod \opair} 
&
\oO(j) \times \Prod_{i=1}^j \oO(k_i\ell) \ar[d]^\ga \\
\oO(k) \times \oO(\ell) \ar[rr]^-\opair \ar[d]_{\iso}
\drrtwocell<\omit>{<0>\qquad \alpha_{k,\ell}}    
&  
& 
\oO(k\ell)
 \ar[d]^{\tau_{\ell,k}}  
 \\
 \oO(\ell) \times \oO(k) \ar[rr]_-\opair 
 & 
 & 
 \oO(k\ell).
} }
\]
\end{enumerate}
Here, in the first $2$-cell above, $D_{\ell, k_*}$ is the distributivity isomorphism specified in the following definition.

\begin{defn}\label{dist}
Given $\ell$, $j$,  and $k_1,\dots,k_j$, let $k=\sum k_i$ and  define 
\[ D_{\ell,k_*} \colon \ul{\ell}\spair(\ul{k_1} \oplus \dots \oplus \ul{k_j}) \rtarr (\ul{\ell}\spair \ul{k_1}) \oplus \dots \oplus (\ul{\ell} \spair \ul{k_j})\]
in the bipermutative category $\Sigma$. It is given explicitly as the permutation 
\[ D_{\ell,k_*} = \left( \tau_{k_1,\ell} \oplus \dots \oplus \tau_{k_j,\ell}\right) \tau_{\ell,k}.\]
Alternatively, using the operad structure on $\Ass$, we can identify $D_{\ell,k_*}$ as
 \[D_{\ell,k_*}=\gamma(\tau_{\ell,j};\Delta^{\ell}(e_{k_1},\dots,e_{k_j})).\]
It is the permutation in $\SI_{\ell k}$ that permutes blocks of sizes $k_1,\dots, k_j, \dots, k_1,\dots,k_j$ according to $\tau_{\ell,j}$.
\end{defn}

\begin{rem}
In the language of Corner and Gurski \cite[Theorem 4.6]{CG}, axiom \eqref{symmpseudocom}
states that we require pseudo-commutativity structures to be symmetric. Axiom \eqref{symmpseudocomequiv} is an equivariance axiom that is necessary in order for $\al$ to induce a map at the monad level (\autoref{prop:alphaNatTran}), but which was unfortunately  omitted in \cite{CG}. 

Axiom \eqref{pseudocommhex1} encodes the compatibility of $\al$ with operadic composition, and is given in \cite[Theorem 4.4]{CG}. 
Unpacking axiom \eqref{pseudocommhex1}, it states that
given $x \in \oO(j)$, $y_i \in \oO(k_i)$, and $z\in \oO(\ell)$, the following diagram commutes:
\[
\scalebox{0.82}{
\xymatrix @R=4ex
{
\gamma(x; y_1 \opair z,\dots , y_j\opair z ) 
{
\big( \tau_{\ell,k_1}\oplus \dots \oplus \tau_{\ell,k_j} \big) D_{\ell,k_*}}
\ar@{=}[r] \ar@{=}[d] 
&
\gamma(x; y_1 \opair z,\dots , y_j\opair z ) 
{
\tau_{\ell,k} }
\ar@{=}[dd]
\\
\gamma\Big(x; (y_1 \opair z)
\tau_{\ell,k_1},\dots , (y_j\opair z)
\tau_{\ell,k_j} \Big) 
{
D_{\ell,k_*}}
\ar[d]_{\gamma(\id;\alpha, \dots, \alpha)
{
D_{\ell,k_*}}}
&
\\
\gamma(x; z\opair y_1, \dots, z\opair y_j ) 
{
D_{\ell,k_*}}
\ar@{=}[d]
& 
\Big( \gamma(x; y_1,\dots , y_j) \opair z \Big) 
{ 
\tau_{\ell,k} }
\ar[ddd]^\alpha
\\
\gamma(x\opair z; \Delta^\ell(y_1), \dots, \Delta^\ell( y_j) ) 
{
D_{\ell,k_*}}
\ar@{=}[d]
& 
\\
\gamma((x\opair z)
\tau_{\ell,j}; \Delta^\ell(y_1, \dots,  y_j) ) 
\ar[d]_{\gamma(\alpha;\id,\dots,\id)}
& 
\\
\gamma(z\opair x; \Delta^\ell(y_1, \dots,  y_j) ) 
\ar@{=}[r]
&
z \opair  \gamma(x; y_1,\dots , y_j).
}}
\]
This formulation is closer to what is stated in \cite{CG}. 
We can summarize it by the equation
\[ (\alpha_{x,z} \opair \id_y) \circ [( \id_x \opair \alpha_{y,z} )D_{\ell,k_*}] =  \alpha_{xy,z}.
\]
This axiom plays the role of a hexagon axiom in our context.  (There is a second axiom 
relating $\al$ to operadic composition in \cite[Theorem 4.4]{CG}, but the two axioms are equivalent in the presence of the symmetry axiom \eqref{symmpseudocom}.)
\end{rem}

\subsection{Coherence axioms for $\Mult(\oO)$}\label{cohMultO}

Here we return to the diagram \autoref{deltan}, which defines the invertible $\sV$-transformations $\de_i$ central to \autoref{MultiO}, and complete that definition.  
In axioms \eqref{OperEquivAxiom}  and \eqref{OperCompAxiom},  we will use
the shorthand notation $\aA_{i, j}$ for the product $\aA_i\times \aA_{i+1}\times\dots\times \aA_{j}.$ 
The map $\mu$ appearing in \eqref{OperCompAxiom} and \eqref{DeiDejCommute} is defined 
immediately following the axioms.  
The axioms require certain pasting diagrams to be equal, and in some cases, it will not be immediately apparent that the boundaries are equal; we address that issue after the axioms as well.
We use $\times$ rather than $\sma$ throughout since the $\oO(n)$ are unbased and their algebras are defined using powers rather than smash powers.

\begin{enumerate}[(i)]

\item\label{OperUnitAxiom} (Unit Object) The $\sV$-transformation $\de_i(0)$ is the identity.

\smallskip

\item\label{OperEquivAxiom} (Equivariance)
For any permutation $\rho \in \SI_n$, we require the $2$-cell
\[
\scalebox{0.85}{
\xymatrix @C=6em{
\aA_{1,i-1} \times \oO(n) \times \aA_i^n \times \aA_{i+1,k}
\ar[r]^{\id \times (\rho \times \id) \times \id}
\ar[d]_{\id \times (\id\times \rho) \times \id}
&
\aA_{1,i-1}\times \oO(n) \times \aA_i^n \times \aA_{i+1,k}
\ar[d]^{\id \times \tha(n) \times \id}
\\
\aA_{1,i-1}\times \oO(n) \times \aA_i^n \times \aA_{i+1,n} 
\ar[r]^{\id \times \tha(n) \times \id}
\ar[d]_{s_i}
&
\aA_{1,i-1} \times \aA_i \times \aA_{i+1,k}
\ar[dd]^F
\\
\oO(n) \times ( \aA_{1,i-1} \times \aA_i \times \aA_{i+1,k} )^n 
\ar[d]_{\id \times F^n}
\\
\oO(n) \times \aB^n 
\ar[r]_{\tha(n)}
&
\aB
\uultwocell<\omit>{<0> \qquad\qquad \de_i(n)}
} }
\]
to be equal to the $2$-cell
\[
\scalebox{0.75}{
\xymatrix @C=3.8em{
\aA_{1,i-1}\times \oO(n) \times \aA_i^n \times \aA_{i+1,k} 
\ar[r]^{\id \times (\rho\times \id) \times \id}
\ar[d]_{s_i}
&
\aA_{1,i-1}\times \oO(n) \times \aA_i^n \times \aA_{i+1,k}
\ar[r]^(0.55){\id \times\tha(n) \times \id}
\ar[d]_{s_i}
&
\aA_{1,i-1}\times \aA_i \times \aA_{i+1,k}
\ar[dd]^F
\\
\oO(n) \times (\aA_{1,i-1}\times \aA_i \times \aA_{i+1,k} )^n
\ar[r]^{\rho\times \id}
\ar[d]_{\id\times F^n}
&
\oO(n) \times (\aA_{1,i-1}\times \aA_i \times \aA_{i+1,k} )^n
\ar[d]_{\id\times F^n}
\\
\oO(n) \times \aB^n
\ar[r]_{\rho\times \id}
& 
\oO(n) \times \aB^n
\ar[r]_{\tha(n)}
&
\aB.
\uultwocell<\omit>{<2> \qquad\qquad \de_i(n)}
} }
\]

\smallskip

\smallskip

\item\label{OperIdentAxiom} (Operadic Identity): The component of $\delta_i(1)$ at an object 
\[(a_1,\dots,a_{i-1},(\oid,a_i),a_{i+1},\dots,a_k)\]
is the identity map, where $\oid\in \oO(1)$ is the unit of the operad $\oO$.
\smallskip

\item\label{OperCompAxiom} (Operadic Composition): 
We require $\de_i$ to be compatible with composition in the operad. 
In order to save space, we choose to display only the case of $i=k$, but the general case is analogous.

The composite $2$-cell
\[
\scalebox{0.75}{
\xymatrixrowsep{1.5cm}\xymatrix@C=5em{
 \aA_{1,k-1}\times \oO(n)\times \Prod_r \bigg( \oO(m_r) \times \cA_k^{m_r}\bigg) \ar[d]_{s_k} \ar[r]^-{\id \times \Prod_r \tha(m_r)} 
 &
  \aA_{1,k-1} \times \oO(n) \times \aA_k^n \ar[d]^{s_k} \ar[r]^-{\id \times \tha(n)}
 & 
 \aA_{1,k-1} \times \aA_k \ar[ddd]^{F}
 \\
 \oO(n) \times \Prod_r \bigg( \aA_{1,k-1} \times \oO(m_r)\times \aA_k^{m_r} \bigg) \ar[d]_{\id \times \Prod_r s_k} \ar[r]^-{\id \times \Prod_r(\id \times \tha(m_r))}
 &
  \oO(n)\times (\aA_{1,k-1} \times \aA_k)^n \ar[dd]^{\id \times F^n}
  \\
 \oO(n)\times \Prod_r \bigg(\oO(m_r)\times (\aA_{1,k-1} \times \aA_k)^{m_r}\bigg) \ar[d]_{\id \times \Prod_r (\id \times F^{m_r})}
  && \ultwocell<\omit>{ \qquad \qquad  \delta_k(n)  }\\
  \oO(n)\times \Prod_r \bigg(\oO(m_r)\times \aB^{m_r}\bigg) \ar[r]_-{\id \times \Prod_r \tha(m_r)} 
  & \oO(n) \times \aB^n \ar[r]_-{\tha(n)} \uultwocell<\omit>{ \qquad \qquad \qquad \id \times \Prod_r \delta_k(m_r)  }
  & \aB
}}
\]
is equal to the $2$-cell
\[
\scalebox{0.75}{
\xymatrixrowsep{1.5cm}\xymatrix@C=5em{
 \aA_{1,k-1}\times \oO(n)\times \Prod_r \bigg( \oO(m_r) \times \cA_k^{m_r}\bigg) \ar[d]_{s_k} \ar[r]^-{\id \times \mu} 
 &
  \aA_{1,k-1} \times \oO(m) \times \aA_k^m \ar[dd]^{s_k} \ar[r]^-{\id \times \tha(m)}
 & 
 \aA_{1,k-1} \times  \aA_k \ar[ddd]^{F}
 \\
 \oO(n) \times \Prod_r \bigg( \aA_{1,k-1} \times \oO(m_r)\times \aA_k^{m_r} \bigg) \ar[d]_{\id \times \Prod_r s_k}  
 &
 \\
 \oO(n)\times \Prod_r \bigg(\oO(m_r)\times (\aA_{1,k-1} \times
 \aA_k)^{m_r}\bigg) \ar[d]_{\id \times \Prod_r (\id \times F^{m_r})} \ar[r]^-{\mu}
  &
   \oO(m)\times (\aA_{1,k-1} \times \aA_k)^m \ar[d]^{\id \times F^m}& \ultwocell<\omit>{ \qquad \qquad  \delta_k(m)  }\\
  \oO(n)\times \Prod_r \bigg(\oO(m_r)\times \aB^{m_r}\bigg) \ar[r]_-{\mu} 
  & \oO(m) \times \aB^m \ar[r]_-{\tha(m)} 
  & \aB
}}
\]

\item\label{DeiDejCommute} (Commutation of $\de_i$ and $\de_j$): 
For $i < j$, and omitting the inactive variables $\aA_h$ for $h\neq i$ or $j$ in order to save space,
the $2$-cell
\[
\scalebox{0.75}{
\xymatrixrowsep{1.5cm}
\xymatrix @C=2.6em{
\oO(m)\times \aA_i^m \times \oO(n) \times \aA_j^n \ar[rr]^-{\id \times \theta(n)} \ar[d]_-{s_i} 
& & \oO(m) \times \aA_i^m \times \aA_j \ar[r]^-{\theta(m) \times \id} \ar[d]_{s_i}
& \aA_i \times \aA_j \ar[ddd]^{F}
\\
\oO(m) \times (\aA_i \times \oO(n) \times \aA_j^n)^m \ar[d]_-{\id \times s_j^m} \ar[rr]^-{\id \times (\id \times \theta(n))^m} & & \oO(m) \times (\aA_i \times \aA_j)^m \ar[d]_{\id \times F^m}
\\
\oO(m) \times (\oO(n)\times (\aA_i\times \aA_j)^n)^m \ar[d]_{\mu} \ar[r]_-{\rule{0pt}{2ex} \id\times (\id \times F^n)^m}& \oO(m) \times (\oO(n)\times \aB^n)^m \ar[d]_{\mu} \ar[r]_-{\rule{0pt}{2ex} \id \times \theta(n)^m}
& \oO(m)\times \aB^m \ar[dr]_{\theta(m)} \ulltwocell<\omit>{<0> \id \times  \delta_j(n)^m \qquad \qquad} 
& \ultwocell<\omit>{ \delta_i(m) \qquad }
\\ 
\oO(mn) \times_{\Sigma_{mn}} (\aA_i \times \aA_j)^{mn} \ar[r]_-{\id \times F^{mn}}& \oO(mn) 
\times_{\Sigma_{mn}} \aB^{mn} \ar[rr]_-{\theta(mn)} 
&&  \aB  \\
}}
\]
is equal to the $2$-cell obtained by pasting the $2$-cell
\[
\scalebox{0.75}{
\xymatrix @C=2em{
& \oO(m) \times \aA_i^m \times \oO(n) \times \aA_j^n \ar[dr]^{s_j} \ar[dl]_{s_i} & \\
\oO(m) \times (\aA_i \times \oO(n) \times \aA_j^n)^m  \ar[d]_-{\id \times (s_j)^m} & &  
\oO(n) \times (\oO(m)\times \aA_i^m \times \aA_j)^n \ar[d]^-{\id \times (s_i)^n}\\
\oO(m)\times (\oO(n) \times (\aA_i \times \aA_j)^n)^m \ar[dr]_{\mu} 
 & \utwocell<\omit>{<0> \alpha_{m,n} \qquad  } & 
\oO(n)\times (\oO(m) \times (\aA_i \times \aA_j)^m)^n \ar[dl]^{\mu} \\
&   \oO(mn) \times_{\Sigma_{mn}} (\aA_i \times \aA_j)^{mn} &  \\
} }
\]
\noindent to the left of the pasting diagram
\[
\scalebox{0.75}{
\xymatrixrowsep{1.5cm}\xymatrix @C=2.6em{
\oO(m)\times \aA_i^m \times \oO(n) \times \aA_j^n \ar[rr]^-{\theta(m) \times \id} \ar[d]_-{s_j} &&  \aA_i \times \oO(n) \times \aA_j^n \ar[r]^-{\id \times \theta(n)} \ar[d]_{s_j}& \aA_i \times \aA_j \ar[ddd]^{F}\\
\oO(n) \times (\oO(m) \times \aA_i^m \times \aA_j)^n \ar[d]_-{\id \times s_i^n} \ar[rr]^-{\id \times (\theta(m) \times \id )^n} & & \oO(n) \times (\aA_i \times \aA_j)^n \ar[d]_{\id \times F^n}\\
\oO(n) \times (\oO(m)\times (\aA_i\times \aA_j)^m)^n \ar[d]_{\mu} \ar[r]_-{\rule{0pt}{2ex} {\id\times (\id \times F^m)^n} }& \oO(n) \times (\oO(m)\times \aB^m)^n \ar[d]_{\mu} 
\ar[r]_-{\rule{0pt}{2ex} \id \times \theta(m)^n} & \oO(n)\times \aB^n \ar[dr]_{\theta(n)} 
\ulltwocell<\omit>{<0> \id \times  \delta_i(m)^n \qquad \qquad} & \ultwocell<\omit>{ \delta_j(n) \qquad}\\
\oO(nm) \times_{\Sigma_{nm}} (\aA_i \times \aA_j)^{nm} \ar[r]_-{\id \times F^{nm}}& \oO(nm) \times_{\Sigma_{nm}} \aB^{nm} \ar[rr]_-{\theta(nm)} & & \aB. 
 \\
} }
\]
\end{enumerate}

\medskip

These  axioms require explanation. They encode the idea that whenever the $\de_i$ 
combine to give two transformations with the same source functor and the same target functor, both with target 
category $\aB$, then they are equal. There is an implicit coherence theorem saying that the diagrams we display generate 
all others. In all of our diagrams, the interior subdiagrams unoccupied by a $2$-cell commute either by
the definition of an operad or by a naturality diagram. 

Axioms \eqref{OperUnitAxiom} and \eqref{OperEquivAxiom} give the compatibilities with basepoints and equivariance necessary for these multimorphisms to give rise to multimorphisms of $\bO$-algebras, as defined by Hyland and Power \cite{HP}.
In \eqref{OperEquivAxiom}, we must check that the source and target functors of the two diagrams displayed are equal. The target 
functors agree trivially. The source functors  agree by the Equivariance Axiom for $\theta$, the naturality of $s_i$,
and the fact that $F^n\com \rh = \rh\com F^n$. 
Axiom \eqref{OperIdentAxiom} corresponds to the Operadic Identity Axiom and
requires no explanation.

In \eqref{OperCompAxiom}, we define the map $\mu$ 
to be the map that shuffles the operad variables to the left and applies the structure map of the operad in those variables.
The source and target functors of the two diagrams agree by the compatibility axioms for $\oO$-algebras.

In \eqref{DeiDejCommute}, 
we abuse notation and again write $\mu$ for the effect of passing to orbits from the $\mu$ used above.
Here the target functors of the first and third diagrams agree trivially but their source functors do not; their left 
vertical composites differ.  After pasting the second diagram to the third, the source
functors of the first diagram and the composite agree.
We note that passage to $\Sigma_{mn}$-orbits in the second diagram is essential, as in \autoref{prop:alphaNatTran}; 
without that, the pseudo-commutativity isomorphism $\al_{m,n}$ would not mediate between
its source and target functors.

\begin{rem}
 These axioms imply further compatibilities of the $\delta_i$ with the unit object 0. In particular, it follows that the component of $\de_i(n)$ at an object 
 \[(a_1,\dots,(x;a_{i,1},\dots,a_{i,n}),\dots,a_k)\]
 is $\id_0$ if  either $a_j\in \aA_j$ is $0$ for some $j\neq i$ or all coordinates $a_{i,r}$ of the $i$th object $a_i\in \aA_i^n$ are $0$.  
Moreover, for $1\leq r\leq n$, the $2$-cell
\[
\scalebox{0.85}{
\xymatrix @C=6em{
\aA_{1,i-1} \times \oO(n) \times \aA_i^{n-1} \times \aA_{i+1,k}
\ar[r]^{\id \times (\si_r \times \id) \times \id}
\ar[d]_{\id \times (\id\times \si_r) \times \id}
&
\aA_{1,i-1}\times \oO(n-1) \times \aA_i^{n-1} \times \aA_{i+1,k}
\ar[d]^{\id \times \tha(n-1) \times \id}
\\
\aA_{1,i-1}\times \oO(n) \times \aA_i^n \times \aA_{i+1,n} 
\ar[r]^{\id \times \tha(n) \times \id}
\ar[d]_{s_1}
&
\aA_{1,i-1} \times \aA_i \times \aA_{i+1,k}
\ar[dd]^F
\\
\oO(n) \times ( \aA_{1,i-1} \times \aA_i \times \aA_{i+1,k} )^n 
\ar[d]_{\id \times F^n}
\\
\oO(n) \times \aB^n 
\ar[r]_{\tha(n)}
&
\aB
\uultwocell<\omit>{<0> \qquad\qquad \de_i(n)}
} }
\]
is equal to the $2$-cell
\[
\scalebox{0.75}{
\xymatrix @C=2.5em{
\aA_{1,i-1}\times \oO(n) \times \aA_i^{n-1} \times \aA_{i+1,k} 
\ar[r]^(.475){\raisebox{-1ex}{\rule{0pt}{1ex}} \id \times (\si_r\times \id) \times \id}
\ar[d]_{s_i}
&
\aA_{1,i-1}\times \oO(n-1) \times \aA_i^{n-1} \times \aA_{i+1,k}
\ar[r]^(.6){\raisebox{-1ex}{\rule{0pt}{1ex}} \id \times\tha(n-1) \times \id}
\ar[d]_{s_i}
&
\aA_{1,i-1}\times \aA_i \times \aA_{i+1,k}
\ar[dd]^F
\\
\oO(n) \times (\aA_{1,i-1}\times \aA_i \times \aA_{i+1,k} )^{n-1}
\ar[r]^-{\si_r\times \id}
\ar[d]_{\id\times F^{n-1}}
&
\oO(n-1) \times (\aA_{1,i-1}\times \aA_i \times \aA_{i+1,k} )^{n-1}
\ar[d]_{\id\times F^{n-1}}
\\
\oO(n) \times \aB^{n-1}
\ar[r]_{\si_r\times \id}
& 
\oO(n-1) \times \aB^{n-1}
\ar[r]_{\tha(n-1)}
&
\aB.
\uultwocell<\omit>{<3> \quad\qquad\qquad \de_i(n-1)}
} }
\]
\end{rem}

\subsection{Coherence axioms for $\Mult(\sD)$}\label{cohMultD}

Here we return to the diagram \autoref{delta}, which defines the invertible $\sV_\bpt$-transformations $\de$ in the $k$-ary morphisms of \autoref{MultiD}, and give the necessary coherence conditions.  The condition on $\PI$ in that definition already incorporates conditions on basepoints and identity morphisms. These are the analogues of axioms \eqref{OperUnitAxiom} and \eqref{OperIdentAxiom} of \autoref{cohMultO}. We require the following condition on composition in $\sD$,  which is analogous to the operadic composition axiom \eqref{OperCompAxiom} there.\\

\noindent
(Categorical Composition Axiom)  We write $\tha^k$ for the left vertical composite
\[\xymatrix@1{ \bigwedge\limits_i \sD(\bm_i,\bn_i)\sma \bigwedge\limits_i \aX_i(\bm_i) \ar[r]^-{t}_-{\cong} & \bigwedge\limits_i\big(\sD(\bm_i,\bn_i)\sma \aX_i(\bm_i)\big)\ar[r]^-{\bigwedge\limits_i\tha}
& \bigwedge\limits_i \aX(\bn_i) \\}\]
in \autoref{delta}. We write $C$ for the 
composition in $\sD^{\esma k}$
\[ \xymatrix@1{ \bigwedge\limits_i \sD(\bn_i,\bp_i) \sma \bigwedge\limits_i \sD(\bm_i,\bn_i) \ar[r]^-{t}_-{\cong} & 
\bigwedge\limits_i \big{(}\sD(\bn_i,\bp_i) \sma \sD(\bm_i,\bn_i)\big{)} \ar[r]^-{\bigwedge\limits_i\com} & \bigwedge\limits_i \sD(\bm_i,\bp_i).\\} \]
The right vertical composite 
\[ \xymatrix@1{\bigwedge\limits_i \sD(\bm_i,\bn_i)\sma \aY(\bm) \ar[r]^-{\smaD^k\sma \id} 
& \sD(\bm, \bn)\sma \aY({\bm}) \ar[r]^-{\tha} & \aY({\bn})\\} \]
in \autoref{delta} is the action $\THA_k$ 
of $\sD^{\esma k}$ on $\aY \com \smaD^{k}$.

With these notations, the following pasting diagrams are required to be equal. 
\[
\scalebox{0.6}{
\xymatrixcolsep{-1pc}\xymatrixrowsep{4pc}\xymatrix{
& \bigwedge\limits_i \sD(\bn_i,\bp_i)\esma \bigwedge\limits_i \sD(\bm_i,\bn_i) \esma \bigwedge\limits_i \aX_i(\bm_i) \ar[rr]^-{\id \esma \id \esma F} \ar[dr]_-{\id \esma \tha^k} \ar[dl]_-{C \esma \id} 
& \drtwocell<\omit>{<0>  \quad { \id \esma \de}} & \bigwedge\limits_i \sD(\bn_i,\bp_i)\esma \bigwedge\limits_i \sD(\bm_i,\bn_i) \esma \aY(\bm) \ar[dr]^-{\id \esma \THA_k}\\
 \bigwedge\limits_i \sD(\bm_i,\bp_i) \esma \bigwedge\limits_i \aX_i(\bm_i) \ar[dr]_-{\tha^k} &  & \bigwedge\limits_i \sD(\bn_i,\bp_i) \esma \bigwedge\limits_i \aX_i(\bn_i) \ar[rr]_-{\id \esma F} \ar[dl]_-{\tha^k}  \drtwocell<\omit>{<0>  \quad { \de}}
 &  & \bigwedge\limits_i \sD(\bn_i,\bp_i) \esma \aY(\bn) \ar[dl]^{\THA_k}\\
& \bigwedge\limits_i \aX_i(\bp_i) \ar[rr]_-{F} && \aY(\bp)
}}
\]

\[
\scalebox{0.6}{
\xymatrixcolsep{-1pc}\xymatrixrowsep{4pc}\xymatrix{
& \bigwedge\limits_i \sD(\bn_i,\bp_i)\esma \bigwedge\limits_i \sD(\bm_i,\bn_i) \esma \bigwedge\limits_i \aX_i(\bm_i) \ar[rr]^-{\id \esma \id \esma F} \ar[dl]_-{C \esma \id} &  & \bigwedge\limits_i \sD(\bn_i,\bp_i)\esma \bigwedge\limits_i \sD(\bm_i,\bn_i) \esma \aY(\bm) \ar[dr]^-{\id \esma \THA_k} \ar[dl]_-{C\esma \id}  \ddtwocell<\omit>{<0>  \quad }\\
 \bigwedge\limits_i \sD(\bm_i,\bp_i) \esma \bigwedge\limits_i \aX_i(\bm_i) \ar[dr]_-{\tha^k} \ar[rr]^-{\id \esma F} &   \drtwocell<\omit>{<0>  \quad { \de}} & \bigwedge\limits_i \sD(\bm_i,\bp_i) \esma \aY(\bm) \ar[dr]_-{\THA_k} & & \bigwedge\limits_i \sD(\bn_i,\bp_i) \esma \aY(\bn) \ar[dl]^{\THA_k}\\
& \bigwedge\limits_i \aX_i(\bp_i) \ar[rr]_-{F} && \aY(\bp)
}}
\]

The unmarked regions in these diagrams commute. For instance, the rhombus in the first diagram commutes because $\aX_1,\dots,\aX_k$ are strict $\sD$-algebras, and hence $\aX_1 \overline{\esma} \dots \overline{\esma} \aX_k$ is a strict $\sD^{\esma k}$-algebra as well. The unlabeled $\sV_\bpt$-transformation in the rhombus in the second diagram is the constraint for the $\sD^{\esma k}$-pseudoalgebra $\aY \circ \smaD^k$. More precisely, it is given by the following whiskering of $\vartheta$, which denotes an iterated version of the coherence $\sV_\bpt$-pseudotransformation from \autoref{beta0}.

\[
\scalebox{0.7}{
 \xymatrixcolsep{3.5pc}\xymatrix{
 \bigwedge\limits_i \sD(\bn_i,\bp_i)\esma \bigwedge\limits_i \sD(\bm_i,\bn_i) \esma \aY(\bm) \ar[r]^-{\id \esma \smaD^k \esma \id }\ar[d]_{t}   \ddrtwocell<\omit>{<0>  \quad \vartheta}
 & \bigwedge\limits_i \sD(\bn_i,\bp_i)\esma \sD(\bm,\bn) \esma \aY(\bm) \ar[r]^-{\id \esma \tha} \ar[d]^{\smaD^k \esma \id}
 & \bigwedge\limits_i \sD(\bn_i,\bp_i)\esma \aY(\bn) \ar[d]^{\smaD^k \esma \id}\\
 \bigwedge\limits_i \big(\sD(\bn_i,\bp_i)\esma  \sD(\bm_i,\bn_i)\big) \esma \aY(\bm) \ar[d]_-{\bigwedge\limits_i \com \esma \id} 
 &  \sD(\bn,\bp)\esma  \sD(\bm,\bn) \esma \aY(\bm) \ar[r]^-{\id \esma \tha} \ar[d]^{\com \esma \id}
 & \sD(\bn,\bp) \esma \aY(\bn) \ar[d]^{\tha}\\
 \bigwedge\limits_i \sD(\bm_i,\bp_i)\esma \aY(\bm) \ar[r]_-{\smaD^k \esma \id} 
 & \sD(\bm,\bp)\esma \aY(\bm) \ar[r]_-{\tha} 
 & \aY(\bp).
}}
\]
Note that the bottom right rectangle commutes because $\aY$ is a strict $\sD$-algebra.

\section{The pseudo-commutativity of $\sD(\oO)$}\label{Rpf}  

We prove \autoref{multithm} here.  Thus let $\oO$ be a pseudo-commutative 
operad in $\VCat$
and $\sD=\sD(\oO)$ the associated category of operators, as in \autoref{DODefn}.
  We must construct a  $\sV_\bpt$-pseudofunctor $\squiggly{\smaD \colon \sD \esma \sD }{\sD}$ and prove that it gives a pseudo-commutative structure. For the sake of clarity, we work with $\times$ rather than $\sma$ in this section; the statements about $\PI$ build in basepoint conditions that imply that all the constructions descend to the smash product.  We break the proof into several parts. Recall that for a morphism $\ph\colon \bm \rtarr \bn$ of $\sF$ and $1\leq j\leq n$,  we  write $\ph_j = |\ph^{-1}(j)|$.  
  
When restricted to $\PI$, $\smaD$ must be $\sma$.  Thus, on objects, $\bm\smaD \bp=\bm\bp$. At the level of Hom categories, the map 
\[ \smaD \colon \sD(\bm,\bn) \times \sD(\bp,\bq) \rtarr \sD(\bm\bp,\bn\bq)\]
sends the summand in the source labeled by $\phi\colon \bm \rtarr \bn$ and $\psi \colon \bp \rtarr \bq$ to the one labeled by $\phi \sma \psi \colon \bm\bp \rtarr \bn\bq$ in the target. Therein, the $\sV$-functor
\[  \Prod_{1\leq j\leq n} \oO(\ph_j) \times
  \Prod_{1\leq k\leq q} \oO(\ps_k) \rtarr \Prod_{1\leq \ell\leq nq} \oO((\ph\sma \ps)_{\ell}),\]
is such that its projection onto the $\ell$th factor is given by first projecting onto 
$\oO(\phi_j)\times\oO(\phi_k)$, where $\ell$ maps to the pair $(j,k)$ under the lexicographic ordering of $\bn\sma \bq$, and then applying the pairing $\opair$ of $\oO$. 
The definition makes sense since 
\[(\ph\sma\ps)_{\ell} = |(\ph\sma \ps)^{-1}(\ell)| = |\ph^{-1}(j)||\ps^{-1}(k)| = \ph_j\ps_k. \] It is immediate from the definition that $\smaD$ restricts to $\sma$ on $\PI$ (along $\io$) and projects to $\sma$ on $\sF$ (via $\xi$), as required.

To complete the construction of the $\sV_\bpt$-pseudofunctor $\smaD$, we must prove the following result, which is the
heart of the proof that $\sD$ is pseudo-commutative.

\begin{prop}\label{isapseudo} The following diagram of $\sV$-functors relating $\smaD$ to composition 
commutes up to an invertible $\sV$-transformation $\vartheta$.  

\[\xymatrix{
\sD(\mb{n},\mb{p})\times \sD(\mb{r},\mb{s})\times \sD(\mb{m}, \mb{n})\times \sD(\mb{q}, \mb{r}) 
\ar[rr]^-{\smaD\times \smaD} \ar[d]_{\id\times t\times \id}^{\iso}   \ddrrtwocell<\omit>{<0>\     \vartheta} & &
\sD(\mb{nr}, \mb{ps})\times \sD(\mb{mq}, \mb{nr})\ar[dd]^\circ\\
\sD(\mb{n},\mb{p})\times \sD(\mb{m}, \mb{n})\times \sD(\mb{r},\mb{s})\times\sD(\mb{q}, \mb{r}) 
\ar[d]_{\circ \times \circ}  & & \\
\sD(\mb{m}, \mb{p}) \times\sD(\mb{q}, \mb{s}) \ar[rr]_-{\smaD}  & & \sD(\mb{mq}, \mb{ps}) 
\\} \]
The collection of such $\sV$-transformations descends to the smash product and makes $\squiggly{\smaD \colon \sD \esma \sD}{\sD}$ into a $\sV_\bpt$-pseudofunctor.
\end{prop}

\begin{proof}
The essential combinatorial claim is that the pseudo-commutativity isomorphisms
$\al$ of $\oO$ from \autoref{pseudocom} assemble to give the required invertible $\sV$-transformations $\vartheta$. 
This is not obvious since the $\al$  give maps that are not obviously relevant to the diagram.
The strategy is to express the results of the source and target of the diagram
in such a way that the invertible $\sV$-transformation
between them becomes obvious. In the following equations, we will use the associativity and equivariance formulas from the definition of an operad to massage the two composites into comparable form.

Since $\PI$ and $\sF$ are permutative categories regarded as $\sV$-$2$-categories,
the diagram clearly commutes when $\sD = \PI$ or $\sD =\sF$.  That is, fixing morphisms
\[ \ps\colon \mb{n}\rtarr \mb{p}, \ \ \nu\colon \mb{r}\rtarr \mb{s}, \ \ \ph\colon \mb{m}\rtarr \mb{n}, \ \text{and}\   \mu\colon \mb{q}\rtarr \mb{r} 
 \]
in $\sF$, we have 
\begin{equation}\label{expected}
(\ps\sma \nu)\circ (\ph\sma\mu) = (\ps\com \ph) \sma (\nu\com \mu). 
\end{equation} 

Thus, for the summand labeled by our fixed morphisms $\ps$, $\nu$,
$\ph$, and $\mu$ in $\sF$, the clockwise and counterclockwise directions land in the same summand of the target. It follows that it suffices to restrict the diagram to these summands.
Let
\[ 1\leq j\leq n, \ \ 1\leq k\leq p, \ \ 1\leq h\leq r, \ \text{and} \ 1\leq i \leq s.\]
Looking at the definitions of the composition $\com$ of $\sD$ in \autoref{Dcomposition}
and of $\smaD$, we see that to chase the diagram
starting at the top left with 
\[ \prod_k \oO(\ps_k) \times \prod_i \oO(\nu_i) \times  \prod_j \oO(\ph_j) \times \prod_h\oO(\mu_h), \]
it suffices to consider its projection onto each term of the product in the target. Thus, we fix $k$ and $i$, 
project to 
\[ \oO(\ps_k) \times \oO(\nu_i)\times  \prod_{j \in \psi^{-1}(k) } \oO(\ph_j) \times
\prod_{h \in \nu^{-1}(i)}  \oO(\mu_h), \]
and then chase. Going around both ways, we land in the term $\oO\Big(\big((\ps \sma \nu)\circ(\ph \sma \mu)\big)_{\ell}\Big)$ of the target, where $(k,i)\in \bp\sma\bs$ corresponds to $\ell \in \bp\bs$ under lexicographical ordering.

Writing in terms of elements to better 
apply formulas rather than chase large diagrams, let
\[ c\in \oO(\ps_k),\ \ a\in \oO(\nu_i), \ \ d_j\in \oO(\ph_j),  \ \text{and}\ b_h\in \oO(\mu_h),\] 
and recall the permutations $\rho$ defined in \autoref{Dcomposition}.
Going clockwise, the tuple $(c, a, \prod_j d_j,\prod_h b_h)$ gets sent to
\begin{equation}\label{expression1}
 \textstyle \ga\Big(c\opair a; \prod\limits_{(\ps\sma \nu) (j,h)=(k,i)} d_j\opair b_h\Big)\,
 \rh_{(k,i)}(\ps\sma\nu,\ph\sma \mu),
\end{equation}
and going counterclockwise, it gets sent to
\begin{equation}\label{expression2}
{\textstyle
\Big(  \ga(c; \prod\limits_{\ps(j)=k} d_j)\rh_k(\ps,\ph)\Big) \opair 
\Big(\ga(a; \prod\limits_{\nu(h)=i} b_h)\rh_i(\nu,\mu) \Big).
}
\end{equation}

In what follows, we use the notation $\diag{x}{\ell}$ to denote the $\ell$-tuple $(x,\dots,x)$. Using the definition of $\opair$, the associativity from the definition of an operad twice and abbreviating 
$\rh_{(k,i)}(\ps\sma\nu,\ph\sma \mu) = \rh_{(k,i)}$, we identify the expression \autoref{expression1} as follows:
\begin{align}
& \hspace{-1em}
\textstyle \ga\Big(c\opair a; \prod\limits_{(\ps\sma \nu) (j,h)=(k,i)} d_j\opair b_h\Big)\,
 \rh_{(k,i)} \nonumber \\
 & \textstyle=\ga\Big(\ga(c; \diag{a}{\ps_k}); \prod\limits_{(\ps\sma \nu) (j,h)=(k,i)} \ga(d_j; \diag{b_h}{\ph_j})\Big)\,
 \rh_{(k,i)}  \nonumber
 \\
  &=  \textstyle \ga\Big(c; \ \prod\limits_{\ps(j)=k}  \ga\big(a; \prod\limits_{\nu(h)=i}  \ga(d_j; \diag{b_h}{\ph_j})\big) \Big)\rh_{(k,i)}  \nonumber \\
  &= \textstyle  \ga \Big( c;\  \prod\limits_{\ps(j)=k} \ga\big(\ga(a; \diag{d_j}{\nu_i}); \prod\limits_{\nu(h)=i}\diag{b_h}{\ph_j}\big) 
 \Big)\rh_{(k,i)} \nonumber \\
 &= \textstyle \ga \Big( c;\  \prod\limits_{\ps(j)=k} \ga\big(a \opair d_j; \prod\limits_{\nu(h)=i}\diag{b_h}{\ph_j}\big) 
 \Big)\rh_{(k,i)}.  \label{newexpression1}
\end{align} 

We abbreviate
$\rh_k(\ps,\ph) = \rh_k$ and $\rh_i(\nu,\mu) = \rh_i$. 
Using \autoref{rem:opair-equivariant} and the associativity axiom of the operad, we identify the expression \autoref{expression2} as follows:
\begin{align}
& \hspace{-1em}
\textstyle
\Big(  \ga(c; \prod\limits_{\ps(j)=k} d_j)\rh_k\Big) \opair 
\Big(\ga(a; \prod\limits_{\nu(h)=i} b_h)\rh_i \Big)
\nonumber \\
&= \textstyle
\Big(  \ga(c; \prod\limits_{\ps(j)=k} d_j)\opair 
\ga(a; \prod\limits_{\nu(h)=i} b_h) \Big)(\rho_k\spair\rho_i)
\nonumber \\
& \textstyle = \ga\Big(  \ga(c; \prod\limits_{\ps(j)=k} d_j); \, 
\diag{\ga(a; \prod\limits_{\nu(h)=i} b_h) }{(\ps\com\ph)_k} \Big)(\rho_k\spair\rho_i) \nonumber \\
& \textstyle =
\ga\Big( c; \ \prod\limits_{\ps(j)=k} \ga(d_j; \diag{\ga(a; \prod\limits_{\nu(h)=i} b_h)}{\ph_j}) \Big) 
(\rho_k\spair\rho_i) \nonumber \\
& \textstyle = \ga \Big( c;\ \prod\limits_{\ps(j)=k} \ga \big(\ga(d_j; \diag{a}{\ph_j}); \diag{\prod\limits_{\nu(h)=i} b_h}{\ph_j}
 \big)\Big) (\rho_k\spair\rho_i) \nonumber \\
 & \textstyle = \ga \Big( c;\ \prod\limits_{\ps(j)=k} \ga \big(d_j\opair a; \diag{\prod\limits_{\nu(h)=i} b_h}{\ph_j}
 \big)\Big) (\rho_k\spair\rho_i) \label{newexpression2}
\end{align}

The similarity between the reinterpretations \autoref{newexpression1} and \autoref{newexpression2} of \autoref{expression1} and \autoref{expression2} is clear.
We use the $\al$'s from \autoref{pseudocom} to build an invertible $\sV$-transformation from the expression \autoref{newexpression1} 
to the expression \autoref{newexpression2}. This will specify $\vartheta$.  For legibility, we omit 
the indices on the $\al$'s. From the pseudo-commutativity of $\oO$, for fixed $i$ and $j$ we have an isomorphism
\[ \xymatrix{
\al\colon (a\opair  d_j)\tau_{\ph_j,\nu_i} \rtarr d_j\opair a,
}\]
where $\tau_{\ph_j,\nu_i}$ is as in \autoref{tau}. It induces the second isomorphism in the  composite
\[ \xymatrix{
\ga\Big(a \opair d_j;\, \prod\limits_{\nu(h) = i}\diag{b_h}{\ph_j} \Big) D_{\ph_j,\mu_*} \ar@{=}[d]
\\
\ga\Big((a \opair d_j)\tau_{\ph_j,\nu_i}; \,  {\diag{\prod\limits_{\nu(h)=i} b_h}{ \ph_j} }\Big) 
\ar[d]^-{\ga(\al;\diag{\id}{\ph_j})}
\\
\ga\Big(d_j \opair a; \,  {\diag{\prod\limits_{\nu(h)=i} b_h}{\ph_j}} \Big). 
\\}
\] 
The first equality is given by the equivariance formula for $\ga$; the permutation 
$D_{\ph_j,\mu_*}$ is as in \autoref{dist}, with $\ast$ running through the set $\nu^{-1}(i)$. This is precisely the permutation of 
$\ph_j\cdot\sum_{\nu(h)=i} \mu_{h} = \ph_j\cdot (\nu\circ \mu)_i$ 
elements which permutes according to $\tau_{\phi_j,\nu_i}$ the $\ph_j\cdot \nu_i$ blocks of lengths given by the tuple $\diag{\prod_{\nu(h)=i} \mu_h}{\phi_j}$. 

By applying $\ga(c;-)$ to the product over $j\in \psi^{-1}(k)$ of the composite above, we obtain the second isomorphism in the composite
\begin{equation}
\begin{gathered}
 \xymatrix{
\ga\Big(c;\, \prod\limits_{\ps(j)=k}\ga\big(a \opair d_j; \, \prod\limits_{\nu(h) = i} \diag{ b_h}{\ph_j} \big) \Big)\Big( \bigoplus\limits_{\ps(j) = k} D_{\ph_j,\mu_*}\Big) \ar@{=}[d]
\\
\ga\Big(c;\prod\limits_{\ps(j)=k} \ga\big(a \opair d_j; 
\, \prod\limits_{\nu(h) = i} \diag{ b_h}{\ph_j} \big)D_{\ph_j,\mu_*} \Big) \ar[d]^{\ga\left(\id;\Prod \ga(\al;\diag{\id}{\phi_j})\right)}
\\
\ga \Big(c; \prod\limits_{\ps(j)=k} \ga \big(d_j \opair a; \, { \diag{\prod\limits_{\nu(h) = i} b_h}{\ph_j} }
\big)\Big).
\\}
\end{gathered}
\label{SecondIso}
\end{equation}
The equality is again given by the equivariance of $\ga$.\\

A straightforward  computation, which we omit, gives that 
\[\rh_{(k,i)} =  \Big(\bigoplus_{\ps(j) = k} D_{\ph_j,\mu_*} \Big)\cdot (\rho_k\spair\rho_i).\] 
Therefore, multiplying the  isomorphism \autoref{SecondIso} by $\rho_k\spair \rho_i$ yields the 
desired isomorphism
\[ \xymatrix{
\ga\Big(c;\, \prod\limits_{\ps(j)=k}\ga\big(a \opair d_j;\prod\limits_{\nu(h) = i}\diag{b_h}{\ph_j}
\big) \Big)\rh_{(k,i)}\ar[d]^{\ga\left(\id;\Prod \ga(\al;\diag{\id}{\phi_j})\right) ( \rho_k \spair \rho_i) }
\\
\ga \Big(c;\ \prod\limits_{\ps(j)=k} \ga \big(d_j \opair a; \,   {\diag{\prod\limits_{\nu(h)=i} b_h}{\ph_j} }
\big)\Big)(\rho_k\spair\rho_i)}
\]
from \autoref{expression1} to \autoref{expression2}. Interpreting the proof diagrammatically shows that we have an invertible $\sV$-transformation $\vartheta$ as needed. Compatibility of $\vartheta$ with identity morphisms and the fact that it descends to the smash product follow from \autoref{be_pi} below. Compatibility with composition in $\sD \esma \sD$ is tedious to check, but boils down to repeated use of Axiom \eqref{pseudocommhex1} of \autoref{pseudocom}. This completes the proof that $\squiggly{\smaD \colon \sD \esma \sD}{\sD}$ is a $\sV_\bpt$-pseudofunctor.
\end{proof}

Condition  \eqref{PPDDcond} of \autoref{pairprod} holds as a result of the following more general lemma.

\begin{lem}\label{beta_PI}
After restricting the domain of the functors in the diagram of \autoref{isapseudo} to 
\[ \sD(\mb{n},\mb{p})\times \sD(\mb{r},\mb{s})\times \Pi(\mb{m}, \mb{n})\times \sD(\mb{q}, \mb{r})  \rtarr \sD(\mb{mq}, \mb{ns}) \]
or
\[ \sD(\mb{n},\mb{p})\times \Pi(\mb{r},\mb{s})\times \sD(\mb{m}, \mb{n})\times \sD(\mb{q}, \mb{r})  \rtarr \sD(\mb{mq}, \mb{ns}), \]
the transformation $\vartheta$ is the identity.
\end{lem}

\begin{proof}
We use the notation of the construction of $\vartheta$ in \autoref{isapseudo}. The first restriction is the case when the
$d_j$ 
are all $\ast$ or $\oid$, and the second restriction is the case when $e$ is  $\ast$ or $\oid$.
The key is that if either $e$ or $d_j$ is $\ast$, then 
\[ \xymatrix{
\al\colon (e\opair d_j)\tau_{\ph_j,\nu_i} \rtarr d_j \opair e
}\] 
is the identity map of $\ast$. Similarly, by \autoref{pseudocom} \eqref{pseudocomid}, if  $e=\oid$, then $\al$ is the identity map 
of $d_j$, whereas if $d_j=\oid$, then $\al$ is the identity map of $e$.
\end{proof}

 Note that this result in particular implies that as a $\sV_\bpt$-pseudofunctor, $\smaD$ restricts to $\sma$ on $\Pi$.

 \begin{lem}
 The $\sV_\bpt$-pseudofunctor $\smaD$ is strictly associative in the sense that the following diagram of $\sV_\bpt$-pseudofunctors commutes.
 \[
 \xymatrix{
 \sD^{\esma 3} \ar@{~>}[r]^-{\id \esma \smaD} \ar@{~>}[d]_-{\smaD\esma\id} & \sD^{\esma 2} \ar@{~>}[d]^-{\smaD}\\
 \sD^{\esma 2} \ar@{~>}[r]_-{\smaD} & \sD  }
 \]
 \end{lem}
 
 \begin{proof}
 Since this is an equality of $\sV_\bpt$-pseudofunctors, we need to check equality of the level of assignments on objects, $\sV_\bpt$-functors on morphisms, and pseudofunctoriality constraints. The equality of assignments on objects follows from the strict associativity of $\sma$ in $\Pi$. The equality at the level of morphisms follows from the strict associativity of the pairing of $\oO$ (\autoref{intrinsicmon}).  {For each composite, the pseudofunctoriality constraint is given by a pasting of two instances of the $\sV_\bpt$-transformation $\vartheta$ of \autoref{isapseudo}. As such they are each constructed using instances of $\al$ and the operadic structure. After some standard simplifications, the equality of these constraints reduces to axiom \eqref{pseudocommhex1} of \autoref{pseudocom}. }
 \end{proof}
 
 \begin{lem}\label{be_pi}
The $\sV_\bpt$-pseudofunctor $\squiggly{\smaD\colon \sD\esma \sD}{\sD}$ has a symmetry $\sV_\bpt$-pseudo\-transfor\-mation
$\ta$ such that the strict monoidal $\sV_\bpt$-$2$-functors $\io\colon \PI\rtarr \sD$ and 
$\xi\colon \sD\rtarr \sF$ preserve the symmetry.
\end{lem}

\begin{proof}
The $\sV_\ast$-pseudofunctors $\smaD$ and $\squiggly{ \smaD\circ t\colon \sD\esma \sD}{\sD}$
we are comparing have the same object functions. Given objects $\mb{m}$ and $\mb{p}$, the 
$1$-cell component of $\tau$ is given by the permutation \mbox{$\ta_{m,p} \colon \mb{mp} \rtarr \mb{pm}$} of \autoref{tau},
thought of as a morphism in $\PI\subset \sD$. 
We need invertible $\sV_\bpt$-transformations 
\[ \xymatrix{
\sD(\bm,\bn) \sma \sD(\bp,\bq)  \ar[r]^-{\opair\circ t} \ar[d]_-{\opair} \drtwocell<\omit>{<0>  \hat{\ta}} & \sD(\bp\bm, \bq\bn) \ar[d]^{(\tau_{m,p})^*}\\
\sD(\bm \bp, \bn\bq)  \ar[r]_{(\tau_{n,q})_*} &   \sD( \bm \bp, \bq \bn).
}\]
As in the previous proofs, we can restrict to the components of $\sD(\mb m,\mb n)$ and
$\sD(\mb p,\mb q)$, which are indexed on morphisms $\ph\colon \mb m\rtarr \mb n$ and $\ps\colon \mb p \rtarr \mb q$
of $\sF$. 
Note that both maps send the component of $(\ph,\ps)$ in the source to that of
\[(\ps\sma\ph) \circ \ta_{m,p}=\ta_{n,q} \circ (\ph \sma \ps)\]
in the target (see \autoref{ttt}). We thus fix such $\phi$ and $\psi$ and start with
$\prod_{j} \oO(\ph_j) \times \prod_k \oO(\ph_k)$,
where $1\leq j\leq n$ and $1\leq k\leq q$.  Again for simplicity we work with elements
$c_j\in \oO(\ph_j)$ and $d_k\in \oO(\ps_k)$.

Considering permutations as morphisms of $\PI\subset \sD$ and using the definition of composition in
$\sD$, we find that 
the clockwise composite sends $\big((c_1,\dots,c_n),(d_1,\dots,d_q)\big)$ to
\[
 {\textstyle\prod\limits_{j,k}} (c_j \opair d_k) \ta_{\ps_j,\ph_k},
\] 
and the counterclockwise composite sends it to
\[
 {\textstyle\prod\limits_{j,k}} (d_k \opair c_j),
\]
with both products ordered in reverse lexicographical order.

Applying a product of maps $\al$ gives the invertible $\sV_\bpt$-transformation $\hat{\ta}$ indicated in the diagram. Note that, similar to \autoref{beta_PI}, we have that the 2-cell $\hat{\ta}$ is the identity when either copy of $\sD$ is restricted to $\Pi$. 
We leave to the reader the verification of compatibility with composition, and axioms \eqref{ta_sym_axiom} and \eqref{ta_hex_axiom} of \autoref{permv2cat}.
\end{proof}

\begingroup%
\setlength{\parskip}{\storeparskip}
\bibliographystyle{plain}
\bibliography{references}

\end{document}